\documentclass[]{elsarticle}
\usepackage{amsmath}
\usepackage{mathtools}

\usepackage{xcolor}

\usepackage{graphicx}
\usepackage{subcaption}
\usepackage{siunitx}

\usepackage{algorithm}
\usepackage{algpseudocode}

\usepackage[utf8]{inputenc}
\usepackage[T1]{fontenc}
\usepackage{lmodern} 
\usepackage{microtype}

\usepackage{keytheorems}
\newkeytheorem{theorem, lemma, proposition, corollary}[parent = section]
\newkeytheorem{example, definition, assumption, remark}[style = definition, sibling = theorem]

\usepackage{zref-clever}
\zcLanguageSetup{english}{cap}
\zcRefTypeSetup{assumption}{
    name-sg = {assumption},
    Name-sg = {Assumption},
    name-pl = {assumptions},
    Name-pl = {Assumptions}
}

\usepackage[normalem]{ulem}
\usepackage{cancel}

\usepackage[notext]{stix2} 

\numberwithin{equation}{section}

\newcommand{\R}{\mathbb{R}}
\newcommand{\C}{\mathbb{C}}

\AtBeginDocument{

}

\title{Momentum Space Algorithm for Electronic Structure of Double-Incommensurate Trilayer Graphene}
\author[1]{Kenneth Beard}
\author[2]{Daniel Massatt}
\date{August 2025}

\begin{document}

\begin{abstract}

    Numerical algorithms for computing electronic structure of incommensurate 2D materials using ab initio models is critical for predicting material properties and guiding experiment. For bilayers, momentum space and continuum models have been introduced to approximate observables of ab initio tight-binding models using a momenta description despite the lack of periodicity in the tight-binding model required for Bloch theory. A similar structure has been introduced for double-incommensurate trilayers using a continuum model, where the three lattices are all mutually incommensurate. However, this description leads to a four-dimensional lattice space, and numerical convergence of the density of states was observed to have poor convergence.
    
    In this work, we introduce a momentum space framework for double incommensurate trilayer graphene, and introduce an efficient truncation scheme of the four-dimensional lattice to drastically improve convergence of the density of states and momentum local density of states (a parallel object to classical band structure). We implement this algorithm on an ab initio model of twisted trilayer graphene and validate convergence estimates. We further verify numerically that the momentum space algorithm, inherently higher order than the continuum model as it is an exact transformation of the tight-binding model, captures altered band behavior near the flat bands at magic angles.
\end{abstract}

\maketitle

\section{Introduction}

Incommensurate 2D materials have become a hotbed of research after the discovery of an unconventional superconducting phase in twisted bilayer graphene (TBG) in 2018 \cite{Cao2018} along with a correlated insulating phase, followed shortly by the discovery of the Fractional Quantum Hall Effect in TBG \cite{Xie2021}. Similar findings have been found for twisted transition metal dichalcogenides (TMDs) \cite{SuperTMD2025}, and in trilayer graphene systems \cite{TTG_super2022}. Quasi-band structure of single-particle models has proven to be a powerful predictor  of these exotic many-body effects via the presence of flatbands at the Fermi level \cite{Bistritzer2011, Cao2018}. In the case of twisted trilayer graphene (TTG) with three different angles, one obtains a ``double'' incommensuration. The computation of electronic structure using a continuum model produces a high dimensional scattering structure \cite{Zoe2020}, which becomes computationally restrictive. In this work, we use the momentum space framework introduced in \cite{massatt_incommensurate_2018, massatt_electronic_2023} to transform tight-binding models for twisted trilayer graphene into momentum space, and construct an efficient algorithm for evaluating the density of states (DoS) and the momentum local density of states (momentum LDoS), an object parallel to the bilayer's quasi-band structure. We note that while we focus on TTG, this algorithm can be used for any material with energy-momenta confinement such as Dirac cones or parabolic bands with sufficiently weak interlayer coupling, for example TMDs at the parabolic bands. We rigorously prove error estimates in all approximation parameters, and numerically verify convergence by testing our algorithm on the ab initio tight-binding model for incommensurate graphene layered materials introduced in \cite{fang_electronic_2016}. We find converged images of the momentum LDoS for small smearing values, clearly illustrating flatband like behavior in magic angle pairs for an ab initio model of TTG. We also numerically observe by employing our truncation scheme on both an approximate continuum model and our momentum space model that higher order corrections inherent in the momentum space model make qualitatively important alterations to the momentum LDoS near the flatband regime (see \zcref{sec:numerics}).

A careful understanding of the single particle model is critical for studying many-body effects. Indeed, many-body moiré effects are modeled by restricting first to a reduced single-particle basis near the flatbands. In \cite{Bernevig2022}, a many-body Heavy Fermion model was developed using the Bistrtizer-MacDonald model (BM model) for TBG as a reduced basis to explain correlated insulation, while in \cite{LinLin2025} a mathematical proof of a correlated phase in twisted bilayer graphene is given for a model using the flatbands of the BM model as a basis for a many-body model. Theoretical investigation of twisted TMDs is reviewed in \cite{TMD2025}, and an explanation for the fractional quantum hall effect for TBG is developed in \cite{Ledwith2020}. 

There has been much development on rigorous derivation of single-particle momentum space models for incommensurate bilayer systems. Derivation of a continuum model for TBG directly from Schr\"{o}dinger using semiclassical calculus was done in \cite{cances2023}.  It was shown in \cite{Quan2025} that the popular continuum models for TBG such as the BM model can be derived via Taylor expansions of the momentum space transformed tight-binding model of TBG, while in \cite{watson2023} the BM model is justified as an approximation of tight-binding using wavepacket analysis. 
for double-incommensurate trilayer graphene, mechanical relaxation continuum models are developed and computed in \cite{Relax2020}, and electronic structure using real space methods incorporating mechanical relaxation is developed in \cite{Meng2023}. In \cite{Mora2019}, approximate symmetries are used to introduce an approximate continuum model with significant reduction in computational complexity, which introduces an operator admitting band structure. In \cite{Zoe2020}, the authors used a continuum model approximation of tight-binding that kept the complexities from broken symmetries, and used a computational truncation exhibiting four-dimensional degree of freedom scaling producing quasi-band structure of the tight binding model. 
See also the review of double-incommensurate trilayer 2D materials in \cite{ReviewTTG2025}. In \cite{wilson2025}, fractal state behavior was numerically predicted using the same continuum model of TTG, indicating singular spectrum effects.

In this work, we directly transform the tight-binding model into momentum space, which we prove produces DoS and momentum LDoS identical to the tight binding model, hence avoiding accruing extra error from using a continuum model approximation. Next, we exploit the Dirac cone structure in the monolayer graphene to exploit an energy-momenta confinement. In essence, momenta where the monolayer energy is far from the energy of the Dirac point will contribute weakly to the DoS and momentum LDoS at energies near the Dirac point, and can be neglected. We will focus exclusively on the case that the twist angles between the graphene layers are small, which we show results in an efficient truncation of degrees of freedom to wavenumbers near the Dirac point, effectively reducing the four-dimensional computational scaling to a two dimensonal scaling. We use Combes-Thomas estimates to estimate the remaining lattice truncation, and use the Kernel Polynomial Method (KPM) \cite{weise_kernel_2006} to approximate the momentum LDoS as the used matrices are too large to efficiently use diagonalization. We also implement the continuum model in \cite{Zoe2020} using our truncation, and compare momentum LDoS results with momentum space to illustrate continuum approximation effects as compared to the tight-binding.

The rest of the paper is organized as follows. In \zcref{sec:model}, we introduce the real space model and the density of states observables. In \zcref{sec:momentum}, we introduce the momentum space transformation, reciprocal space, and the momentum LDoS, a parallel object to band structure for periodic crystals. In \zcref{sec:algorithm}, we develop our new algorithm and give rigorous convergence rates, which we numerically verify and show for TTG using an ab initio model \cite{fang_electronic_2016} including the flatband behavior at magic angle pairs and angle pairs with an irrational ratio in \zcref{sec:numerics}. We conclude in \zcref{sec:conclusion}.

\section*{Acknowledgments}

DM was supported by AFRL grant FA9550-24-1-0177. KB was supported by the LSU Huel D. Perkins Doctoral Fellowship.

\section{The Model}
\label{sec:model}

\subsection{Monolayer Graphene}
Monolayer graphene is a layer of carbon atoms arranged on a honeycomb lattice. We define its Bravais lattice $\mathcal{R} = A\mathbb{Z}^2$ and $B$-sublattice vector where
\begin{align*}
   & A = \frac{a}{2}\begin{bmatrix}
        2 & 1 \\
        0 & \sqrt{3}
        \end{bmatrix}, &     \tau_B = \frac{a}{2}\begin{bmatrix}
        1\\
        \frac{\sqrt{3}}{3}
    \end{bmatrix}.
\end{align*}
with lattice constant $a = 1.42\sqrt{3}$.
The reciprocal lattice is $\mathcal{R}^*  = B\mathbb{Z}^2$ with unit cell $\Gamma^*$,
\begin{align}
    &\Gamma^* = B[0,1)^2, &    B = 2\pi A^{-T} = \frac{2\pi}{3a}\begin{bmatrix}
        3 & 0 \\
        -\sqrt{3} & 2\sqrt{3}
    \end{bmatrix}.
\end{align}
Throughout this work, we consider a two orbital model, that is, each atomic site is identified with an orbital. By $\Omega$ we denote the degrees of freedom, that is,
\begin{equation}
\Omega = \mathcal{R}\times\mathcal{A}
\end{equation}
where $\mathcal{A} = \{A,B\}$ are the sublattice indices. The real space Hamiltonian $H$ acts on a wave function $\psi\in\ell^2(\Omega)$
by 
\begin{equation}
\left(H\psi\right)_{(R,\alpha)} = \sum_{(R',\beta)\in\Omega}\left(h_\beta^\alpha\right)_{R-R'}\psi_{(R',\beta)}
\end{equation}
for prescribed discrete hopping functions $h_\beta^\alpha\in\ell^1(\mathcal{R})$. 
We define the unitary Bloch transform $\mathcal{G} : \ell^2(\Omega)\rightarrow L^2(\Gamma^*;\mathbb{C}^2)$ by
\begin{equation}
    (\mathcal{G} \psi)_\alpha(q) = |\Gamma^*|^{-1/2} \sum_{R \in \mathcal{R}} e^{-iq\cdot R}\psi_{R\alpha},
\end{equation}
then
\begin{equation}
[\mathcal{G} H\mathcal{G}^*]^\alpha_\beta(q) = \sum_{R\in\mathcal{R}}e^{-iq\cdot(R+\tau_{\alpha}-\tau_{\beta})}(h_{\beta}^{\alpha})_R.
\end{equation}
Here $\mathcal{G} H\mathcal{G}^*$ is understood as a matrix-valued multiplication operator acting point-wise in $q$.
Conic dispersion patterns appear at $K,K' \in \Gamma^*$,
\begin{align}
    &K = \frac{4\pi}{3a}\begin{bmatrix}
        1\\ 0
    \end{bmatrix}, &     K'  = \frac{2\pi}{3a}\begin{bmatrix}
        1\\
        \sqrt{3}
    \end{bmatrix}.
\end{align}
For example, the Taylor expansion of $\mathcal{G} H\mathcal{G}^*(q)$ about $K$ to leading order is $v_f q \cdot (\sigma_1,-\sigma_2) := v_f(q_x \sigma_1 - q_y\sigma_2)$ with eigenvalues $\pm v_f|q|$ where $v_f$ is the Fermi velocity and $\sigma_j$ is the $j$th standard Pauli matrix.

\subsection{Double-Incommensurate Twisted Trilayer Graphene}
\begin{figure}[htpb]
        \centering
        \includegraphics[width = 1\textwidth, alt = {An image of the atomic configuration of twisted trilayer graphene for $\theta = [-1.4,0,2.8]$}]{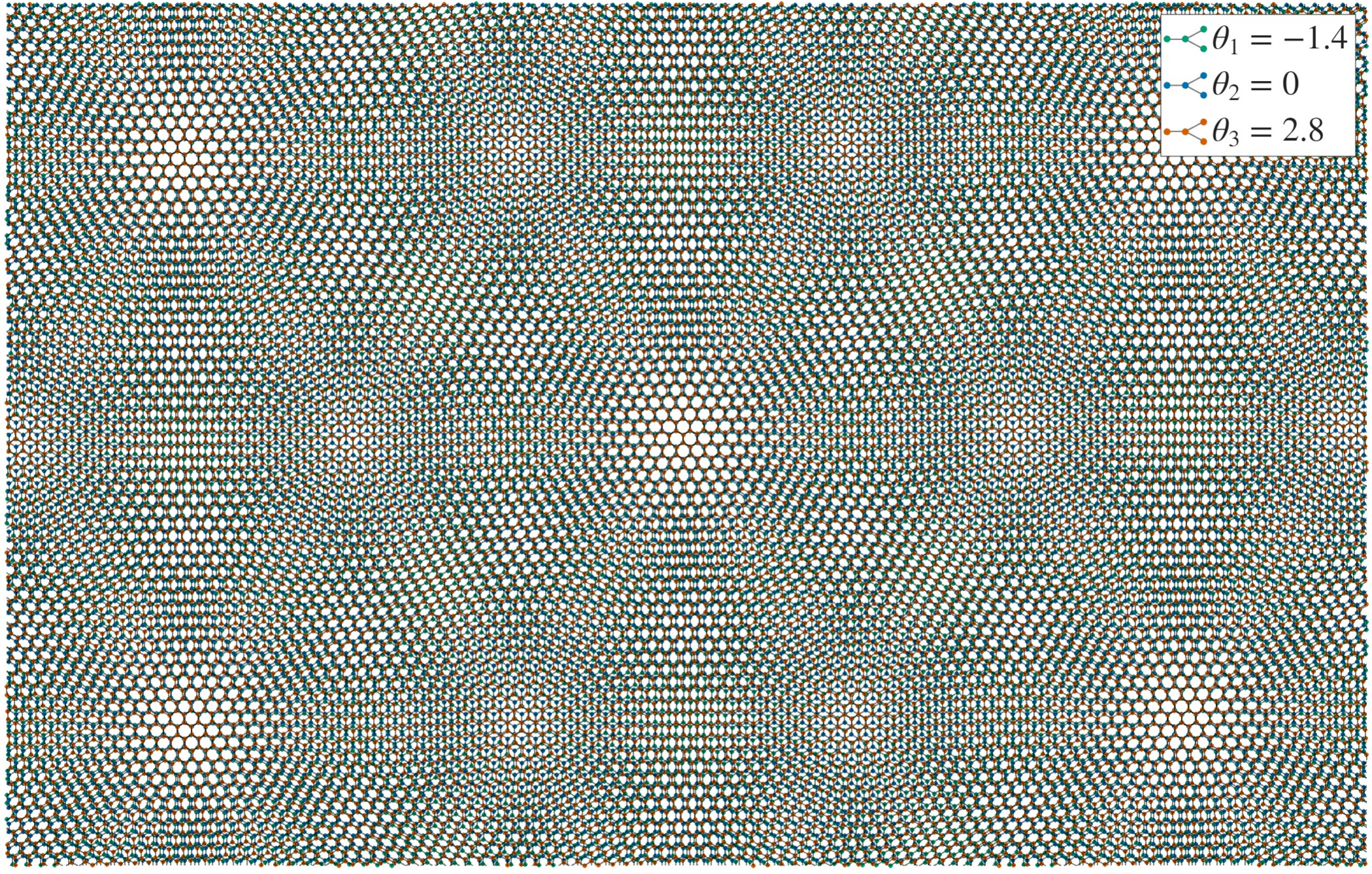}
    \caption{Atomic configuration of twisted trilayer graphene for $\theta = [-1.4,0,2.8]$. The interatomic distance has been scaled from $1.42$ to $0.5$ in order to better display the geometric moir\'{e} pattern. The interfering moir\'{e} patterns preclude the existence of a periodic supercell.}
    \label{fig:multiGraph}
\end{figure}
Let $(\theta_{j})_{j=1}^{N}$ be a sequence of independent twist angles. In addition, let $\mathfrak{R}_\theta$ denote the counterclockwise $2$D rotation matrix, that is,
\begin{equation}
    \mathfrak{R}_\theta = 
    \begin{bmatrix}
        \cos{\theta} & -\sin{\theta}\\
        \sin{\theta} & \cos{\theta}
    \end{bmatrix}.
\end{equation}
Twisted trilayer graphene (TTG) is modeled by a sequence of Bravais lattices $\mathcal{R}_j = \mathfrak{R}_{\theta_j}\mathcal{R}$ with associated unit cells $\Gamma_j =\mathfrak{R}_{\theta_j}\Gamma$ and a sequence of orbital indices $\mathcal{A}_j = \{jA, jB\}$ where $\tau_{jB}=\mathfrak{R}_{\theta_j}\tau_{B}$ and $j\in\{1,2,3\}$. We often denote by $\mathcal{A}$ the union of orbital indices over the layers, that is, $\mathcal{A} = \bigcup_{j=1}^3\mathcal{A}_j$.
The associated sequence of reciprocal lattices is $\mathcal{R}_j^* = \mathfrak{R}_{\theta_j}\mathcal{R}_j^*$ with reciprocal unit cells $\Gamma_j^* = \mathfrak{R}_{\theta_j}\Gamma_j^*$ where $j\in\{1,2,3\}$. The $K$-points of each layer are similarly rotated, that is, $K_j = \mathfrak{R}_{\theta_j}K$ and $K_j'=\mathfrak{R}_{\theta_j}K'$ where $j\in\{1,2,3\}$.

In order to distinguish between stacks of TTG that have some commensuration between subsequent layers and those with no commensuration throughout the entire heterostructure, we introduce the following definition.
\begin{definition}
    For $k\in\{1,2,\ldots,N-1\}$, we say that a sequence of Bravais lattices $(\mathcal{R}_j)_{j=1}^N$ is \textbf{$k$-incommensurate} if 
    \emph{every} sub-sequence $(\mathcal{R}_j)_{j \in I}$ with $|I| \leq k+1$ satisfies
    \begin{equation}
        \sum_{j \in I} R_j = 0 
        \quad \Longleftrightarrow \quad 
        R_j = 0 \text{ for all } j \in I \text{ where } R_j\in\mathcal{R}_j.
    \end{equation}
\end{definition}
As evidence that this is a stronger definition than that of pairwise incommensuration, we give the following example. 
\begin{example}
    The stack of $1$d chains defined by the periodic triple $\left(1, \frac{1}{\pi}, \frac{1}{\pi+1}\right)$ is $2$-incommensurate. However, the reciprocal stack with periodic triple $(2\pi, 2\pi^2, 2\pi+2\pi^2)$ is only $1$-incommensurate.
\end{example}
Consequently, we make the following assumption regarding the TTG stacks we consider in this work:
\begin{assumption}
    \label{ass:incom}
    Both the sequence of Bravais lattices $(\mathcal{R}_j)_{j=1}^{3}$ and the reciprocal sequence $(\mathcal{R}_j^*)_{j=1}^{3}$ are $2$-incommensurate.
\end{assumption}
Henceforth, we refer to $2$-incommensurate TTG as double-incommensurate (with regard to both the minimum number of angles and the minimum number of unique incommensurations). 

By $\Omega$ we denote the degrees of freedom associated with $((\mathcal{R}_j)_{j=1}^3,\mathcal{A})$, that is,
\begin{equation}
    \Omega = \bigsqcup_{j=1}^3\mathcal{R}_j\times\mathcal{A}_j
\end{equation}
where $\mathcal{R}_j$ and $\mathcal{A}_j$ are the layer $j$ Bravais lattices and atomic orbitals, respectively. Observe that
\begin{equation}
    \ell^2(\Omega)\cong\bigoplus_{j=1}^3\ell^2(\Omega_j)\cong\bigoplus_{j=1}^3\ell^2(\mathcal{R}_j)\otimes \mathbb{C}^{2}.    
\end{equation}
The real space Hamiltonian $H$ acts as a $3\times 3$ block operator matrix on a compatibly partitioned wave function $\psi\in\ell^2(\Omega)$, that is, 
\begin{gather}
    H\psi = 
    \begin{bmatrix}
        H_{11} & T_{12} & T_{13}\\
        T_{12}^* & H_{22} & T_{23}\\
        T_{13}^* & T_{23}^* & H_{33}
    \end{bmatrix}
    \begin{bmatrix}
        \psi_1\\ \psi_2\\ \psi_3
    \end{bmatrix}
    =
    \begin{bmatrix}
        H_{11}\psi_1+T_{12}\psi_2+T_{13}\psi_3\\
        T_{12}^*\psi_1+H_{22}\psi_2+T_{23}\psi_3\\
        T_{13}^*\psi_1+T_{23}^*\psi_2+H_{33}\psi_3
    \end{bmatrix}.
\end{gather}
where
\begin{align}
    \left[H_{jk}\psi\right]_{(R,j\alpha)} &= \delta_{k}^{j}\sum_{(R,j\beta)\in\mathcal{R}_j}[h_{j\beta}^{j\alpha}]_{R-R'}\psi_{(R',j\beta)}\\
    &+(1-\delta_{k}^{j})\sum_{(R',k\beta)\in\mathcal{R}_{k}}h_{k\beta}^{j\alpha}(R+\tau_{j\alpha}-R'-\tau_{k\beta})\psi_{(R',k\beta)}
\end{align}
for prescribed interlayer hopping functions $h_{k\beta}^{j\alpha}\in C(\R^2;\C)$.

We are interested in bounded operators that satisfy certain decay estimates.
\begin{assumption}
    \label{ass:bound}
    There exists $\gamma\in\R_+^2$ such that for all $(j\alpha, k\beta)\in\mathcal{A}^2$, we have
    \begin{gather}
        \|h_{k\beta}^{j\alpha}\| = \sup_{x\in\mathbb{R}^d}|h_{k\beta}^{j\alpha}(x)|e^{\gamma_1\|x\|_2}+\sup_{\xi\in\mathbb{R}^d}|\widehat{h}_{k\beta}^{j\alpha}(\xi)|e^{\gamma_2\|\xi\|_2} \quad\text{and}\\
        \|h_{j\beta}^{j\alpha}\|_{\ell_\gamma^\infty(\mathcal{R}_j)}=\sup_{R\in\mathcal{R}_j}|(h_{j\beta}^{j\alpha})_R|e^{\gamma\|R\|_2}.
    \end{gather}
\end{assumption}
Here $\hat{h}_{k\beta}^{j\alpha}$ is the continuous Fourier transform
\begin{equation}
    \hat{h}_{k\beta}^{j\alpha}(\xi)=\frac{1}{4\pi^2}\int_{\mathbb{R}^2}h_{k\beta}^{j\alpha}(x)e^{-i\xi\cdot x}dx.
\end{equation}

\begin{remark}
 \zcref{ass:bound} guaranties that the discrete Fourier transform of the intralayer hopping functions and the continuous Fourier transform of the interlayer hopping functions are simultaneously analytic within a strip of width $\min(\gamma_1,\gamma_2)$.
\end{remark}
As with most physical operators, we require that $H$ be self-adjoint. Further, we use the model assumption that only nearest neighbor layers interact as in \cite{fang_electronic_2016}, and hence $3$-layer Hamiltonian $H$ can be represented as an $3\times 3$ tridiagonal block operator.
\begin{assumption}
    \label{ass:herm}
    For all $(j\alpha, k\beta)\in\mathcal{A}^2$ such that $j\neq k$, the following hermiticity conditions hold
    \begin{equation}
        (h_{j\beta}^{j\alpha})_{R} = \overline{(h_{j\alpha}^{j\beta}})_{-R}\;\;\text{and}\;\; h_{k\beta}^{j\alpha}(x) = \overline{h_{j\alpha}^{k\beta}(-x)}
    \end{equation}
    for all $R\in\mathcal{R}_{j}$ and $x\in\R^2$. Further, $T_{13} = O$ and $T_{31} = O$ where $O$ denotes the zero operator.
\end{assumption}

The primary goal of this work is to produce a convergent algorithm to approximate the regularized density of states and the momentum LDoS, which is a comparable object to quasi-band structure. Let $\delta_\epsilon:\mathbb{R}\to\mathbb{R}$ denote a Gaussian with standard deviation $\epsilon$, 
\begin{gather}
   \delta_\epsilon(E) = \frac{1}{\epsilon\sqrt{2\pi}}e^{-\frac{E^2}{2\epsilon^2}}.
\end{gather}
Additionally, let $\mathcal{O}(H)$ denote the set of all complex valued functions $f$ analytic in some open neighborhood $U\subset\C$ such that $\sigma\left(H\right)\subset U$. Since the existence of the density of states measure is unknown for incommensurate systems, we introduce the regularized density of states.
\begin{definition}
    Given the TTG tight-binding Hamiltonian $H$ described above, the \textit{real regularized density of states} $\mathcal{D}_\epsilon:\mathbb{R}\to\mathbb{C}$ is defined by
    \begin{equation}
        \mathcal{D}_\epsilon(E) = \underline{\operatorname{Tr}}\;\delta_\epsilon(E-H)
    \end{equation}
    where $\underline{\operatorname{Tr}}:\mathcal{O}(H)\to\mathbb{C}$ is a \textit{thermodynamic limit trace} defined by
    \begin{equation}
        \underline{\operatorname{Tr}} f(H) = \lim_{r\to\infty}\frac{1}{|\Omega_{r}|}\sum_{(R,j\alpha)\in\Omega_r}[f(H)]_{(R,j\alpha)}^{(R,j\alpha)}
    \end{equation}
    and
    \begin{equation}
        \Omega_r = \{(R,j\alpha)\in\Omega:\|R\|_2<r\}.
    \end{equation}
\end{definition}
\begin{remark}The limit $\epsilon \rightarrow 0$ defines the density of states as a spectral measure, and when it exists, $\mathcal{D}(E) = \lim_{\epsilon \rightarrow 0}\mathcal{D}_\epsilon(E)$ is the density of states function or the Radon-Nikodym derivative of the density of states measure. For incommensurate materials, the existence of this Radon-Nikodym derivative is unknown. For applications, however, the regularized $DoS$ with $\epsilon$ sufficiently small is sufficient.
\end{remark}

\section{Momentum and Reciprocal Space}
\label{sec:momentum}

In this section, we introduce the momentum space transformation that describes how momenta couple. We also introduce a reciprocal space formulation, which is an infinite matrix capturing all scattering channels.
\begin{definition}
     \label{defn:Bloch}
    For each $j\in\{1,2,3\}$, the \textit{layer Bloch transform} $\mathcal{G}_j:\ell^2(\Omega_j)\to L^2(\Gamma_j^*;\mathbb{C}^2)$ is defined by
    \begin{equation}
        [\mathcal{G}_j\psi_j]_{j\alpha}(q) = \frac{1}{c_j^*}\sum_{R\in\mathcal{R}_j}e^{-iq\cdot(R+\tau_{j\alpha})}\psi_{(R, j\alpha)}\quad\text{where}\quad c_j^* = \sqrt{|\Gamma_j^*|}
    \end{equation}
    The \textit{Bloch transform} $\mathcal{G}:\ell^2(\Omega)\to \oplus_{j=1}^3L^2(\Gamma_j^*;\mathbb{C}^6)$ is defined by
    \begin{equation}
        \mathcal{G} = \bigoplus_{j=1}^{3}\mathcal{G}_j.
    \end{equation}
\end{definition}
It is easily seen that $\mathcal{G}_j$ is quasi-periodic, that is,
\begin{gather}
    [\mathcal{G}_j\psi_{j}]_\alpha(q+G_j)= e^{-iG_j\cdot\tau_{j\alpha}}[\mathcal{G}_j\psi_j]_{\alpha}(q).
\end{gather}
This leads us to our first proposition, where we obtain the momentum space representation of the Hamiltonian. We use the notation $f(\cdot)$ to indicate a multiplication operator, for example $[f(\cdot)\tilde\psi](q) = f(q)\tilde\psi(q)$.
\begin{proposition}
    \label{prop:momentum}
    Given a real space Hamiltonian $H\in\mathcal{L}(\ell^2(\Omega))$ for double incommensurate TTG, the \textit{momentum space Hamiltonian} $\widetilde{H}\in\mathcal{L}(L^2(\Gamma_{\oplus}^*;\mathbb{C}^6))$ is the unique operator satisfying the intertwining relation
    \begin{equation}
        \mathcal{G} H\psi = \widetilde{H}\mathcal{G}\psi.
    \end{equation}
    In particular,
    \begin{equation}
        \label{eq:momentum}
        \widetilde{H}_{k\beta}^{j\alpha} = \delta_{k}^{j}\sum_{R\in\mathcal{R}R_{j}}e^{-i(\cdot)\cdot(R+\tau_{j\alpha}-\tau_{j\beta})}\left(h_{j\beta}^{j\alpha}\right)_{R}+(1-\delta_{k}^{j})c_j^*c_k^*\sum_{G\in\mathcal{R}_{j}^*}e^{iG\cdot\tau_{j\alpha}}\widehat{h}_{k\beta}^{j\alpha}(\cdot+G)T_{G}
    \end{equation}
    for all $(j\alpha,k\beta)\in\mathcal{A}^2$ where $\widehat{h}_{k\beta}^{j\alpha}$ is the continuous Fourier transform
    \begin{equation}
        \widehat{h}_{k\beta}^{j\alpha}(\xi)=\frac{1}{(2\pi)^2}\int_{\mathbb{R}^2}h_{k\beta}^{j\alpha}(x)e^{-i\xi\cdot x}dx.
    \end{equation}
\end{proposition}
Taking advantage of the ergodic structure of the momentum Hamiltonian $\widetilde{H}$, we unfold onto an infinite double-incommensurate lattice to produce the reciprocal Hamiltonian $\widehat{H}(q)$ where the momenta $q$ will center the scattering description. To define the unfolding, we create the reciprocal degree of freedom space denoted $\Omega^*$,
\begin{align}
   &\Omega^* = \bigsqcup_{j=1}^{N}\Omega_{j}^*, & \Omega_{j}^* =  \kappa_j^*\times\mathcal{A}_j,\hspace{.5cm} &\kappa_j^* = \left\{G\in\bigoplus_{t = 1}^{N}\mathcal{R}_t^*: G_j = 0\right\}.
\end{align}
Each degree of freedom denotes a momenta, and the operator $\widehat{H}$ we have yet to derive will describe the hopping functions between the momenta indexed by $\Omega^*$.
Let $\chi_j^a$ denote the space of analytic square integrable complex vector valued functions defined on $\Gamma_j^*$.  $\chi_j^a$ comes equipped with the standard $L^2$ inner product, that is,
\begin{equation}
    (\widetilde{\psi},\widetilde{\phi})_{\chi_j^a} = \sum_{j\alpha\in\mathcal{A}_j}\int_{\Gamma_j^*}\overline{\widetilde{\psi}_{j\alpha}(q)}\widetilde{\phi}_{j\alpha}(q)dq
\end{equation}
for all $\widetilde{\psi},\widetilde{\phi}\in\widetilde{\chi_j}^a$.
With slight abuse of notation, we use $\widetilde{\psi}_j\in\chi_j^a$ also to denote its periodic extension.  
We define an unfolding operator as follows.
\begin{definition}
    \label{defn:unfolding}
    For each $q\in\mathbb{R}^2$ and $j\in\{1,2,3\}$, the \textit{unfolding operators} $\mathcal{E}_{jq}:\chi_j^a\to\mathcal{E}_{jq}(\chi_j^a)\subset\ell^\infty(\Omega_j^*)$ and $\mathcal{E}_{q}:\chi^a\to\mathcal{E}_q(\chi^a)$ are defined by
    \begin{align}
       &\left(\mathcal{E}_{jq}\widetilde{\psi}_j\right)_{(G,j\alpha)}  = \widetilde{\psi}_{j\alpha}\left(q+\sum_{t = 1}^{3}G_t\right), & \mathcal{E}_q = \bigoplus_{j=1}^{3}\mathcal{E}_{jq}.
    \end{align}
    For each $j\in\{1,2,3\}$, the \textit{layer unfolding operator} $\mathcal{E}_j:\chi_j^a\to \mathcal{H}_j$ is defined by 
    \begin{equation}
        \mathcal{E}_j = \int_{\Gamma_j^*}^{\oplus}\mathcal{E}_{jq}dq\quad\text{where}\quad
        \mathcal{H}_j = \int_{\Gamma_j^*}^\oplus\mathcal{E}_{jq}(\chi_j^a).
    \end{equation}
\end{definition}
In order to define $\mathcal{E}_{jq}^*$, we introduce an ergodic inner product space.
\begin{lemma}
    \label{lem:ergProd}
    Under \zcref{ass:incom}, the binary function $(\cdot,\cdot)_{j,\mathrm{erg}}:\mathcal{E}_{jq}(\chi_j^a)\times\mathcal{E}_{jq}(\chi_j^a)\to\C$, defined by
    \begin{equation}
        (\widehat{\psi},\widehat{\phi})_{j,\mathrm{erg}} = \lim_{r\to\infty}\frac{1}{|\Omega_{jr}^*|}\sum_{(G,j\alpha)\in\Omega_{jr}^*}\overline{\widetilde{\psi}_{(G,j\alpha)}}\widetilde{\phi}_{(G,j\alpha)},
    \end{equation}
    is well defined and is an inner product, which we denote the ergodic inner product. Moreover, we have the relation
    \begin{equation}
        (\mathcal{E}_{jq}\tilde \psi,\mathcal{E}_{jq}\tilde \phi)_{j,\mathrm{erg}} = |\Gamma_j^*|^{-1} (\widetilde{\psi},\widetilde{\phi})_{\chi_j^a}.
    \end{equation}
\end{lemma}
This leads us to our second proposition, where we obtain the reciprocal representation of the Hamiltonian.
\begin{proposition}
    \label{prop:reciprocal}
    Given a momentum space Hamiltonian $\widetilde{H}\in\mathcal{L}(\oplus_{j=1}^3L^2(\Gamma_j^*;\mathbb{C}^2))$ for double incommensurate TTG, for each $q\in\mathbb{R}^2$, the \textit{reciprocal Hamiltonian} $\widehat{H}_q\in\mathcal{L}(\mathcal{E}_q\chi^a)$ satisfies the intertwining relation
    \begin{equation}
        \left(\mathcal{E}_q\widetilde{H}\widetilde{\psi}\right)_{(G,j\alpha)} = \left(\widehat{H}_q\mathcal{E}_q\widetilde{\psi}\right)_{(G,j\alpha)}.
    \end{equation}
    In particular,
    \begin{equation}
        \label{eq:reciprocal}
            \left[\widehat{H}_q\right]_{(G',k\beta)}^{(G,j\alpha)} = \delta_{k}^j\delta_{G'}^{G}\widetilde{h}_{j\beta}^{j\alpha}\left(q+G_k+G_l\right)+(1-\delta_k^j)\delta_{G_l'}^{G_l}T_{(G_j',\alpha)}^{(G_k,\beta)}\hat{h}_{k\beta}^{j\alpha}\left(q+G_j'+G_k+G_l\right)
    \end{equation}
    for all $((G,j\alpha),(G',k\beta))\in\Omega_j^*\times\Omega_k^*$ where
    \begin{gather}
        \widetilde{h}_{j\beta}^{j\alpha}(q) =\sum_{R\in\mathcal{R}_j}e^{-iq\cdot(R+\tau_{j\alpha}-\tau_{j\beta})}\left(h_{j\beta}^{j\alpha}\right)_{R}\quad\text{and}\quad
        T_{(G_j',\alpha)}^{(G_k,\beta)} = c_j^*c_k^*e^{i(G_j'\cdot\tau_{j\alpha}-G_k\cdot\tau_{k\beta})}.
    \end{gather}
\end{proposition}
We also define the momentum local density of states (LDoS) as follows.
\begin{definition}
     The \textit{momentum regularized LDoS per layer} $\widehat{\mathcal{D}}_{j\epsilon}:\mathbb{R}\times\mathbb{R}^2\to\mathbb{R}_+$ is defined by
    \begin{equation}
        \widehat{\mathcal{D}}_{j\epsilon}(E, q) = \frac{1}{|\mathcal{A}_j|}\sum_{j\alpha\in\mathcal{A}_j}[\delta_\epsilon(E-\widehat{H}_q)]_{(0,j\alpha)}^{(0,j\alpha)}.
    \end{equation}
    The \textit{averaged momentum regularized LDoS} $\widehat{\mathcal{D}}:\mathbb{R}\times\mathbb{R}^2\to\mathbb{R}_+$ is defined by
    \begin{equation}
        \widehat{\mathcal{D}}_\epsilon(E,q) = \frac{1}{|\mathcal{A}|}\sum_{j\alpha\in\mathcal{A}}[\delta_\epsilon(E-\widehat{H}_q)]_{(0,j\alpha)}^{(0,j\alpha)}.
    \end{equation}
\end{definition}
 We define the thermodynamic limit trace  $\underline{\operatorname{Tr}}:\mathcal{O}(\widehat{H}_q)\to\mathbb{C}$ by
    \begin{equation}
        \underline{\operatorname{Tr}}\; f(\widehat{H}_q) = \lim_{r\to\infty}\frac{1}{|\Omega_r^*|}\sum_{(G,j\alpha)\in\Omega_r^*}[f(\widehat{H}_q)]_{(G,j\alpha)}^{(G,j\alpha)}
    \end{equation}
    for
    \begin{equation}
        \Omega_{r}^* = \left\{(G,j\alpha)\in\Omega^*: \left\|\bigoplus_{k\neq j}G_k\right\|_2<r\right\}.
    \end{equation}
    We define the reciprocal space DoS as
        \begin{equation}
        \widehat{\mathcal{D}}_{\epsilon}(E) = \underline{\operatorname{Tr}}\;\delta_\epsilon(E-\widehat{H}_0).
    \end{equation}

Before we proceed to the equivalence of the different regularized densities of states, we introduce the following computable complex linear functional acting on the collection of reciprocal matrices parameterized by starting momenta $\widehat{H} = \{\widehat{H}(q)\}_{q \in \mathbb{R}^2}$:
\begin{definition}
    \label{defn:functional}
    We define the set of generated operators $\mathcal{A}\left(\widehat{H}\right)$ by
    \begin{equation}
        \mathcal{A}\left(\widehat{H}\right) = \left\{f\left(\widehat{H}\right): f\in\mathcal{O}\left(\widehat{H}\right)\right\}.
    \end{equation}
    The complex linear functional $\mathcal{T}: \mathcal{A}\left(\widehat{H}\right)\to\C$ is defined by
    \begin{equation}
        \mathcal{T}f\left(\widehat{H}\right) = \nu^*\sum_{j\alpha\in\mathcal{A}}\int_{\Gamma_j^*}\left[f\left(\widehat{H}_q\right)\right]_{(0,j\alpha)}^{(0,j\alpha)}dq.
    \end{equation}
\end{definition}
The following theorem lays the foundation for computing the regularized density of states via the complex linear functional.
\begin{theorem}
    \label{thm:equivtrace}
    Under \zcref{ass:bound,ass:incom}, the relation
    \begin{gather}
        \underline{\operatorname{Tr}} \; f(\widehat{H}_0)=\mathcal{T} f(\widehat{H})=\underline{\operatorname{Tr}}\; f(H)
    \end{gather}
    holds for all $f\in\mathcal{O}(\widehat{H})$.
\end{theorem}
We immediately obtain the following corollary, where the conjugate symmetry of $\delta_\epsilon$ and the hermiticity of $\widehat{H}$ yield only real values.
\begin{corollary}
    \label{cor:trtofun}
    Under \zcref{ass:bound,ass:incom}, the relation
    \begin{equation}
        \widehat{\mathcal{D}}_\epsilon(E) = \mathcal{T}\delta_\epsilon(E-\widehat{H}) = \mathcal{D}_{\epsilon}(E)
    \end{equation}
    holds for all $E\in\mathbb{R}$.
\end{corollary}

\section{The Algorithm}
\label{sec:algorithm}

In this section, we outline the algorithm for computing the density of states near the Fermi energy. We note that the algorithm to obtain the momentum LDoS is identical, except no integral discretization is required, as the momentum LDoS is simply the integrand in the definition of $\widehat{\mathcal{D}}_\epsilon(E)$. We introduce the algorithm through a sequence of approximations, and for each piece, we quantify the error.
\subsection{Integral discretization}
We first discretize the linear functional $\mathcal{T}$. Define the uniform discretization $\mathfrak{S}(N)$ of the layer $j$ Brillouin zone $\Gamma_j^*$ by
\begin{equation}
    \mathfrak{S}_j(N) = \frac{|\Gamma_j^*|}{N^2}B_j\left([0,N-1)^2\cap\mathbb{Z}^2\right).
\end{equation}
Then the uniform discretization $\mathcal{T}_{N}$ of the linear functional $\mathcal{T}:\mathcal{O}(\widehat{H})\to\mathbb{R}$ is given by
\begin{equation}
    \mathcal{T}_{N} f(\widehat{H}) = \frac{\nu^*}{N^2}\sum_{j=1}^3|\Gamma_j^*|\sum_{\substack{j\alpha\in\mathcal{A}_j\\q\in\mathfrak{S}_j(N)}}[f(\widehat{H}_q)]_{(0,j\alpha)}^{(0,j\alpha)}.
\end{equation}
Naturally, this leads to the following lemma regarding the relative error of the uniform discretization.
\begin{lemma}
    \label{lem:errorN}
    Under \zcref{ass:bound}, the following bound holds
    \begin{equation}
        |\mathcal{T} f(\widehat{H})-\mathcal{T}_{N} f(\widehat{H})| \lesssim \sup_{z\in \mathcal{C}}|f(z)|e^{-\widetilde{\rho} N}
    \end{equation}
    for all $f\in\mathcal{O}(\widehat{H})$, and
    \begin{equation}
        \widetilde{\rho} \leq \min\left(\frac{\gamma_1}{2},\frac{\epsilon}{C}\right)\min_{j\in\{1,2,3\}}\|A_j^{-T}\|_2^{-1}
    \end{equation}
    where $\epsilon$ is the distance from the contour $\mathcal{C}$ to the spectrum of $\widehat{H}$ and $C\in\mathbb{R}_+$ is some computable constant.
\end{lemma}
In particular, for $f = \delta_\epsilon$, we obtain the bound
\begin{equation}
    |\widehat{\mathcal{D}}_\epsilon(E)-\widehat{\mathcal{D}}_\epsilon^{N}(E)|\lesssim \epsilon^{-1}e^{-\widetilde{\rho}N}
\end{equation}
where we denote by $\widehat{\mathcal{D}}_\epsilon^{N}(E)$ the discretized regularized density of states, that is,
\begin{equation}
    \widehat{\mathcal{D}}_\epsilon^{N}(E) = \mathcal{T}_{N} \delta_\epsilon(E-\widehat{H}).
\end{equation}

\subsection{Hopping function truncation}
In order to obtain a a sparse matrix, we restrict the hopping distance to be within a ball of radius $\tau$. Let
\begin{equation}
    g_{\tau}(x) = 
    \begin{cases}
            1, & \|x\|_2\leq\tau-\delta,\\
            \frac{e^{-\frac{\delta}{\tau-\|x\|_2}}}{e^{-\frac{\delta}{\tau-\|x\|_2}}+e^{-\frac{\delta}{\|x\|_2-(\tau-\delta)}}}, & \tau-\delta<\|x\|_2<\tau,\\
            0, & \|x\|_2\geq\tau
    \end{cases}
\end{equation}
for some fixed width $\delta$. Clearly, $g_{\tau}:\R^2\to\R$ is a smoothly truncated, compactly supported bump function. Define the associated truncation $\widehat{H}^\tau_q$ by
\begin{gather}
    [\widehat{H}_q^{\tau}]_{(G',k\beta)}^{(G,j\alpha)}:=\delta_k^j\delta_{G'}^G\sum_{R\in\mathcal{R}_{j\tau}}e^{-i(q+G_k+G_l)\cdot (R+\tau_{j\alpha}-\tau_{j\beta})}(h_{j\beta}^{j\alpha})_R\\
    +(1-\delta_k^j)\delta_{G_l'}^{G_l}T_{(G_j',\alpha)}^{(G_k,\beta)}\widehat{h}_{k\beta}^{j\alpha}(q+G_j'+G_k+G_l)g_\tau(q+G_j'+G_k+G_l)
\end{gather}
for all $q\in\mathbb{R}^2$. We denote by $\widehat{\mathcal{D}}_\epsilon^{(N,\tau)}(E)$ the discretized regularized density of states with respect to $\widehat{H}_q^\tau$, that is,
\begin{equation}
    \widehat{\mathcal{D}}_\epsilon^{(N,\tau)}(E) = \mathcal{T}_{N}\delta_\epsilon(E-\widehat{H}_q^\tau).
\end{equation}
By directly appealing to the Estimation Lemma, we obtain the following result on the relative error introduced by $\tau$-truncation.

\begin{lemma}
    For double incommensurate TTG, the relative error of the regularized density of states introduced by $\tau$-truncation satisfies
    \begin{equation}
        |\widehat{\mathcal{D}}_\epsilon^{N}(E)-\widehat{\mathcal{D}}_\epsilon^{(N,\tau)}(E)|\lesssim\epsilon^{-1}e^{-(\tau-\delta)\min(\gamma_1,\gamma_2)}.
    \end{equation}
\end{lemma}

\subsection{$W$-truncation: removing momenta far from $K$-point.}

In order to obtain a finite matrix, we must restrict the reciprocal degrees of freedom. This is done in two stages. For the first stage, we introduce the $W$-truncation. For a generic set $X$, we denote by $\mathcal{P}(X)$ the power set of $X$, that is,
\begin{equation}
    \mathcal{P}(X) = \{Y:Y\subset X\}.
\end{equation}
Let $W\in\mathbb{R}_+^3$ be a vector of $W$-truncation parameters. The layer $W$-truncation function $\mathcal{W}_j^*:\Gamma_j^*\times\mathbb{R}_+\to\mathcal{P}(\kappa_j^*)$ restricts the reciprocal lattice vector pairs of layer $j$ based on the shifted momentum-radius $W_j$:
\begin{equation}
    \mathcal{W}_{j}^*(q,W_j)=\{G\in\kappa_j^*:\|q-\widetilde{K}+(B_k-B_j)B_k^{-1}G_k+(B_l-B_j)B_l^{-1}G_l\|_2<W_j\}\times\mathcal{A}_j
\end{equation}
where $q\in B_{W_j}(K_j)\cup B_{W_j}(K_j')$ and $\widetilde{K}$ is the $K$-point closest to $q$, where we assume $W_j < \omega_j := \frac{1}{2}|K_j-K_j'|$. Let $\omega = \min_j\omega_j$. The total $W$-truncation function $\mathcal{W}^*:\bigcup_{j=1}^3\Gamma_j^*\times[0,\omega)^3\to\mathcal{P}(\kappa^*)$ is defined as the union of the layer $W$-truncation functions:
\begin{equation}
    \mathcal{W}^*(q, W) = \bigcup_{j=1}^3\mathcal{W}_j^*(q, W_j).
\end{equation}
We include the following figure, which illustrates how increasing $W$ results in movement along the $z$-axis of the Dirac cones.

\begin{figure}[ht]
    \centering
    \begin{subfigure}[b]{0.45\textwidth}
        \centering
        \includegraphics[width = \textwidth, alt = {A side view of the three dirac cones as showing that an increase of $W$ moves up the $z$-axis.}]{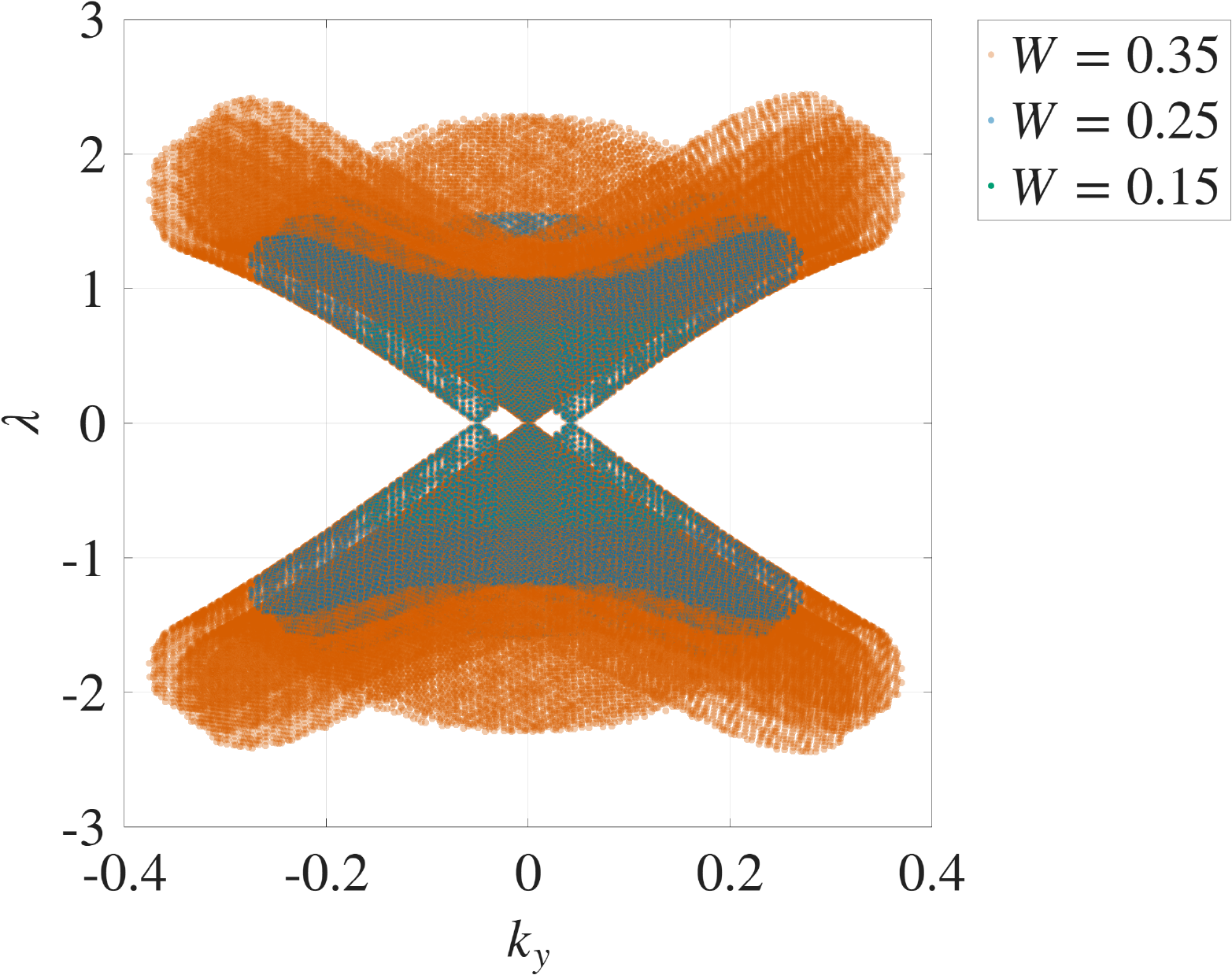}
        \caption{Side-view}
    \end{subfigure}
    \hfill
    \begin{subfigure}[b]{0.45\textwidth}
        \centering
        \includegraphics[width = \textwidth, alt = {A top-down view of the three dirac cones provided to better see the change in $W$.}]{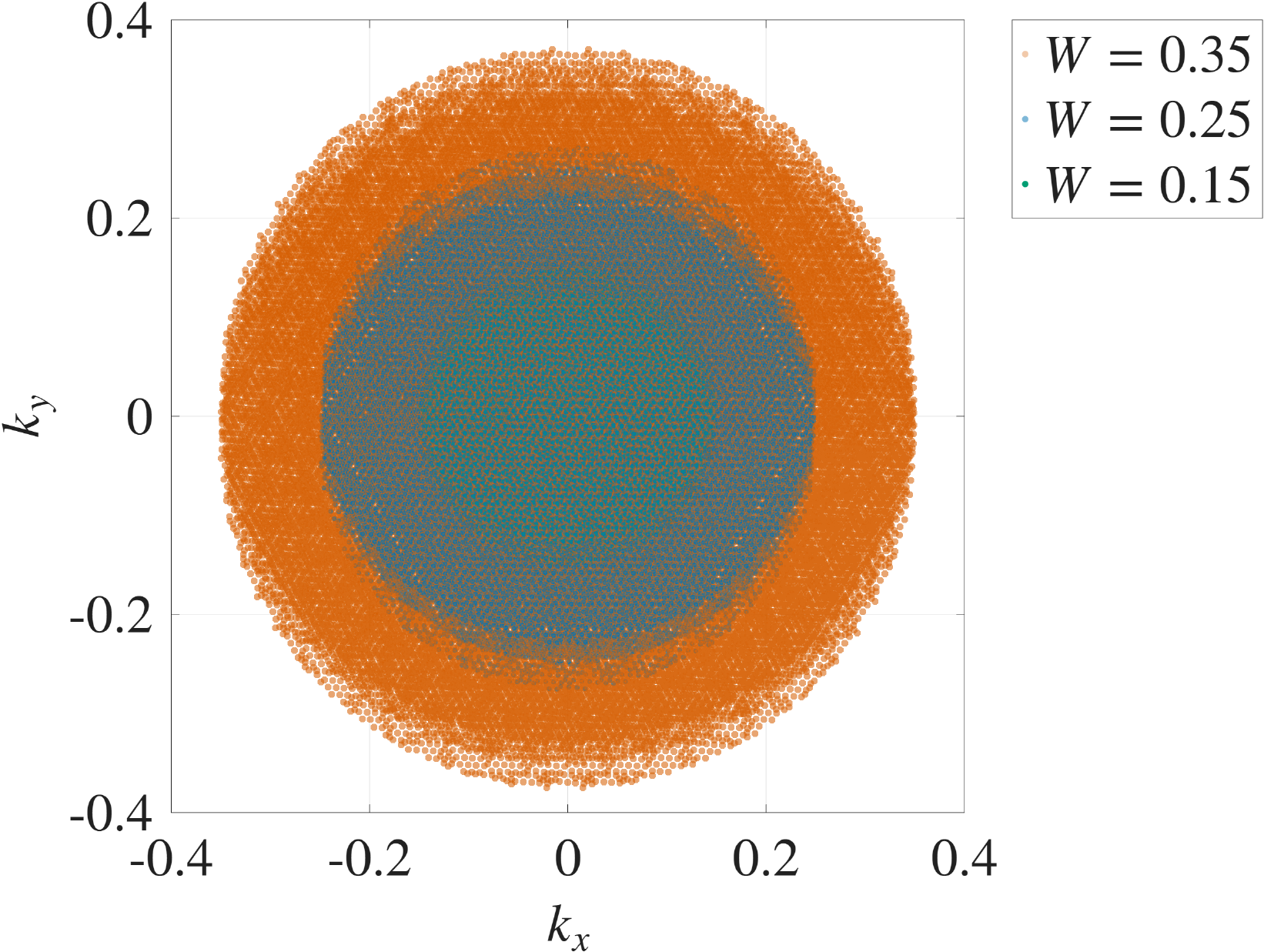}
        \caption{Top-Down}
    \end{subfigure}
    \caption[Increasing $W$ moves up the $z$-axis of the Dirac cones]{Fixing $L = 20$ and $\theta = \left[-\frac{\sqrt{3}}{2}, 0, \frac{\sqrt{2}}{2}\right]$, increasing $W$ results in movement along the $z$-axis of the monolayer Dirac cones. In (a) we give a side-view of all three monolayer Dirac cones as $W$ increases from 0.15 (green) to 0.35 (orange). In (b) we give a top-down view of all three monolayer dirac cones as $W$ increases from 0.15 (green) to 0.35 (orange).}
    \label{fig:cones}
\end{figure}

Next, we define the projection operator, which will produce the $W$-truncated fiber reciprocal Hamiltonian. For each $q\in\bigcup_{j=1}^3\Gamma_j^*$ and $W\in\mathbb{R}_{+}^3$, let $J_W(q):\ell^2(\mathcal{W}^*(q,W))\to\ell^2(\Omega^*)$ be the operator that extends a state by zero outside of $\mathcal{W}^*(q, W)$. We define the $W$-truncation operator $P_W(q)\in\mathcal{L}(\ell^2(\Omega^*))$ as the orthogonal projection:
\begin{equation}
    P_W(q) = J_W(q)J_W^*(q).
\end{equation}
Specifically, for any state $\psi\in\ell^2(\Omega^*)$, the action of $P_W(q)$ is given by:
\begin{equation}
    (P_W(q)\psi)_x = 
    \begin{cases}
        \psi_x, & x\in\mathcal{W}^*(q, W),\\
        0, & x\in\Omega^*\setminus\mathcal{W}^*(q, W).
    \end{cases}
\end{equation}
The $W$-truncation of the $\tau$-truncated fiber reciprocal Hamiltonian $\widehat{H}_q^{(\tau, W)}$ is defined as
\begin{equation}
    \widehat{H}_q^{(\tau,W)} = P_W(q)\widehat{H}_q^{\tau}P_W(q)\quad\text{for all}\quad q\in\mathbb{R}^2
\end{equation}
and the associated discretized regularized density of states $\widehat{\mathcal{D}}_{\epsilon}^{(N,\tau,W)}$ is defined as
\begin{equation}
    \widehat{\mathcal{D}}_{\epsilon}^{(N,\tau,W)}(E) = \mathcal{T}_{N}^W\delta_\epsilon(E-\widehat{H}_q^{(\tau,W)})\quad\text{for all}\quad E\in\mathbb{R}
\end{equation}
where $\mathcal{T}_{N}^W$ is the restriction of the discretized functional $\mathcal{T}_{N}$ to the $W$-balls about the $K$-points of each layer, that is,
\begin{gather}
    \mathcal{T}_{N}^Wf(\widehat{H}^{(\tau,W)}) = \frac{\nu^*}{N^2}\sum_{j=1}^3|\Gamma_j^*|\sum_{\substack{j\alpha\in\mathcal{A}\\q\in\mathfrak{S}_j(N,W)}}[f(\widehat{H}_q^{(\tau, W)})]_{(0,j\alpha)}^{(0,j\alpha)}\quad\text{with}\\ \mathfrak{S}_j(N,W) = \bigcup_{\widetilde{K}\in\{K_j,K_j'\}}\{q\in\mathfrak{S}_j(N): \|q-\widetilde{K}\|_2<W\}.
\end{gather}
Here we are using $\widehat{H}^{(\tau, W)}=\{\widehat{H}_q^{(\tau, W)}\}_{q\in\mathbb{R}^2}$. Since
\begin{equation}
    \mathcal{W}^*(q, W) = \emptyset
\end{equation}
for all $q\in \R^2\setminus B_W(0)$, we have that
\begin{equation}
    \mathcal{T}_{N}^{W}f(\widehat{H}^{(\tau, W)}) = \mathcal{T}_{N}f(\widehat{H}^{(\tau, W)}).
\end{equation}

Before we discuss the relative error introduced by $W$-truncation, we first need some preliminary results on the connectedness of $\mathcal{W}^*(q, W)$. We begin by rigorously defining connectedness for the reciprocal degrees of freedom. As we develop connectedness, we obtain upper and lower bounds on $W$ related to the geometry of the lattices and the existence of non-trivial paths between momenta of the same layer. We now define $(\mu,d)$-connectedness for $d(\cdot,\cdot)$ a notion of distance called a premetric, and $\mu \in \mathbb{R}_+$.
\begin{definition}
    We define a symmetric premetric space $(X, d)$ to be \textit{$(\mu, d)$-connected} if, for every pair $(x,y)\in X^2$, there exists a \textit{$(\mu, d)$-path} from $x$ to $y$, that is, a finite sequence $(z_t)_{t=1}^N\subset X$ such that
    \begin{equation}
        z_1 = x,\;\;z_N = y\;\;\text{and}\;\;d(z_{t+1},z_t)\leq\mu\;\;\text{for all}\;\;t\in\{1,2,\ldots, N-1\}.
    \end{equation}
    We denote by $\mathcal{P}_{xy}^\mu(X, d)$ the subset of all $(\mu, d)$-paths from $x$ to $y$. A symmetric premetric space $(X, d)$ is \textit{$(\mu, d)$-connected relative to $x\in X$} if $\mathcal{C}_x^\mu(X, d) = X$, where $\mathcal{C}_x^\mu:(X, d)\to\mathcal{P}(X)$ is the \textit{$\mu$-reachability operator at $x$} defined by
    \begin{equation}
        \mathcal{C}_x^\mu(X, d) = \{y\in X: \mathcal{P}_{xy}^\mu(X, d)\neq\emptyset\}.
    \end{equation}
\end{definition}

Clearly, if $x$ is $(\mu, d)$-path connected to $y$, then $y$ is $(\mu, d)$-path connected to $x$ by the reversal. As in \cite{massatt_incommensurate_2018}, we define additional sets describing which $(\mu_j,d_j)$-connected indices have momenta $W_j$-close to the Dirac point closest to $q$ for layer $j$. 
\begin{definition}
    Let $W\in\mathbb{R}_+^3$ be a vector of $W$-truncation parameters, $\mu\in\mathbb{R}_+^3$ be a vector of maximal path lengths, and $d = \{d_1, d_2, d_3\}$ be a collection of metrics. The \textit{momentum window selection} function $\lambda_j:\Gamma_j^*\times\mathbb{R}_+\to \mathcal{P}(\Omega_j^*)$ is defined by
    \begin{equation}
        \lambda_j(q,W_j) = \{G\in\kappa_j^*:\|[q+G_k+G_l]_j-\widetilde{K}\|_2<W_j\}\times\mathcal{A}_j
    \end{equation}
    where $q\in B_{W_j}(K_j)\cup B_{W_j}(K_j')$, $\widetilde{K}$ is the $K$-point closest to $q$, and $[\cdot]_j:\mathbb{R}^2\to\Gamma_j^*$ is defined by
    \begin{equation}
        [x]_j = x-B_j\lfloor B_j^{-1}x\rfloor,
    \end{equation}
    selects the initial set of indices whose shifted momenta fall within a radius $W_j$ of the Dirac point for layer $j$. The \textit{global connection} function $\Lambda_{\mu d}^*:\mathbb{R}^2\times\mathbb{R}_+^3\to\mathcal{P}(\Omega^*)$ is defined by
    \begin{gather}
        \Lambda_{\mu d}^*(q, W) = \bigcup_{j=1}^3\Lambda_{\mu_{j}d_{j}}^*(q,W_j)\quad\text{where}\quad \Lambda_{\mu_{j}d_{j}}^*(q,W_j) = \mathcal{C}_0^{\mu_j}(\lambda_j(q,W_j), d_j)
    \end{gather}
    is the set of all index pairs that are $(\mu_j,d_j)$-connected relative to the $0$ index.
\end{definition}

We define the following metric for use in this setting.
\begin{definition}
    Define the metric $\mathfrak{d}_{jj}:[\mathcal{W}_{j}^*(q, W_j)]^2\to\mathbb{R}$ by
    \begin{equation}
        \mathfrak{d}_{jj}\bigl((G,j\alpha), (G',j\beta)\bigr) = \|(G_k\oplus G_l)-(G_k'\oplus G_l')\|_2.
    \end{equation}
\end{definition}

Let $S = \{\pm e_1, \pm e_2, \pm(e_1+e_2)\}$ where $e_j$ is the $j$th standard ordered $2$D basis vector. For $j\in\{1,2,3\}$, define the set of admissible lattice steps $\mathcal{S}_j = B_kS\cup B_lS$ and the set of  admissible $\mathbb{R}^2$ steps $P_j = (B_k-B_j)S\cup (B_l-B_j)S$. Consider the two-vectors in $P_j$ ordered by their polar angle counter-clockwise, beginning at $0$. Then we define $\vartheta_j\in[0,\pi)$ to be the maximum angular separation between adjacent vectors. Additionally, define $r_j= \max_{p\in P_j}\|p\|_2$, that is, the maximum distance for a single allowed step in $\mathbb{R}^2$. We obtain the following results on connectedness as it relates to $\mathcal{W}^*$ and $\Lambda_{\mu d}^*$.
\begin{theorem}
     \label{thm:layerconnect}
    For $q\in B_{W_j}(K_j)\cup B_{W_j}(K_j')$, if 
    \begin{gather}
        \frac{1}{2}\mathfrak{u}_j\mu_j\ll\delta_{W_j}<\frac{1}{2}\|K_j-K_j'\|_2-W_j,\quad \mu_j\geq\max_{s\in\mathcal{S}_j}\|B_js\|_2,\\ 
        \mathfrak{u}_j = \|[(B_k-B_j)B_k^{-1}, (B_l-B_j)B_l^{-1}]\|_2,\quad \frac{r_j}{2}\sec\frac{\vartheta_j}{2}<W_j
    \end{gather}
    then $\mathcal{W}_j^*(q, W_j)$ is $(\mu_j,\mathfrak{d}_{jj})$-connected relative to $0$ and $\mathcal{W}_j^*(q, W_j) = \Lambda_{\mu_jd_j}^*(q,W_j)$.
\end{theorem}

By taking the appropriate minimums and maximums of the layer parameters in \zcref{thm:layerconnect}, one immediately obtains the following corollary.
\begin{corollary}
    \label{cor:connect}
    For $q\in\bigcup_{j=1}^3 B_{W}(K_j)\cup B_W(K_j')$, if
    \begin{gather}
        \frac{1}{2}\mathfrak{u}\mu\ll \delta_W<\frac{1}{2}\min_{j\in\{1,2,3\}}\|K_j-K_j'\|_2-W,\quad\mu\geq\max_{j\in\{1,2,3\}}\max_{s\in\mathcal{S_j}}\|B_js\|_2,\\
        \mathfrak{u} = \max_{j\in\{1,2,3\}}\|[(B_k-B_j)B_k^{-1}, (B_l-B_j)B_l^{-1}]\|_2,\quad \max_{j\in\{1,2,3\}}\frac{r_j}{2}\sec\frac{\vartheta_j}{2}<W,
    \end{gather}
    then $\mathcal{W}^*(q, W)$ is $(\mu, \mathfrak{d})$-connected relative to $0$ and $\mathcal{W}^*(q, W) = \Lambda_{\mu d}^*(q, W)$.
\end{corollary}

We now present the lemma on the relative error introduced by $W$-truncation.
\begin{lemma}
    \label{lem:errorW}
    Let $\Sigma\subset\mathbb{R}$ be a bounded energy window. Choose a base truncation radius $W_0$ such that
    \begin{equation}
         W_0\geq \frac{1}{v_f}\left(\sup_{E\in\Sigma}|E|+\frac{8+3\alpha}{2}\left\|\widehat{H}^{\tau}_{\mathrm{inter}}\right\|_2+2\epsilon\right)
    \end{equation}
    where $v_f$ is the Fermi velocity and $\alpha\in\mathbb{R}_+$. For double-incommensurate TTG, the following bound holds
    \begin{gather}
        |\widehat{\mathcal{D}}_{\epsilon}^{(N,\tau)}(E)-\widehat{\mathcal{D}}_{\epsilon}^{(N,\tau,W)}(E)|\lesssim\epsilon^{-3}e^{-\eta^2\epsilon^{-2}}
            +\epsilon^{-4}e^{-\ln{\frac{2+\alpha}{2}}\left\lceil\frac{1}{2}\left\lfloor\frac{W-W_0}{2\mathfrak{u}}+1\right\rfloor\right\rceil}\\
         +\epsilon^{-1}e^{-\frac{v^2\mathfrak{u}^2}{8\epsilon^2}\left\lfloor\frac{W-W_0}{2\mathfrak{u}}+2\right\rfloor^2}
    \end{gather}
    for all $E\in\Sigma$ where
    \begin{gather}
        \frac{1}{2}\mathfrak{u}\mu\ll \delta_W<\frac{1}{2}\min_{j\in\{1,2,3\}}\|K_j-K_j'\|_2-W,\quad\mu\geq\max_{j\in\{1,2,3\}}\max_{s\in\mathcal{S}_j}\|s\|_2,\\
        \mathfrak{u} = \max_{j\in\{1,2,3\}}\|[(B_k-B_j)B_k^{-1}, (B_l-B_j)B_l^{-1}]\|_2,\quad \max_{j\in\{1,2,3\}}\frac{r_j}{2}\sec\frac{\vartheta_j}{2}<W,\\
        \eta = (2+\alpha)\|\widehat{H}_{\mathrm{inter}}^{\tau}\|_2.
    \end{gather}
\end{lemma}

\subsection{$L$-truncation: truncation of scattering channels near $K$-point}

Let $L\in\mathbb{R}_+^3$ be a vector of $L$-truncation parameters. The layer $L$-truncation function $\mathcal{L}_j^*:\mathbb{R}_+\to\mathcal{P}(\Omega_j^*)$ restricts the reciprocal lattice vector pairs of layer $j$ based on the maximum norm, that is,
\begin{equation}
    \mathcal{L}^*_j(L_j) = \{G\in\kappa_j^*:\max(\|G_k\|_2, \|G_l\|_2)<L_j\}\times\mathcal{A}_j.
\end{equation}
The total $L$-truncation function $\mathcal{L}^*:\mathbb{R}_+^3\to\mathcal{P}(\Omega^*)$ is defined as the union of the layer $L$-truncation functions, that is,
\begin{equation}
    \mathcal{L}^*(L) = \bigcup_{j=1}^3\mathcal{L}^*_j(L_j).
\end{equation}
For each $L\in\mathbb{R}_3^+$, let $J_L:\ell^2(\mathcal{L}^*(L))\to\ell^2(\Omega^*)$ be the operator that extends a state by zero outside of $\mathcal{L}^*(L)$. We define the $L$-projection operator $P_L\in\mathcal{L}(\ell^2(\Omega^*))$ as the orthogonal projection:
\begin{equation}
    P_L = J_LJ_L^*.
\end{equation}
Specifically, for any state $\psi\in\ell^2(\Omega^*)$, the action of $P_L$ is given by
\begin{equation}
    (P_L\psi)_x =
    \begin{cases}
        \psi_x, & x\in\mathcal{L}^*(L),\\
        0, & x\in\Omega^*\setminus\mathcal{L}^*(L).
    \end{cases}
\end{equation}
The $L$-truncation of the $(\tau,W)$-truncated fiber reciprocal Hamiltonian $\widehat{H}_q^{(\tau,W,L)}$ is defined by
\begin{equation}
    \widehat{H}_q^{(\tau,W,L)} = P_L\widehat{H}_q^{(\tau,W)}P_L\quad\text{for all}\quad q\in\mathbb{R}^2
\end{equation}
and the associated truncated regularized density of states $\widehat{\mathcal{D}}_\epsilon^{(N,\tau,W,L)}$ is defined by
\begin{equation}
    \widehat{\mathcal{D}}_\epsilon^{(N,\tau,W,L)}(E) = \mathcal{T}_{N}^W\delta_\epsilon(E-\widehat{H}_q^{(\tau,W,L)})\quad\text{for all}\quad E\in\mathbb{R}.
\end{equation}
The following figure illustrates how $L$ changes the density of the reciprocal degrees of freedom.

\begin{figure}[ht]
    \centering
    \begin{subfigure}[b]{0.32\textwidth}
        \centering
        \includegraphics[width = \textwidth, alt = {The $WL$-truncated degrees of freedom centered at $\widetilde{K}$ for $L=30$ at $\theta = [-1.4,0,2.8]$.}]{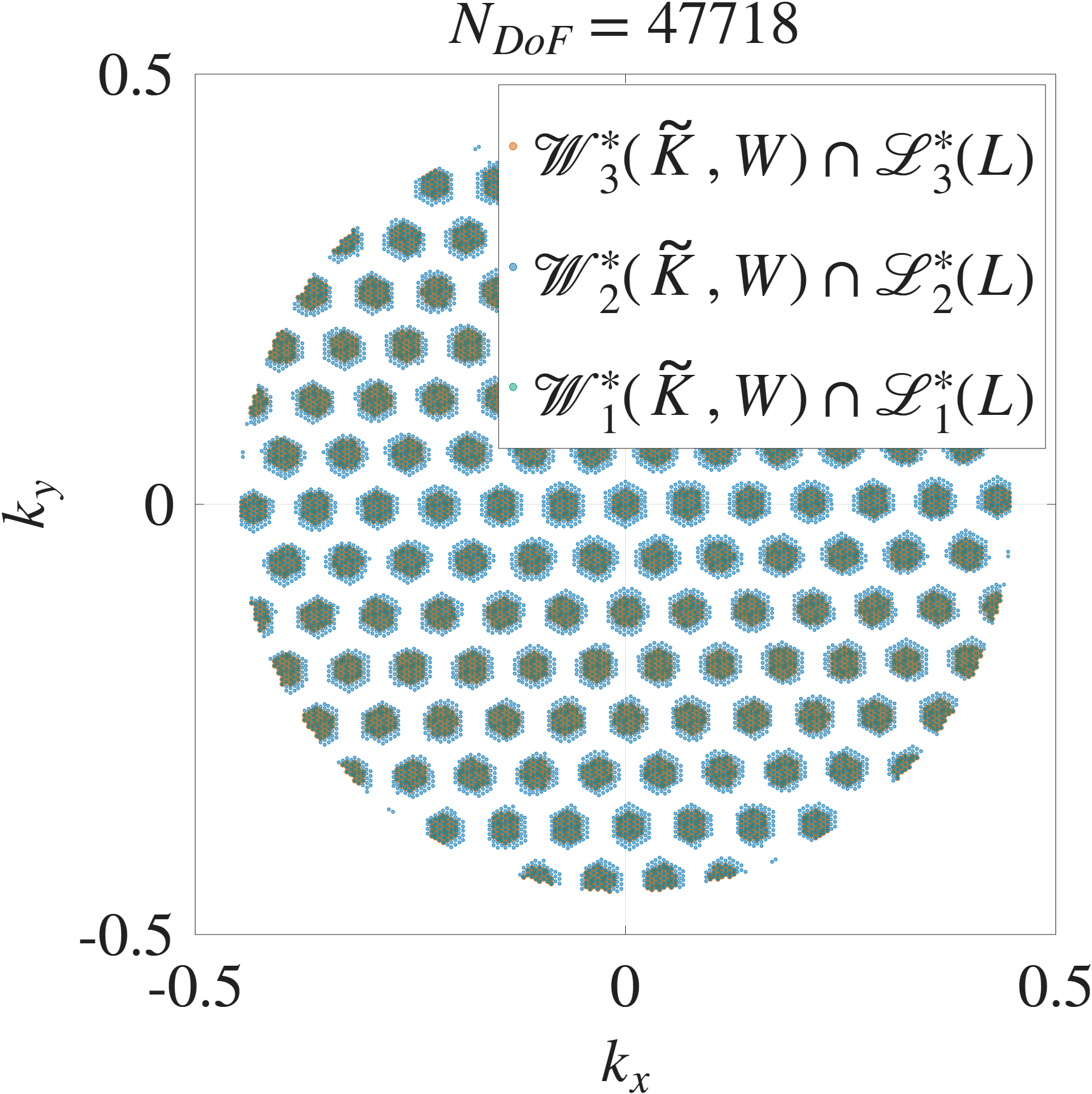}
        \caption{$L = 30$}
    \end{subfigure}
    \hfill
    \begin{subfigure}[b]{0.32\textwidth}
        \centering
        \includegraphics[width = \textwidth, alt = {The $WL$-truncated degrees of freedom centered at $\widetilde{K}$ for $L=60$ at $\theta = [-1.4,0,2.8]$.}]{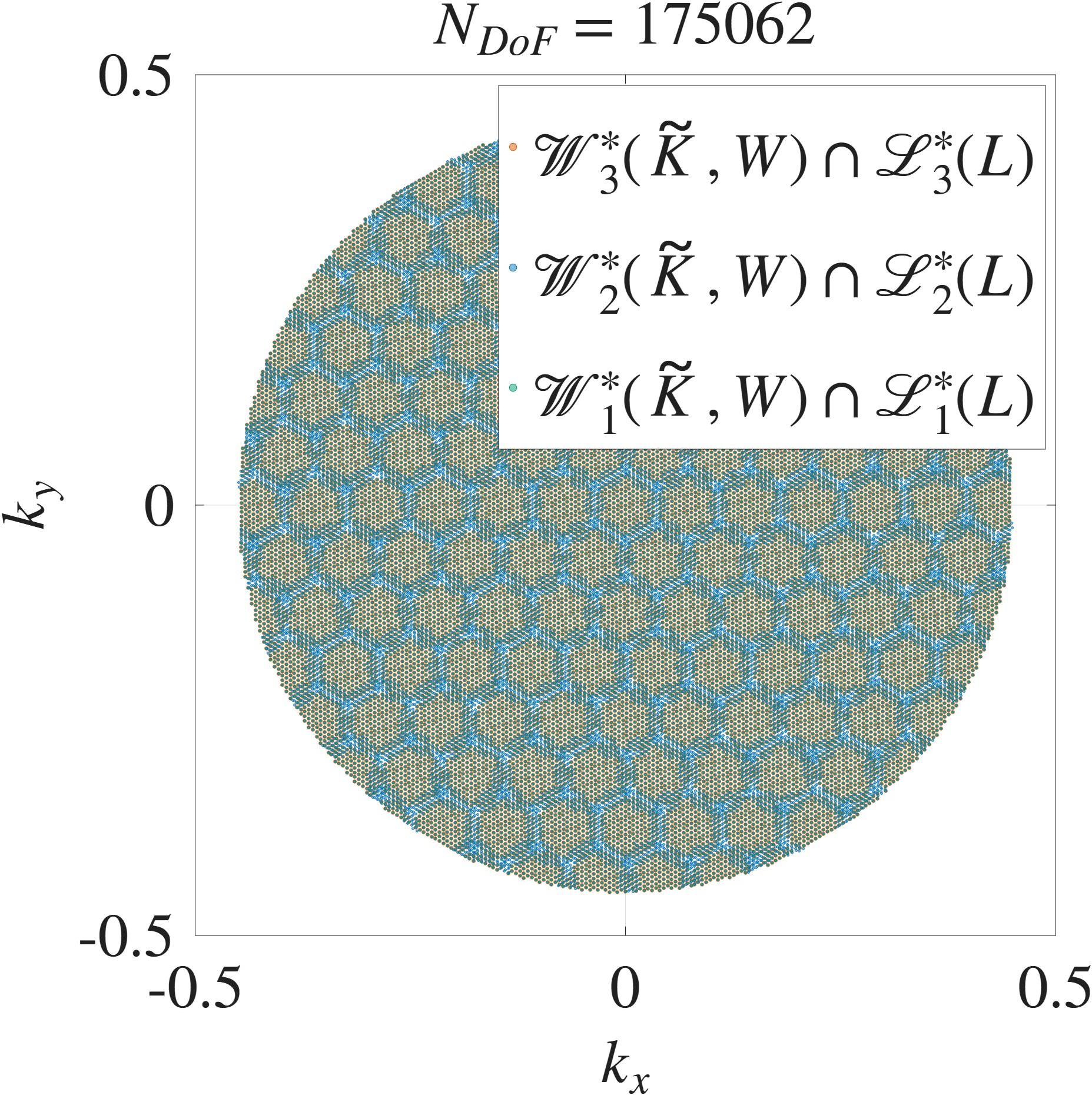}
        \caption{$L = 60$}
    \end{subfigure}
    \hfill
    \begin{subfigure}[b]{0.32\textwidth}
        \centering
        \includegraphics[width = \textwidth, alt = {The $WL$-truncated degrees of freedom centered at $\widetilde{K}$ for $L=120$ at $\theta = [-1.4,0,2.8]$.}]{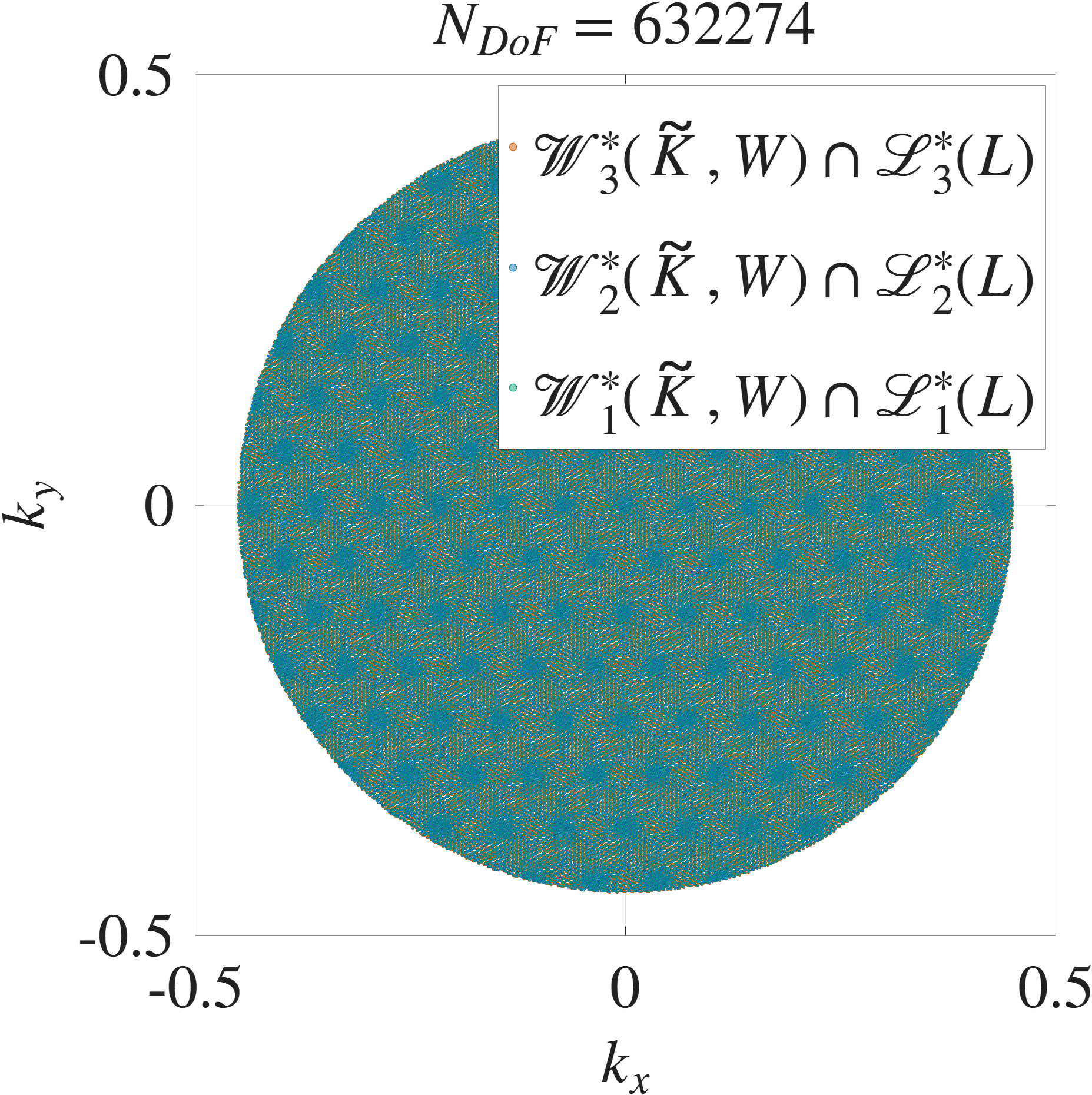}
        \caption{$L = 120$}
    \end{subfigure}
    \caption[Increasing $L$ increases the density of points in the $W$-ball]{For fixed $W = 0.45$ and $\theta = [-1.4, 0, 2.8]$, we demonstrate the filling of the moir\'{e} mapped $\mathcal{W}_j^*(q,W)$, that is, $(B_k-B_j)B_k^{-1}G_k+(B_l-B_j)B_l^{-1}G_l$ for all $G\in \mathcal{W}^*(q,W)\cap\mathcal{L}^*(L)$.}
    \label{fig:density}
\end{figure}

With regard to the relative error introduced by the $L$-truncation, we obtain the following lemma.
\begin{lemma}
     \label{lem:errorL}
    For double-incommensurate TTG, the following bound holds
    \begin{equation}
        |\widehat{\mathcal{D}}_\epsilon^{(N,\tau,W)}(E)-\widehat{\mathcal{D}}_\epsilon^{(N,\tau,W,L)}(E)|\lesssim \epsilon^{-3}\gamma_\epsilon^{-1}Le^{-\gamma_\epsilon L}
    \end{equation}
    for all $E\in\mathbb{R}$ where 
    \begin{equation}
        \gamma_\epsilon=\frac{C'\epsilon}{\ln(\epsilon^{-1})}.
    \end{equation}
    for some $C'\in\mathbb{R}_+$.
\end{lemma}

\subsection{Kernel Polynomial Method for momentum LDoS}
While we have a finite matrix $\widehat{H}_q^{(\tau, W, L)}$, the number of reciprocal degrees of freedom scales as $O(L^2W^2)$, so that for converged $L$, the size of the matrix prohibits a direct eigensolve. Instead, we utilize the Kernel Polynomial Method (KPM) \cite{weise_kernel_2006}. Let $P$ be a prescribed Chebyshev polynomial order. The KPM approximation of a delta function is given by
\begin{equation}
    \delta_{\mathrm{KPM}}^P(E, \widehat{H}_q^{(\tau, W, L)}) = \frac{1}{\pi\sqrt{1-(sE)^2}}\sum_{n=0}^{P-1}g_n^P[T_n(s\widehat{H}_q^{(\tau, W, L)})]T_n(sE)
\end{equation}
where $T_n$ is the $n$th Chebyshev polynomial defined by the recursion
\begin{align}
    T_0(x) &= 1,\\
    T_1(x) &= x,\\
    T_{n+1}(x) &= 2xT_n(x)-T_{n-1}(x)
\end{align}
and $s\in\mathbb{R}_+$ such that $\sigma\left(s\widehat{H}_q^{(\tau, W, L)}\right)\subset(-1,1)$ and $sE\in(-1,1)$. In particular, we choose $s = \|\widehat{H}_{\widetilde{K}}^{(\tau, W, L)}\|_\infty^{-1}$. The coefficient $g_n^P$ is the $n$th Jackson coefficient defined by
\begin{equation}
    g_n^P = (2-\delta_{n0})\frac{(P-n+1)\cos{\frac{\pi n}{P+1}}+\sin{\frac{\pi n}{P+1}}\cot{\frac{\pi}{P+1}}}{P+1}.
\end{equation}
In particular, we have that
\begin{equation}
    [\delta_{\mathrm{KPM}}^P(E,\widehat{H}_q^{(\tau, W, L)})]_{0\alpha}^{0\alpha}\approx[\delta_\epsilon(E-\widehat{H}_q^{(\tau, W, L)})]_{0\alpha}^{0\alpha}\quad\text{where}\quad \epsilon\approx\frac{\pi}{sP}.
\end{equation}
In a more recent paper \cite{yi_high-order_2025}, a high-order regularized delta-Chebyshev method for computing spectral densities is proposed. In particular, they show that the Jackson kernel KPM is $O(P^{-2})$ if the density of states is twice continuously differentiable, and their regularized kernel is $O(P^{-m})$.

With regard to the KPM, the final computables are given by 
\begin{equation}
    \widehat{\mathcal{D}}_{P}^{(\tau, W,L)}(q,E) = \sum_{j=1}^3\frac{1}{|\mathcal{A}_j|}\sum_{\alpha\in\mathcal{A}_j}[\delta_{\mathrm{KPM}}^P(E,\widehat{H}_q^{(\tau,W,L)})]_{0\alpha}^{0\alpha}
\end{equation}
and
\begin{equation}
    \widehat{\mathcal{D}}_P^{(\tau, W, L, N)}(E) = \nu^*\sum_{j=1}^3\frac{|\Gamma_{m(j)}^*|}{N^2}\sum_{\widetilde{K}\in\{K_j,K_j'\}}\sum_{\substack{q\in\mathfrak{S}_j(N,\widetilde{K})\\G\in\mathcal{R}_{m(j)}^*\cap B_W(0)}}[\delta_{KPM}^P(E,\widehat{H}_q^{(\tau, W, L)})]_{G\alpha}^{G\alpha}
\end{equation}
where
\begin{equation}
    m(j) = 
    \begin{cases}
        12, & j = 1,\\
        12\;\operatorname{OR}\;23, & j=2,\\
        23, & j = 3.
    \end{cases}
\end{equation}
and
\begin{equation}
    \mathfrak{S}_j(N,\widetilde{K}) = \widetilde{K}+N^{-1}B_{(m_j)}[0, N-1)^2.
\end{equation}
For the total DoS discretization scheme, we are exploiting the symmetry for $G\alpha \in \Omega_j^*$ for $G_l=G_k = 0$,
\begin{equation}
    f(\widehat{H}_q^{(\tau,W,L)})]_{G\alpha}^{G\alpha} = [f(\widehat{H}^{(\tau,W,L)}_{q+G_{j k}})]_{0\alpha}^{0\alpha}
\end{equation}
where $G_{j k} = (B_j-B_k)n$ for $|k - j| = 1$ and $B_jn = G_j$.

We now provide an abbreviated algorithmic overview of the process defined throughout this section:

\noindent\makebox[\textwidth][c]{
\begin{minipage}{0.75\textwidth}
    \begin{algorithm}[H]
    \caption{Computing $\widehat{\mathcal{D}}_{P}^{(\tau,W,L, N)}(E)$}
    \label{alg:dos}
        \begin{algorithmic}[1]
            \State Get geometry for angle triple $(\theta_1,\theta_2,\theta_3)$
            \State Construct degrees of freedom based on $WL$-truncation
            \State Build a Hamiltonian constructor function
            \State Construct the uniform discretization $\mathfrak{S}(N,\widetilde{K})$
            \State Compute $\delta_{\mathrm{KPM}}^P(E, \widehat{H}^{(\tau,W,L)}_q)$ for $q\in\mathfrak{S}(N,\widetilde{K})$.
        \end{algorithmic}
    \end{algorithm}
\end{minipage}
}
\vspace{1em}

\subsection{Full approximation result for the density of states}

Lastly, we present the full accumulated relative error, neglecting the relative error associated with the KPM. 
\begin{theorem}
    \label{thm:cumulative}
     Let $\Sigma\subset\mathbb{R}$ be a bounded energy window. Choose a base truncation radius $W_0$ such that
    \begin{equation}
         W_0\geq \frac{1}{v_f}\left(\sup_{E\in\Sigma}|E|+\frac{8+3\alpha}{2}\left\|\widehat{H}^{\tau}_{\mathrm{inter}}\right\|_2+2\epsilon\right)
    \end{equation}
    where $v_f$ is the Fermi velocity and $\alpha\in\mathbb{R}_+$. For double-incommensurate twisted trilayer graphene, the following accumulated relative error bound holds
    \begin{align}
        |\widehat{\mathcal{D}}_\epsilon(E)-\widehat{\mathcal{D}}_\epsilon^{(\tau,W, L,N)}(E)|&\lesssim \epsilon^{-1}e^{-\widetilde{\rho} N}+\epsilon^{-1}e^{-(\tau-\delta)\min(\gamma_1,\gamma_2)}+\epsilon^{-3}e^{-\eta^2\epsilon^{-2}}\\
        &+\epsilon^{-4}e^{-\ln{\frac{2+\alpha}{2}}\left\lceil\frac{1}{2}\left\lfloor\frac{W-W_0}{2\mathfrak{u}}+1\right\rfloor\right\rceil}+\epsilon^{-1}e^{-\frac{v^2\mathfrak{u}^2}{8\epsilon^2}\left\lfloor\frac{W-W_0}{2\mathfrak{u}}+2\right\rfloor^2}\\
        &+\epsilon^{-3}\gamma_\epsilon^{-1}Le^{-\gamma_\epsilon L}.
    \end{align}
    for all $E\in\Sigma$ where
    \begin{gather}
        \frac{1}{2}\mathfrak{u}\mu\ll \delta_W<\frac{1}{2}\min_{j\in\{1,2,3\}}\|K_j-K_j'\|_2-W,\quad\mu\geq\max_{j\in\{1,2,3\}}\max_{s\in\mathcal{S}_j}\|s\|_2,\\
        \mathfrak{u} = \max_{j\in\{1,2,3\}}\|[(B_k-B_j)B_k^{-1}, (B_l-B_j)B_l^{-1}]\|_2,\quad \max_{j\in\{1,2,3\}}\frac{r_j}{2}\sec\frac{\vartheta_j}{2}<W,\\
        \eta = (2+\alpha)\|\widehat{H}_{\mathrm{inter}}^{\tau}\|_2,\quad \gamma_\epsilon = D\frac{\epsilon}{\ln(\epsilon^{-1})},\\
        \widetilde{\rho} \leq \min\left(\frac{\gamma_1}{2},\frac{\epsilon}{C}\right)\min_{j\in\{1,2,3\}}\|A_j^{-T}\|_2^{-1},
    \end{gather}
    for some $C, D\in R_+$ (the constants $r_j$, $\vartheta_j$, and the set $\mathcal{S}_j$ are defined in \zcref{thm:layerconnect}).
\end{theorem}

\section{Numerics}
\label{sec:numerics}
In this section, we present converged results for TTG and numerically validate the above convergence rates.
We use the ab initio tight-binding model constructed from DFT simulations as introduced in \cite{fang_electronic_2016} and monolayer terms from \cite{jung_tight-binding_2013}.
 For details on the tight-binding model and how we simplify computations of the interlayer tunneling function Fourier transform, see \ref{app:hop}. We will compare the results of our algorithm to the continuum model from \cite{Zoe2020}. For details on how to relate the two models, see \ref{app:bm}.

 We implement the continuum model alongside our tight-binding model using the same lattice geometry and $WL$-truncation scheme in order to illustrate the importance of the higher order effects inherent in the full momentum space model that are lost by the continuum approximation. We apply a unitary transformation to the momentum space transformed tight-binding model such that under appropriate Taylor expansions of the trilayer BM model, we denote in reciprocal space as $\widehat{H}_q^\text{BM}$, found in \cite{Zoe2020}. In particular, we define the unitary operator $U\in\mathcal{L}(\ell^2(\Omega^*))$ by
\begin{equation}
    U = \bigoplus_{(G,j\alpha)\in\Omega^*}\begin{bmatrix}
        \prod_{t=1}^3e^{iG_t\cdot\tau_{t\alpha}} & 0\\
        0 & \prod_{t=1}^3e^{iG_t\cdot\tau_{t\beta}}
    \end{bmatrix},
\end{equation}
and note that after confinement near a $K$ point,
\begin{equation}
    [\widehat{H}_q^\text{BM}]^W \approx   [U\widehat{H}_qU^*]^W.
\end{equation}
The entries of the RHS are
\begin{equation}
    [U\widehat{H}_qU^*]_{(G',k\beta)}^{(G,j\alpha)} = \prod_{t=1}^3e^{iG_t\cdot\tau_{t\alpha}}\prod_{s=1}^3e^{-iG_s'\cdot\tau_{s\beta}}[\widehat{H}_q]_{(G',k\beta)}^{(G,j\alpha)}
\end{equation}
for all $(G,j\alpha),(G',k\beta)\in\Omega^*$.
\begin{figure}[ht]
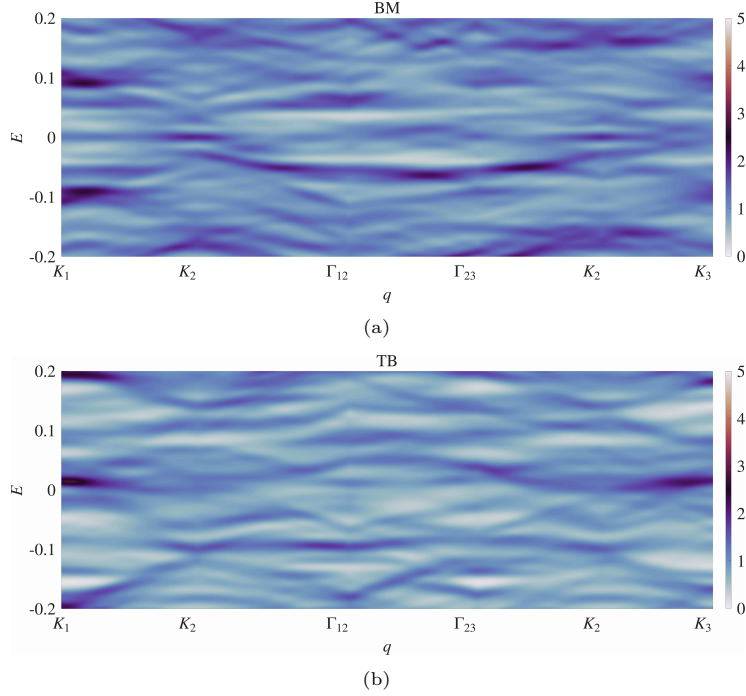

    \centering   
    \begin{subfigure}{0.8\textwidth} 
        \centering
        \includegraphics[width=\textwidth, alt = {The local density of states for the continuum approximation at $\theta = \left[-\frac{\sqrt{2}}{2}, 0, \frac{\sqrt{2}}{2}\right]$.}]{lDoSBMA}
        \caption{}
        \label{fig:lDoSBMA}
    \end{subfigure}
    
    \vspace{0.5em} 
    
    \begin{subfigure}{0.8\textwidth}
        \centering
        \includegraphics[width=\textwidth, alt = {The local density of states for the tight-binding model at $\theta = \left[-\frac{\sqrt{2}}{2}, 0, \frac{\sqrt{2}}{2}\right]$.}]{lDoSTBA}
        \caption{}
        \label{fig:lDoSTBA}
    \end{subfigure}
    \caption[For $\theta = \left\lbrack-\frac{\sqrt{3}}{2}, 0, \frac{\sqrt{2}}{2}\right\rbrack$ degrees, we compare the continuum and tight-binding models]{For $\theta = \left\lbrack-\frac{\sqrt{3}}{2}, 0, \frac{\sqrt{2}}{2}\right\rbrack$ degrees, $W = 0.3771$, $L = 30.5555$ and $P = 8000$, we compare the continuum (a) and tight-binding models (b). We observe that concentrations of the local density of states in the continuum model are lost in the tight-binding model.}
    \label{fig:lDoSA}
\end{figure}

We now present several momentum LDoS plots for both our tight-binding model and the continuum model (BM) proposed in \cite{Zoe2020} with all energy units given in eV. Each of these plots is along a line cut through the indicated high symmetry points. Each interval was sampled at $60$ points. Observe that substantial changes are visible. For example at the magic angles in \zcref{fig:lDoSC}, the flat-band profile at $E \approx 0$ changes dramatically. Also note that for the irrationally related angle pairs in \zcref{fig:lDoSA}, the momentum LDoS profile is dramatically different indicating the importance of the higher order effects inherent in the momentum space approach.
\begin{figure}[ht]
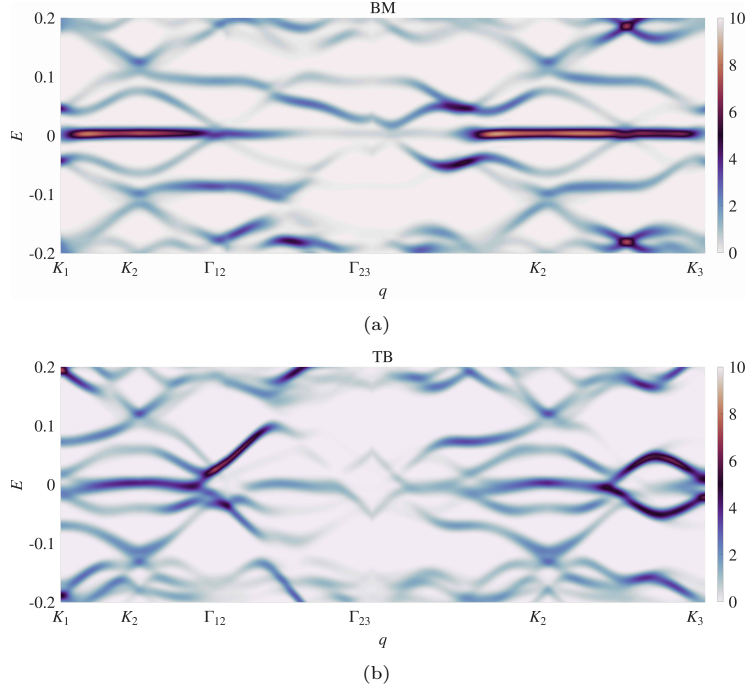

    \centering
    \begin{subfigure}{0.8\textwidth}
        \centering
        \includegraphics[width=\textwidth, alt = {The local density of states for the continuum approximation at $\theta = [-1.4, 0, 2.8]$.}]{lDoSBMB}
        \caption{}
        \label{fig:lDoSBMC}
    \end{subfigure}
    
    \vspace{0.5em}
    
    \begin{subfigure}{0.8\textwidth}
    \centering
        \includegraphics[width=\textwidth, alt = {The local density of states for the tight-binding model at $\theta = [-1.4, 0, 2.8]$.}]{lDoSTBB}
        \caption{}
        \label{fig:lDoSTBC}
    \end{subfigure}
    
    \caption[For $\theta = \left\lbrack-1.4, 0, 2.8\right\rbrack$ degrees, we compare the continuum and tight-binding models]{For $\theta = \left[-1.4, 0, 2.8\right]$ degrees, $W = 0.47422$, $L = 19.40479$, $P = 8000$, we compare the continuum (a) and tight-binding models (b). We observe that the mostly flat band near the Fermi energy in the continuum model, has significantly changed in the tight-binding model. While this may still be sufficient for a superconducting phase, it suggests that the flat band is not the core mechanism.}
    \label{fig:lDoSC}
\end{figure}  

We numerically validate the analytical error bounds for $W$, $L$, and $N$ as described in \zcref{thm:cumulative}. In \zcref{fig:ErrorWL}, the relative $W$-error for the LDoS is defined by
\begin{equation}
    \eta_{W_t}^P(q) = \frac{\max_{E\in\Sigma}|\widehat{\mathcal{D}}_{P}^{(\tau,W_t,L)}(q,E)-\widehat{\mathcal{D}}_{P}^{(\tau,W^*,L)}(q,E)|}{\max_{E\in\Sigma}|\widehat{\mathcal{D}}_{P}^{(\tau,W^*,L)}(q,E)|}
\end{equation}
where $(W_t)_{t=1}^{N_W}$ is a sequence of $W$-truncation parameters such that $W_{N_W} < W^*$, $W^*$ is the maximum $W$-truncation considered, and $\Sigma$ is the energy window of interest. We begin by fixing $\tau = 8.5$ and $\delta = 6.5$. Then $L$ is chosen such that
\begin{equation}
    L\propto\frac{W^*}{\min_{j\in\{1,2,3\}}(|\theta_j|)}.
\end{equation}
In particular, we choose the constant of proportionality to be $1$. Next, $W^*$ is chosen such that
\begin{equation}
    W^* = \frac{3}{8}\min_{j\in\{1,2,3\}}\|K_j-K_j'\|_2,
\end{equation}
which is well within the requirements for connectedness. Additionally, $W_1$ is chosen such that
\begin{equation}
    W_1 = \frac{5}{2}\max_{j\in\{1,2,3\}}\frac{r_j}{2}\sec\frac{\vartheta_j}{2}.
\end{equation}

\begin{figure}[ht]
     \centering
     \begin{subfigure}{0.45\textwidth}
         \centering
         \includegraphics[width=\textwidth, alt = {The relative $W$-error for $\theta = \left[-\frac{\sqrt{3}}{2}, 0, \frac{\sqrt{2}}{2}\right]$.}]{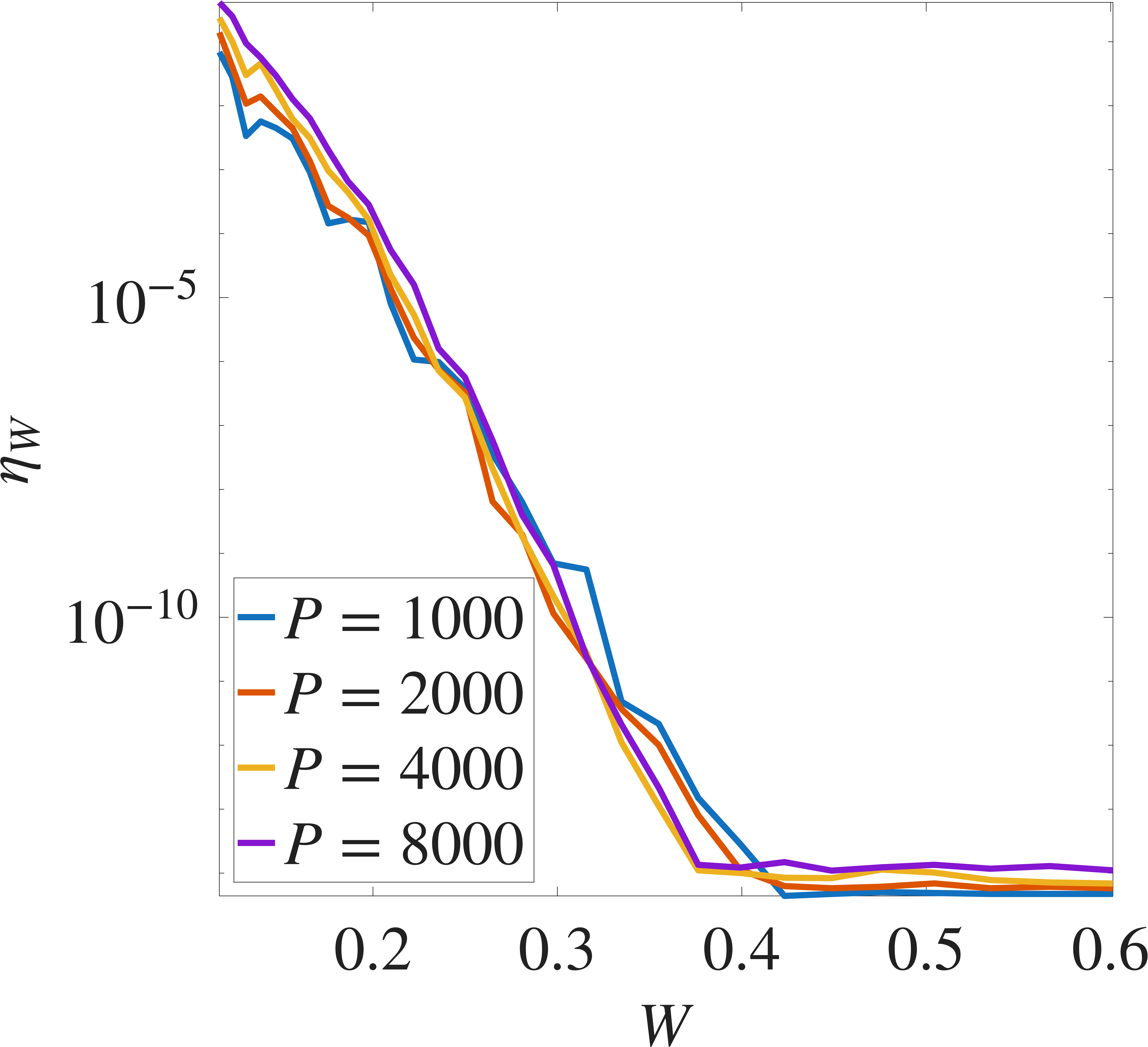}
         \label{fig:AErrorW}
         \caption{}
     \end{subfigure}
     \hfill
     \begin{subfigure}{0.45\textwidth}
         \centering
         \includegraphics[width=\textwidth, alt = {The relative $W$-error for $\theta = [-1.4,0,2.8]$.}]{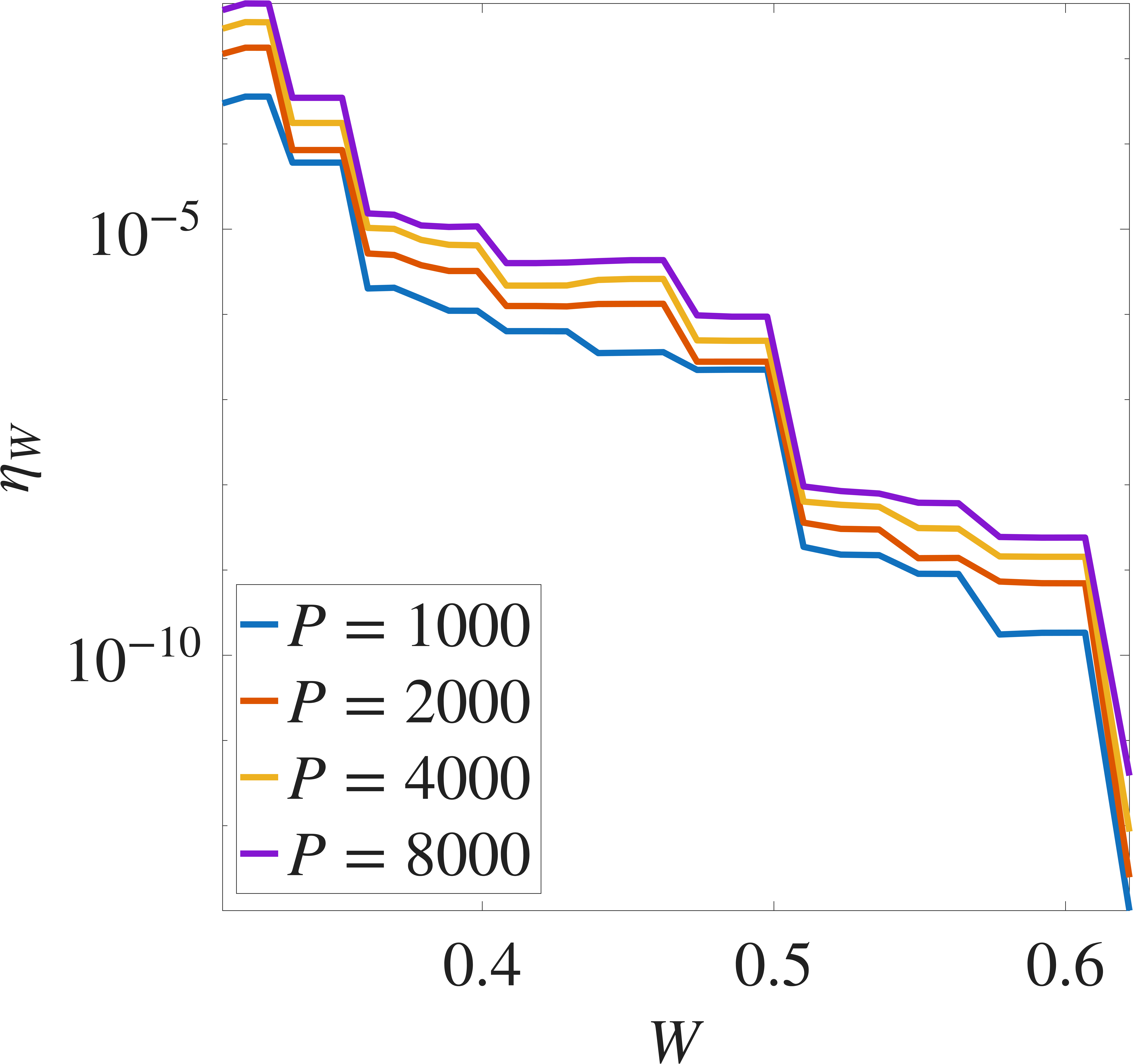}
         \label{fig:BErrorW}
         \caption{}
     \end{subfigure}

    \vspace{0.5em}
     
      \begin{subfigure}{0.45\textwidth}
         \centering
         \includegraphics[width=\textwidth,  alt = {The relative $L$-error for $\theta = [-\frac{\sqrt{3}}{2}, 0, \frac{\sqrt{2}}{2}]$.}]{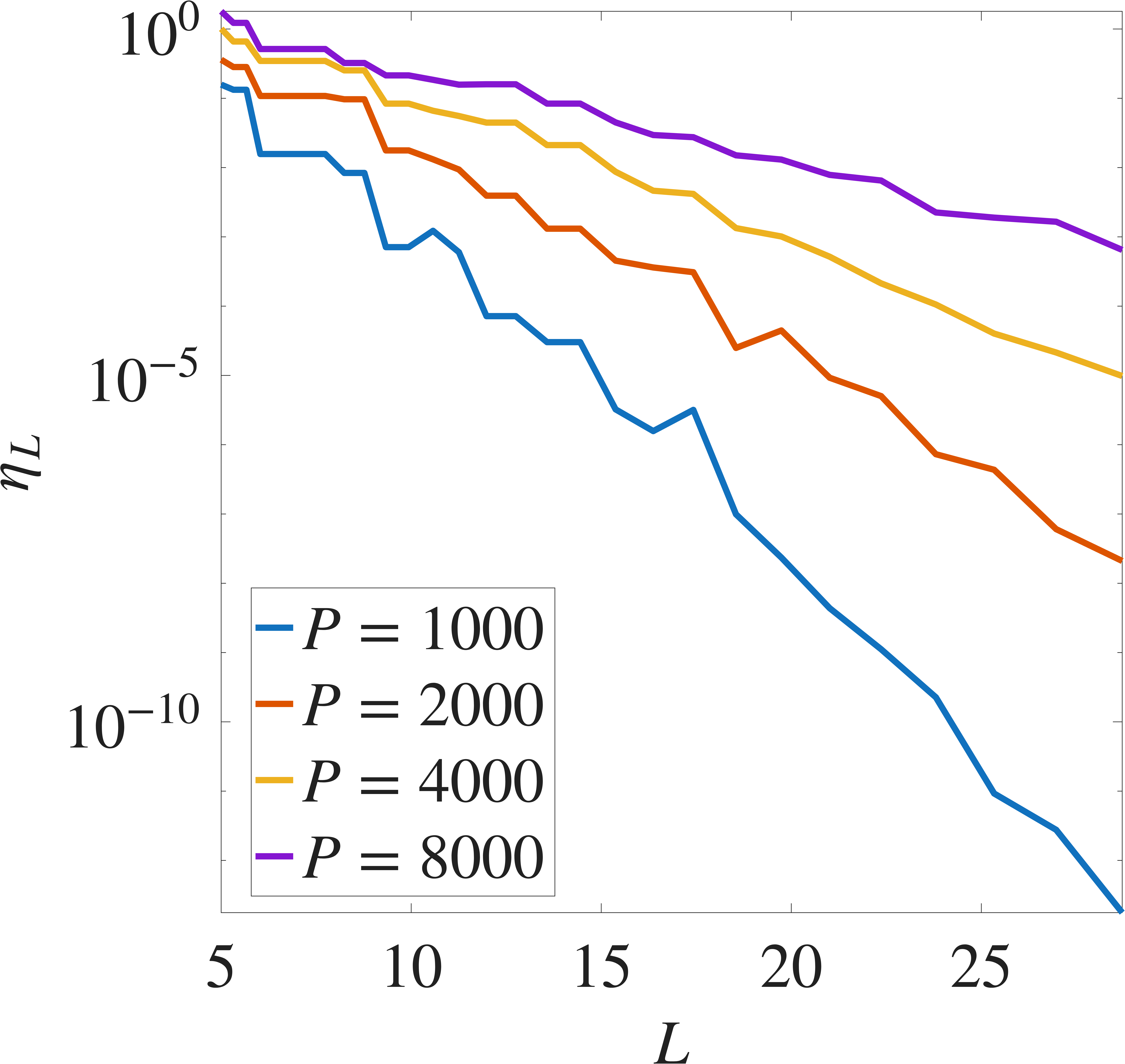}
        \label{Fig:AErrorL}
        \caption{}
     \end{subfigure}
     \hfill
     \begin{subfigure}{0.45\textwidth}
         \centering
         \includegraphics[width=\textwidth,  alt = {The relative $L$-error for $\theta = [-1.4,0,2.8]$.}]{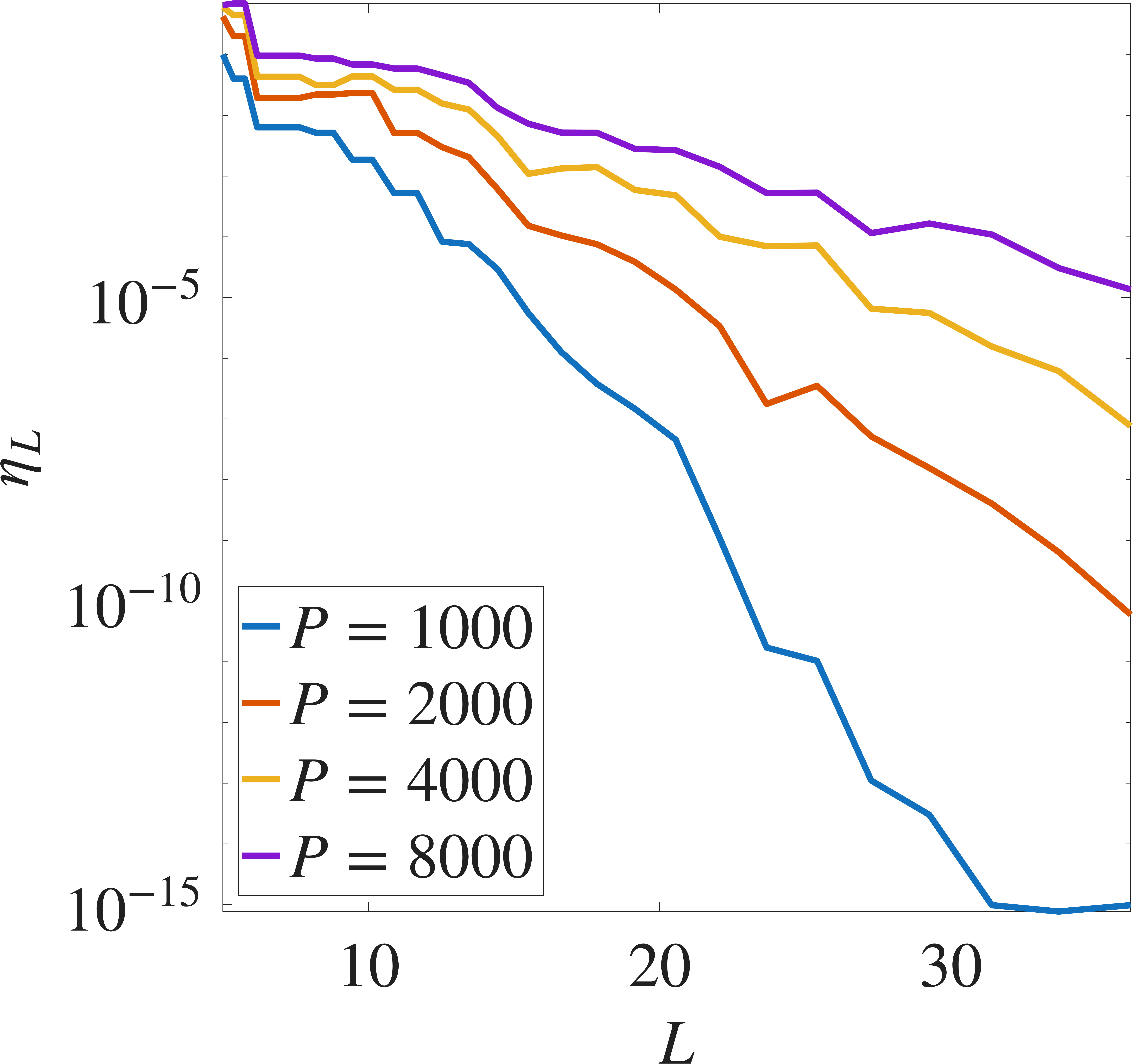}
         \label{fig:BErrorL}
         \caption{}
     \end{subfigure}

    \caption[Four semilogy plots of the relative $W$-error and relative $L$-error]{Semilogy of the relative $W$-error for a range of Chebyshev order $P$: (a) $\theta = \left\lbrack -\frac{\sqrt{3}}{2},0,\frac{\sqrt{2}}{2}\right\rbrack$ at $L = 35.2831$, (b) $\theta = [-1.4,0,2.8]$ at $L = 51.5415$. Semilogy of the relative $L$-error for a range of Chebyshev orders $P$: (c) $\theta = \left\lbrack -\frac{\sqrt{3}}{2},0,\frac{\sqrt{2}}{2}\right\rbrack$ at $W=0.3771$, (d) $\theta = [-1.4,0,2.8]$ at $W = 0.47422$. Exponential convergence is observed for both angles.}
 \label{fig:ErrorWL}
\end{figure}
We only run these simulations for the tight-binding model since the continuum model will converge faster under our scheme. We sample $30$ logarithmically spaced points between $W_1$ and $W^*$. To ease computational costs, we choose $q = K_G$ and restrict ourselves to Layer $1$'s center sites. The energy window is fixed at $\Sigma = [-0.2, 0.2]$, from which $4,000$ linearly spaced points are sampled, giving a resolution of $10^{-4}\unit{eV}$. We choose Chebyshev polynomial orders $P = 1,000\times 2^{(0:3)}$, making sure that the Gaussian width is not smaller than the energy resolution.

In \zcref{fig:ErrorWL}, the relative $L$-error for the LDoS is defined by
\begin{equation}
    \eta_{L_t}^P(q) = \frac{\max_{E\in\Sigma}|\widehat{\mathcal{D}}_{P}^{(\tau,W,L_t)}(q,E)-\widehat{\mathcal{D}}_{P}^{(\tau,W,L^*)}(q,E)|}{\max_{E\in\Sigma}|\widehat{\mathcal{D}}_{P}^{(\tau,W,L^*)}(q,E)|}
\end{equation}
where $(L_t)_{t=1}^{N_L}$ is a sequence of $L$-truncation parameters such that $L_{N_L} < L^*$, $L^*$ is the maximum $L$-truncation considered, and $\Sigma$ is the energy window of interest.
We choose $W$ such that
\begin{equation}
    W = \frac{W_1+W^*}{2}.
\end{equation}
We then choose $L^*$ such that
\begin{equation}
    L^*\propto\frac{W}{\min_{j\in\{1,2,3\}}(|\theta_j|)}.
\end{equation}

     
     

In \zcref{fig:convergence}, the relative $N$-error for the DoS is defined by
\begin{equation}
    \eta_{N_t}^P = \frac{\max_{E\in\Sigma}|\widehat{\mathcal{D}}_P^{(N_t,\tau,W,L)}(E)-\widehat{\mathcal{D}}_P^{(N^*,\tau,W,L)}(E)|}{\max_{E\in\Sigma}|\widehat{\mathcal{D}}_P^{(N^*,\tau,W,L)}(E)|}
\end{equation}
where $(N_t)_{t=1}^{N_N}$ is a sequence of discretization point counts such that $N_{N_N}<N^*$, $N^*$ is the maximum discretization point count considered and $\Sigma$ is the energy window of interest. 

Computing the DoS is expensive, even more so for a range of $q$-discretization. Consequently, we utilized the LSU high performance computing cluster Supermike-III. This is an 8-node cluster with 4 NVIDIA Ampere A100 GPUs ($40$GB) per node. For $\theta = [-1.4,0,2.8]$ and $\theta'=[-\frac{\sqrt{3}}{2}, 0, \frac{\sqrt{2}}{2}]$ with $N\in 3(1:10)$, we created a master list of jobs based on momenta. This list was batched so that each batch utilized $90\%$ of its assigned GPU.

\begin{figure}[ht]
    \centering
    \begin{subfigure}{0.45\textwidth}
        \centering
        \includegraphics[width=\textwidth]{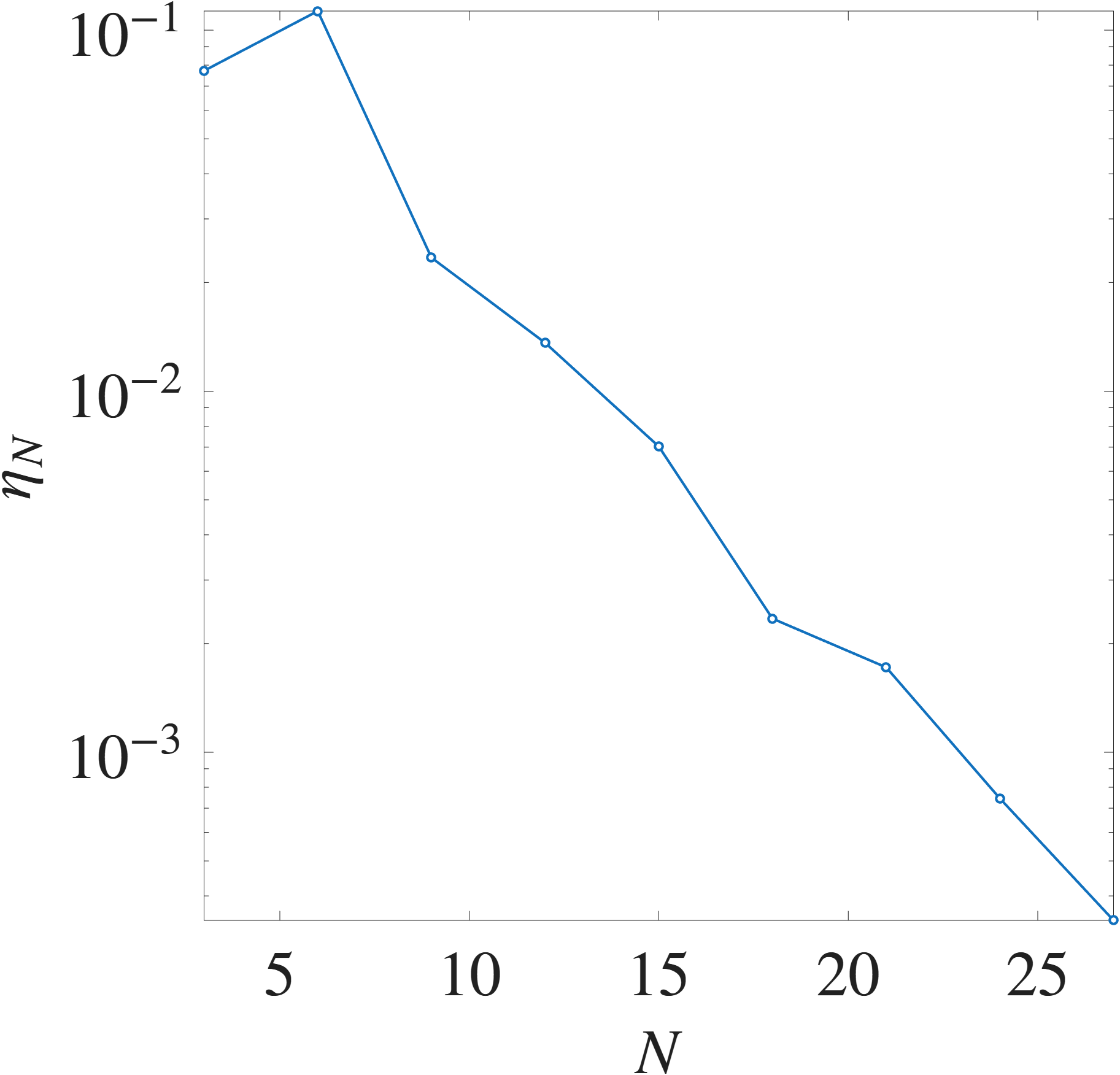}
        \caption{}
        \label{fig:AErrorN}
    \end{subfigure}
    \hfill
    \begin{subfigure}{0.45\textwidth}
        \centering
        \includegraphics[width=\textwidth]{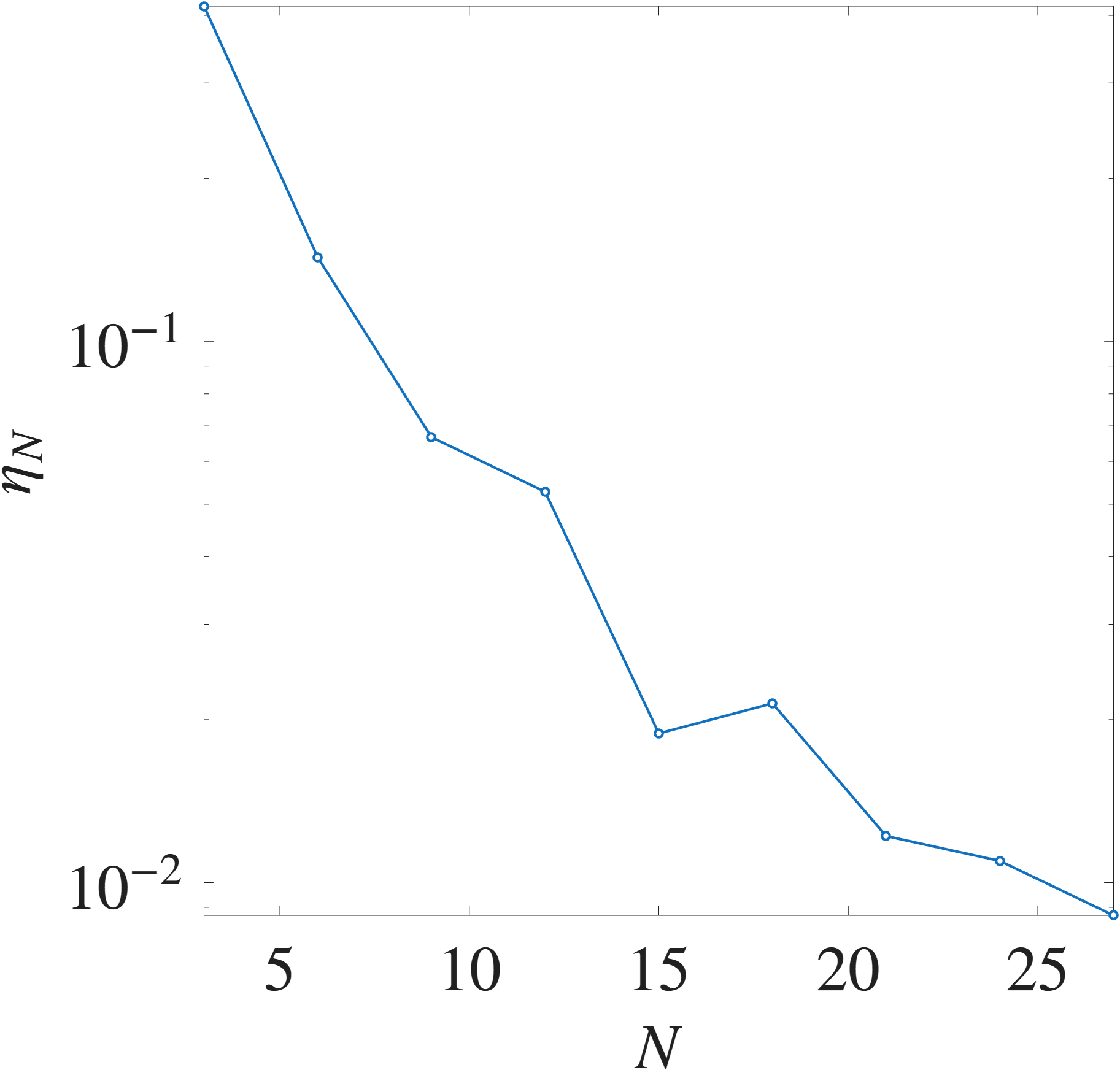}
        \caption{}
        \label{fig:BErrorN}
    \end{subfigure}
    \caption[Two semilogy plots of the relative $N$-error]{Two semilogy plots of the relative $N$-error for $P = 4000$: (a) $\theta = [-\frac{\sqrt{3}}{2},0,\frac{\sqrt{2}}{2}]$ at $W = 0.1743$ and $L = 73.7463$, (b) $\theta = [-1.4, 0, 2.8]$ at $W = 0.4324$ and $L = 73.7463$. Exponential convergence is observed for both angles.}
    \label{fig:convergence}
\end{figure}

In particular, we had $96$ batches processed in $6$ sets of $16$ (the per account limit), with each set completing in approximately $8$ hours, requiring a total of approximately $2.5$ days. In \zcref{fig:convergence}, we show discretization convergence results. In \zcref{fig:dos}, we present the converged total DoS in an energy range of $200$ meV around the Fermi level.

\begin{figure}[ht]
    \centering
    \begin{subfigure}{0.45\textwidth}
        \centering
        \includegraphics[width=\textwidth]{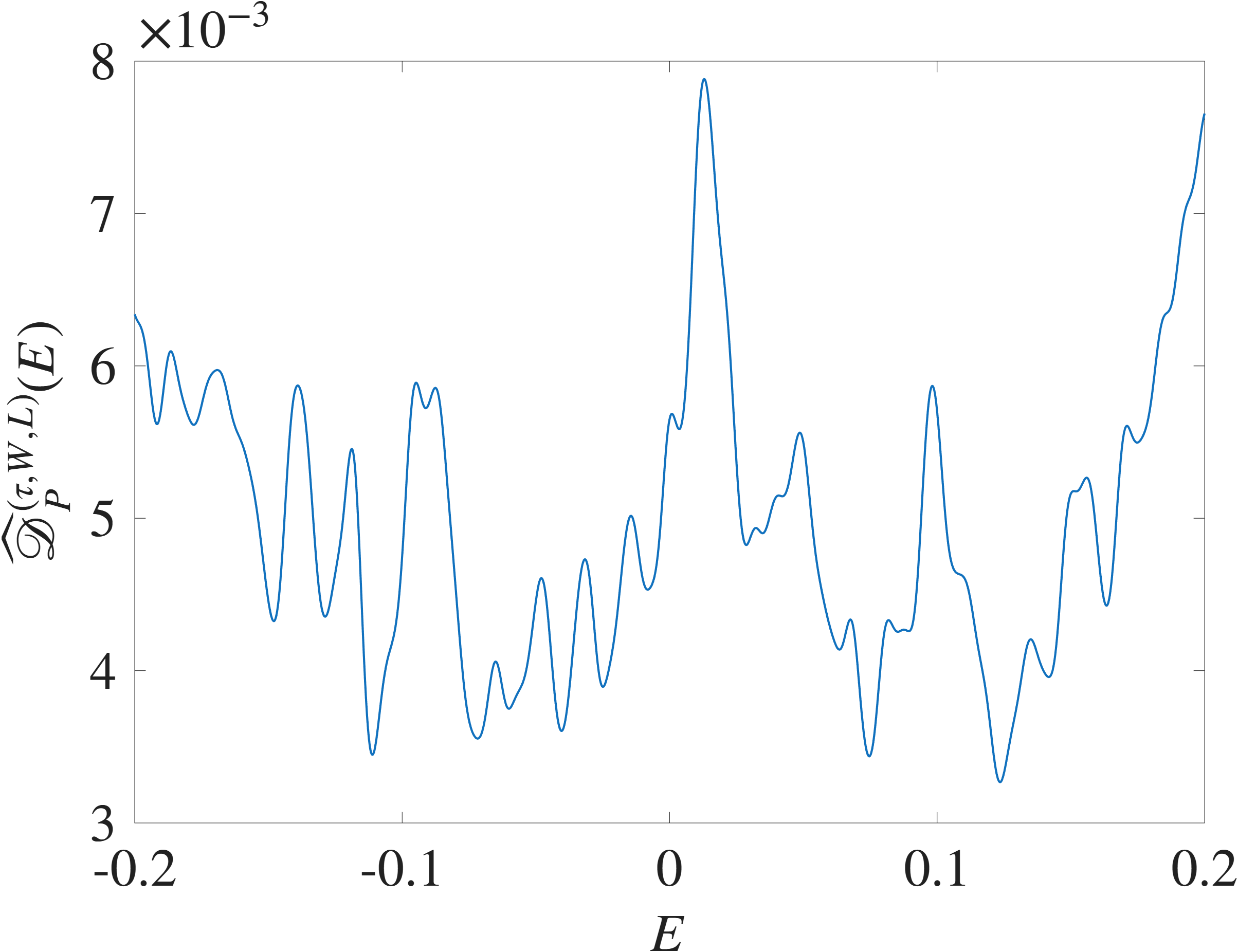}
        \caption{}
        \label{fig:DoSA}
    \end{subfigure}
    \hfill
    \begin{subfigure}{0.45\textwidth}
        \centering
        \includegraphics[width=\textwidth]{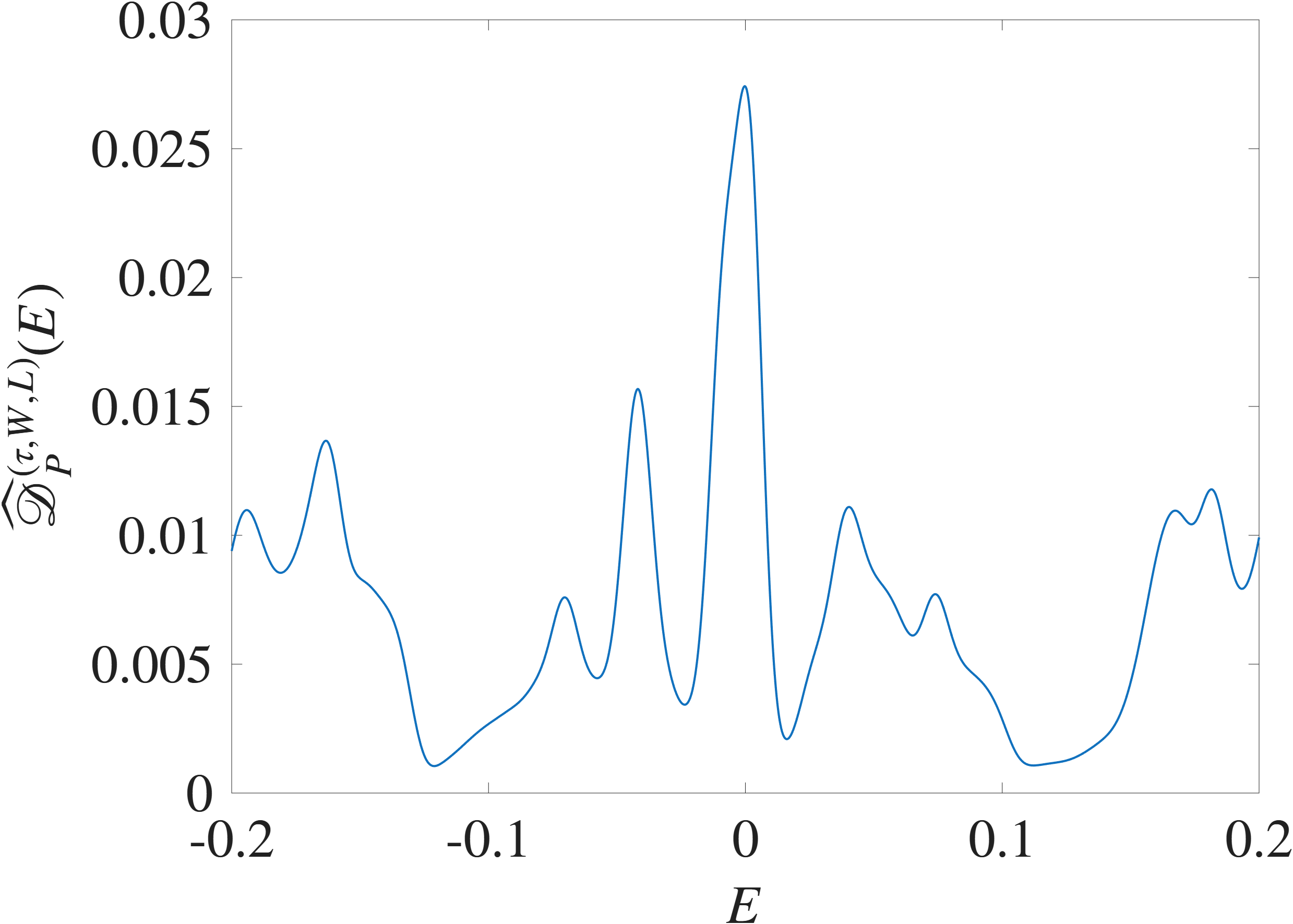}
        \caption{}
        \label{fig:DoSB}
    \end{subfigure}
    \caption[Two DoS plots]{Two DoS plots for $P = 4000$ and $N = 30$: (a) $\theta = [-\frac{\sqrt{3}}{2},0,\frac{\sqrt{2}}{2}]$ at $W = 0.1743$ and $L = 73.7463$, (b) $\theta = [-1.4, 0, 2.8]$ at $W = 0.4324$ and $L = 73.7463$.}
    \label{fig:dos}
\end{figure}

\section{Conclusion}
\label{sec:conclusion}

In conclusion, we presented an approach to transform double-incommensurate tight-binding models into momentum space via a Hamiltonian defined over a four-dimensional reciprocal lattice description describing the hopping through momentum {\em scattering channels.} We found explicit formulas for these hopping coefficients, which avoids errors associated with continuum model approximation. Secondly, we introduced careful truncation schemes by which we localize in momenta near the $K$ and $K'$ points ($W$-truncation), and then a second truncation which controls how many hoppings ($L$-truncation) can happen that stay within a neighborhood of the $K$-points. We derive explicit error bounds, which we numerically verify on an ab initio model of double-incommensurate TTG. We compute  converged momentum LDoS results using the KPM, as the matrices are sparse but too large for explicit diagonalization as can be one in the twisted bilayer case. We present converged momentum LDoS and total DoS results for an ab initio tight-binding model of twisted bilayer graphene, which we compute on an HPC.

A carefully converged single-particle model is essential for predicting electronic properties, predicting flatbands for correlated physics, and giving a roadmap to constructing an effective single-particle reduced order basis for the construction of many-body models. This algorithm gives explicitly a technique to generate converged single-particle results for TTG, but is easily generalized to other double-incommensurate trilayers as long as an appropriate momentum-energy confinement exists such as Dirac cones or parabolic bands with weak interlayer tunneling.

\section{Proofs}
\label{sec:proofs}

\subsection{Proof of \zcref{prop:momentum}}
\begin{proof}
    Let $\alpha\in\mathcal{A}_1$. Without loss of generality, suppose that \zcref{ass:herm} is satisfied. Then
    \begin{align}
        [\mathcal{G} H\psi]_{1\alpha}(q) &= \frac{1}{c_1^*}\sum_{R\in\mathcal{R}_1}e^{-iq\cdot(R+\tau_{1\alpha})}\left(\sum_{(R',1\beta)\in\Omega_1}(h_{1\beta}^{1\alpha})_{R-R'}\psi_{(R',1\beta)}\right.\\
        &\left.+\sum_{(R'',2\gamma)\in\Omega_2}h_{2\gamma}^{1\alpha}(R-R''+\tau_{1\alpha}-\tau_{2\gamma})\psi_{(R'', 2\gamma)}\right).
    \end{align}
    Examining the intralayer sum, we have
    \begin{align}
        S_{\mathrm{intra}} &= \frac{1}{c_1^*}\sum_{R(R',1\beta)}e^{-iq\cdot(R-R'+\tau_{1\alpha}-\tau_{1\beta})}(h_{1\beta}^{1\alpha})_{R-R'}e^{-iq\cdot (R+\tau_{1\beta})}\psi_{(R',1\beta)}\\
        &= \sum_{(R,1\beta)}e^{-iq\cdot(R+\tau_{1\alpha}-\tau_{1\beta})}(h_{1\beta}^{1\alpha})_{R}[\mathcal{G}_1\psi_1]_{\beta}(q)
    \end{align}
    which implies
    \begin{equation}
        \widetilde{H}_{1\beta}^{1\alpha} = \sum_{R\in\mathcal{R}_1}e^{-i(\cdot)\cdot(R+\tau_{1\alpha}-\tau_{1\beta})}[h_{1\beta}^{1\alpha}]_R.
    \end{equation}
    Examining the interlayer sum, we have
    \begin{equation}
        S_{\mathrm{inter}} = \frac{c_2^*}{c_1^*}\sum_{(R'',2\gamma)}\int_{\R^2}e^{i(\xi-q)\cdot(R''+\tau_{1\alpha})}\widehat{t}_{2\gamma}^{1\alpha}(\xi)[\mathcal{G}_2\psi_2]_{\gamma}(\xi)d\xi
    \end{equation}
    where the Poisson summation formula
    \begin{equation}
        \sum_{R}e^{i\xi\cdot R} = (c_1^*)^2\sum_{G}\delta(\xi-G)
    \end{equation}
    gives
    \begin{equation}
        S_{\mathrm{inter}}=c_1^*c_2^*\sum_{(G,2\gamma)}e^{iG\cdot\tau_{1\alpha}}\widehat{t}_{2\gamma}^{1\alpha}(q+G)[\mathcal{G}_2\psi_2]_{\gamma}(q+G)
    \end{equation}
    which implies
    \begin{equation}
        \widetilde{H}_{2\gamma}^{1\alpha} = c_1^*c_2^*\sum_{G\in\mathcal{R}_1^*}e^{iG\cdot\tau_{1\alpha}}\widehat{t}_{2\gamma}^{1\alpha}(\cdot+G)T_G.
    \end{equation}
 From this, we deduce \zcref{eq:momentum}. This completes the proof of \zcref{prop:momentum}. 
\end{proof}

\subsection{Proof of \zcref{lem:ergProd}}
\begin{proof}
    Let $\widehat{\psi},\widehat{\phi},\widehat{\omega}\in\widehat{\mathcal{E}}_{jq}(\chi_j^a)$ and $c_1,c_2\in\C$. First, we prove that $(\cdot,\cdot)_{j,\mathrm{erg}}$ is conjugate symmetric. To that end, we have
    \begin{align}
        \overline{(\widehat{\psi},\widehat{\phi})_{j,\mathrm{erg}}} &= \overline{\lim_{r\to\infty}\frac{1}{|\Omega_{jr}^*|}\sum_{(G,j\alpha)}\overline{\widetilde{\psi}_{(G,j\alpha)}}\widetilde{\phi}_{(G,j\alpha)}} = \lim_{r\to\infty}\frac{1}{|\Omega_{jr}^*|}\sum_{(G,j\alpha)}\overline{\overline{\widetilde{\psi}_{(G,j\alpha)}}\widetilde{\phi}_{(G,j\alpha)}}\\
        &=\lim_{r\to\infty}\frac{1}{|\Omega_{jr}^*|}\sum_{(G,j\alpha)}\overline{\widetilde{\phi}_{(G,j\alpha)}}\widetilde{\psi}_{(G,j\alpha)}\\
        &=(\widehat{\phi},\widehat{\psi})_{j,\mathrm{erg}}.
    \end{align}
    Next, we prove that $(\cdot,\cdot)_{j,\mathrm{erg}}$ is linear in the second argument. To that end, we have
    \begin{align}
        (\widehat{\psi},c_1\widehat{\phi}+c_2\widehat{\omega})_{j,\mathrm{erg}} &= \lim_{r\to\infty}\frac{1}{|\Omega_{jr}^*|}\sum_{(G,j\alpha)}\overline{\widetilde{\psi}_{(G,j\alpha)}}\left(c_1\widetilde{\phi}_{(G,j\alpha)}+c_2\widetilde{\omega}_{(G,j\alpha)}\right)\\
        & = c_1\lim_{r\to\infty}\frac{1}{|\Omega_{jr}^*|}\sum_{(G,j\alpha)}\overline{\widetilde{\psi}_{(G,j\alpha)}}\widetilde{\phi}_{(G,j\alpha)}\\
        &+c_2\lim_{r\to\infty}\frac{1}{|\Omega_{jr}^*|}\sum_{(G,j\alpha)}\overline{\widetilde{\psi}_{(G,j\alpha)}}\widetilde{\omega}_{(G,j\alpha)}\\
        & = c_1(\widehat{\psi},\widehat{\phi})_{j,\mathrm{erg}}+c_2(\widehat{\psi},\widehat{\omega})_{j,\mathrm{erg}}.
    \end{align}
    Lastly, we prove that $(\cdot,\cdot)_{j,\mathrm{erg}}$ is positive definite. Set $\widetilde{\phi} = \widetilde{\psi}$. Then
    \begin{equation}
         (\widehat{\psi},\widehat{\psi})_{j,\mathrm{erg}} = \lim_{r\to\infty}\frac{1}{|\Omega_{jr}^*|}\sum_{(G,j\alpha)}|\widetilde{\psi}_{(G,j\alpha)}|^2\geq 0.
    \end{equation}
    Hence, $(\cdot,\cdot)_{j,\mathrm{erg}}$ is positive semidefinite. It remains to show that
    \begin{equation}
        (\widehat{\psi},\widehat{\psi})_{j,\mathrm{erg}} = 0 \iff \widehat{\psi} = 0.
    \end{equation}
 Since $\widetilde{\psi}\in\chi_j^a$, it follows from Proposition 3.5 in \cite{cancesGeneralizedKuboFormulas2017} and the identity theorem that
    \begin{align}
        (\widehat{\psi},\widehat{\psi})_{j,\mathrm{erg}} &= \lim_{r\to\infty}\frac{1}{|\Omega_{jr}|}\sum_{(G,j\alpha)\in\Omega_{jr}}\left|\psi_{j\alpha}\left(q+\sum_{t=1}^{N}G_t\right)\right|^2dq\\
        &=\sum_{\alpha\in\mathcal{A}_j}\intbar_{\Gamma_j^*}|\psi_{j\alpha}(q)|^2dq = |\Gamma_j^*|^{-1}(\widetilde{\psi},\widetilde{\psi})_{\chi_j^a} = 0 \iff \widetilde{\psi} = 0.
    \end{align}
    This completes the proof of \zcref{lem:ergProd}.
\end{proof}

\subsection{Proof of \zcref{prop:reciprocal}}

\begin{proof}
    Without loss of generality, let $(G,1\alpha)\in\Omega_1^*$ and suppose that \zcref{ass:herm} is satisfied. Then
    \begin{align}
        \left(\mathcal{E}_{q}\widetilde{H}\widetilde{\psi}\right)_{(G, 1\alpha)} &= \sum_{1\beta\in\mathcal{A}_1}\widetilde{h}_{1\beta}^{1\alpha}\left(q+\sum_{t=1}^{3}G_t\right)\widetilde{\psi}_{1\beta}\left(q+\sum_{t=1}^{3}G_t\right)
        \\
        &+\sum_{\substack{G_1\in\mathcal{R}_1^*\\ 2\beta\in\mathcal{A}_2}}e^{iG_1\cdot\tau_{1\alpha}}\widehat{h}_{2\beta}^{1\alpha}\left(q+G_1+\sum_{t=1}^{3}G_t\right)\widetilde{\psi}_{2\beta}\left(q+G_1+\sum_{t=1}^{3}G_t\right)\\
        &= \sum_{(G',1\beta)\in\Omega_1^*}\widetilde{h}_{1\beta}^{1\alpha}\left(q+\sum_{t=1}^{3}G_t\right)\delta_{G'}^{G}\widetilde{\psi}_{1\beta}\left(q+\sum_{t=1}^{3}G_t'\right)
        \\
        &+\sum_{(G'', 2\beta)\in\Omega_2^*}T_{(G'',2\beta)}^{(G,1\alpha)}\widehat{h}_{2\beta}^{1\alpha}\left(q+G_1''+\sum_{t=1}^{3}G_t\right)\delta_{G''\setminus G_1''}^{G\setminus G_2}\widetilde{\psi}_{2\beta}\left(q+\sum_{t=1}^{3}G_t''\right)\\
        & = \sum_{(G', 1\beta)\in\Omega_1^*}\left[\widehat{H}_q\right]_{(G',1\beta)}^{(G, 1\alpha)}\left(\mathcal{E}_{1q}\widetilde{\psi}_1\right)_{(G',1\beta)}\\
        &+\sum_{(G'',2\beta)\in\Omega_2^*}\left[\widehat{H}_q\right]_{(G'', 2\beta)}^{(G, 1\alpha)}\left(\mathcal{E}_{2q}\widetilde{\psi}_2\right)_{(G'',2\beta)}
    \end{align}
    where in the interlayer term we have used quasi-periodicity to pull $e^{-iG_2\cdot\tau_{2\beta}}$ out of $\widetilde{\psi}_{2\beta}$. The remaining entries are deduced from the above structure. This completes the proof of \zcref{prop:reciprocal}.
\end{proof}

To prove \zcref{thm:equivtrace}, we first require a lemma on the intertwining of the resolvents.
\begin{lemma}
     \label{lem:resIntertwin}
    Under \zcref{ass:bound,ass:incom}, the intertwining relation
    \begin{equation}
        \mathcal{E}_qR_z(\widetilde{H}) = R_z(\widehat{H}_q)\mathcal{E}_q
    \end{equation}
    holds for all $z\not\in\sigma(\widetilde{H})\cup\sigma(\widehat{H})$.
\end{lemma}
\begin{proof}
    By Proposition 3.2, we have that $\mathcal{E}_{q}\widetilde{H} = \widehat{H}_q\mathcal{E}_q$. Let $z\in\C\setminus [\sigma(\widetilde{H})\cup\sigma(\widehat{H})]$. Then
    \begin{equation}
        \mathcal{E}_q(z-\widetilde{H})=(z-\widehat{H}_q)\mathcal{E}_q.
    \end{equation}
    Multiplying on right by $R_z(\widetilde{H})$ and on the left by $R_z(\widehat{H}_q)$ yields
    \begin{equation}
        R_z(\widehat{H})\mathcal{E}_q(z-\widetilde{H})R_z(\widetilde{H}) = R_z(\widehat{H})(z-\widehat{H}_q)\mathcal{E}_qR_z(\widetilde{H})
    \end{equation}
    or equivalently
    \begin{equation}
        R_z(\widehat{H})\mathcal{E}_q=\mathcal{E}_qR_z(\widetilde{H}).
    \end{equation}
\end{proof}
We are now ready to prove the equivalence of the thermodynamic limit traces and the complex linear functional.
\begin{proof}
    Observe that
    \begin{equation}
        \nu^*|\Gamma_j^*| = \lim_{r\to\infty}\frac{|\mathcal{R}_{jr}|}{|\Omega_{r}|}\quad\text{and}\quad
        \nu^*|\Gamma_j^*| = \lim_{r\to\infty}\frac{|\kappa_{jr}^*|}{|\Omega_{r}^*|}.
    \end{equation}
    Suppose $H$ satisfies \zcref{ass:bound,ass:incom} and let $f\in\mathcal{O}(\widehat{H})$. First, we prove that $\mathcal{T} f(\widehat{H}) = \underline{\operatorname{Tr}} f(H)$. Starting on the right hand side, we have
    \begin{align}
        \underline{\operatorname{Tr}} \; f(H) &= \lim_{r\to\infty}\frac{1}{|\Omega_r|}\sum_{(R,j\alpha)\in\Omega_r}[f(H)]_{(R,j\alpha)}^{(R,j\alpha)}\\
        &= \lim_{r\to\infty}\frac{1}{|\Omega_r|}\sum_{(R,j\alpha)\in\Omega_r}\oint_{C}f(z)[R_z(H)]_{(R,j\alpha)}^{(R,j\alpha)}dz.
    \end{align}
    Let $F_r:\mathbb{C}\to\mathbb{C}$ be defined by
    \begin{equation}
        F_r(z) = \frac{1}{|\Omega_r|}\sum_{(R,j\alpha)\in\Omega_r}[R_z(H)]_{(R,j\alpha)}^{(R,j\alpha)}.
    \end{equation}
    Then
    \begin{gather}
        |F_r(z)|\leq\frac{1}{|\Omega_r|}\sum_{(R,j\alpha)\in\Omega_r}\left|[R_z(H)]_{(R,j\alpha)}^{(R,j\alpha)}\right|\leq \frac{1}{|\Omega_r|}\sum_{(R,j\alpha)\in\Omega_r}\|R_z(H)\|_2\\
        \leq \frac{1}{|\Omega_r|}\sum_{(R,j\alpha)\in\Omega_r}M = M
    \end{gather}
    for some $M\in\mathbb{R}_+$. Hence, the function $g(z) = M|f(z)|$ dominates. By the Dominated Convergence Theorem, the limit can be brought inside the contour so that
    \begin{equation}
        \underline{\operatorname{Tr}} f(H) = \frac{1}{2\pi i}\oint_Cf(z)\underline{\operatorname{Tr}} \; R_z(H)dz.
    \end{equation}
    By \zcref{lem:ergProd} we have
    \begin{equation}
        |\Gamma_j^*|^{-1}(\psi,\phi)_{\chi_j} = \underline{\operatorname{Tr}}_{\Omega_j^*}(\mathcal{E}_{j0}\psi)^*(\mathcal{E}_{j0}\phi).
    \end{equation}
    By \zcref{lem:ergProd} and \zcref{lem:resIntertwin}, we have that
    \begin{align}
        [R_z(H)]_{(R,j\alpha)}^{(R,j\alpha)} &= [\mathcal{G}_j^*R_z(\widetilde{H})\mathcal{G}_j]_{(R,j\alpha)}^{(R,j\alpha)}\\
        &= |\Gamma_j^*|^{-1}\left(e_{j\alpha}e^{-iq\cdot (R+\tau_{j\alpha})}, R_z(\widetilde{H})e_{j\alpha}e^{-iq\cdot (R+\tau_{j\alpha})}\right)_{\chi_j}\\
        &= \underline{\operatorname{Tr}}_{\Omega_j^*}[\mathcal{E}_{j0}e_{j\alpha} e^{-iq\cdot(R+\tau_{j\alpha})}]^*[\mathcal{E}_{j0}R_z(\widetilde{H})e_{j\alpha} e^{-iq\cdot(R+\tau_{j\alpha})}]\\
        &= \underline{\operatorname{Tr}}_{\Omega_j^*}[\mathcal{E}_{j0}e_{j\alpha} e^{-iq\cdot(R+\tau_{j\alpha})}]^*[R_z(\widehat{H}_0)\mathcal{E}_{j0}e_{j\alpha} e^{-iq\cdot(R+\tau_{j\alpha})}]\\
        &= \lim_{r\to\infty}\frac{1}{|\kappa_{jr}^*|}\sum_{G\in\kappa_{jr}^*}\sum_{G'\in\kappa_j^*}[R_z(\widehat{H}_0)]_{(G,j\alpha)}^{(G',j\alpha)}\prod_{t\neq j}e^{iR\cdot (G_t-G_t')},
    \end{align}
    for all $(R,j\alpha)\in\Omega_j$. By ergodicity (see Proposition 3.5 in \cite{cancesGeneralizedKuboFormulas2017}), we have
    \begin{equation}
        \lim_{r\to\infty}\frac{1}{|\kappa_{jr}^*|}\sum_{G'\in\kappa_{jr}^*}[R_z(\widehat{H}_0)]_{(G,j\alpha)}^{(G',j\alpha)}\prod_{t\neq j}e^{-iR\cdot (G_t-G_t')} = \intbar_{\Gamma_j^*}[R_z(\widehat{H}_q)]_{(0,j\alpha)}^{(G',j\alpha)}\prod_{t\neq j}e^{-iR\cdot (G_t')}dq
    \end{equation}
    which implies
    \begin{align}
         [R_z(H)]_{(R,j\alpha)}^{(R,j\alpha)} = \sum_{G\in\kappa_j^*}\left(\prod_{t\neq j}e^{-iR\cdot G_t}\right)\intbar_{\Gamma_j^*}[R_z(\widehat{H}_q)]_{(0,j\alpha)}^{(G,j\alpha)}dq.
    \end{align}
    In addition, observe that
    \begin{equation}
        \lim_{r\to\infty}\frac{1}{|\mathcal{R}_{jr}|}\sum_{R_j\in\mathcal{R}_{jr}}\prod_{t\neq j}e^{-iR_j\cdot G_t} = \prod_{t\neq j}\delta_{G_t}^0 = \delta_G^0.
    \end{equation}
    Making the appropriate substitutions, we have
    \begin{gather}
        \underline{\operatorname{Tr}} f(H) = \frac{1}{2\pi i}\oint_Cf(z)\underline{\operatorname{Tr}} \; R_z(H)dz\\
         = \frac{1}{2\pi i}\oint_Cf(z)\left(\lim_{r\to\infty}\frac{1}{|\Omega_r|}\sum_{(R,j\alpha)\in\Omega_r}[R_z(H)]_{(R,j\alpha)}^{(R,j\alpha)}\right)dz\\
         = \frac{\nu^*}{2\pi i}\oint_{C}f(z)\left(\sum_{j\alpha\in\mathcal{A}}|\Gamma_j^*|\lim_{r\to\infty}\frac{1}{|\mathcal{R}_{jr}|}\sum_{R\in\mathcal{R}_{jr}}[R_z(H)]_{(R,j\alpha)}^{(R,j\alpha)}\right)dz\\
         = \frac{\nu^*}{2\pi i}\oint_Cf(z)\left[\sum_{(G,j\alpha)\in\Omega^*}\left(\lim_{r\to\infty}\frac{1}{|\mathcal{R}_{jr}|}\sum_{R\in\mathcal{R}_{jr}}\prod_{t\neq j}e^{-iR\cdot G_t}\right)\int_{\Gamma_j^*}[R_z(\widehat{H}_q)]_{(0,j\alpha)}^{(G,j\alpha)}dq\right]dz\\
         = \frac{\nu^*}{2\pi i}\oint_Cf(z)\left[\sum_{(G,j\alpha)\in\Omega^*}\delta_{G}^0\int_{\Gamma_j^*}[R_z(\widehat{H}_q)]_{(0,j\alpha)}^{(G,j\alpha)}dq\right]dz\\
         = \nu^*\sum_{j\alpha\in\mathcal{A}}\int_{\Gamma_j^*}[f(\widehat{H}_q)]_{(0,j\alpha)}^{(0,j\alpha)}dq = \mathcal{T} f(\widehat{H}).
    \end{gather}
    Next, we prove that $\underline{\operatorname{Tr}} f(\widehat{H}) = \mathcal{T} f(\widehat{H})$. Starting on the left hand side, we have
    \begin{align}
        \underline{\operatorname{Tr}} f(\widehat{H}) & = \lim_{r\to\infty}\frac{1}{|\Omega_r^*|}\sum_{(G,j\alpha)\in\Omega_r^*}[f(\widehat{H})]_{(G,j\alpha)}^{(G,j\alpha)}\\
        & = \lim_{r\to\infty}\frac{1}{|\Omega_r^*|}\sum_{(G,j\alpha)\in\Omega_r^*}\frac{1}{2\pi i}\oint_Cf(z)[R_z(\widehat{H})]_{(G,j\alpha)}^{(G,j\alpha)}dz\\
        & = \nu^*\sum_{j\alpha\in\mathcal{A}}|\Gamma_j^*|\lim_{r\to\infty}\frac{1}{|\kappa_{jr}^*|}\sum_{G\in\kappa_{jr}^*}\frac{1}{2\pi i}\oint_Cf(z)[R_z(\widehat{H})]_{(G,j\alpha)}^{(G,j\alpha)}dz.
    \end{align}
    By ergodicity, we have
    \begin{equation}
        \lim_{r\to\infty}\frac{1}{|\kappa_{jr}^*|}\sum_{G\in\kappa_{jr}^*}[R_z(\widehat{H})]_{(G,j\alpha)}^{(G,j\alpha)}=\intbar_{\Gamma_j^*}[R_z(\widehat{H}_q)]_{(0,j\alpha)}^{(0,j\alpha)}dq.
    \end{equation}
    Substitution and Fubini's theorem give
    \begin{equation}
        \underline{\operatorname{Tr}} f(\widehat{H}) = \nu^*\sum_{j\alpha\in\mathcal{A}}\int_{\Gamma_j^*}[f(\widehat{H}_q)]_{(0,j\alpha)}^{(0,j\alpha)}dq = \mathcal{T} f(\widehat{H}).
    \end{equation}
    This completes the proof of \zcref{thm:equivtrace}. 
\end{proof}

\subsection{Proof of \zcref{thm:layerconnect}}

In order to prove \zcref{thm:layerconnect}, we require the following lemma.
\begin{lemma}
    \label{lem:descend}
    For any $x\in\mathbb{R}^2$ satisfying
    \begin{equation}
        \label{eq:step}
        \|x\|_2>\frac{r_j}{2}\sec{\frac{\vartheta_j}{2}}
    \end{equation}
    there exists at least one step in $p\in P_j$ such that
    \begin{equation}
        \|x+p\|_2<\|x\|_2.
    \end{equation}
\end{lemma}

\begin{proof}
    Since $\mathcal{S}_j$ is centrally symmetric and spans $\mathbb{Z}^4$, the set $P_j$ is centrally symmetric, and its angular span partitions $\mathbb{R}^4$. Clearly, $\vartheta_j<\pi$. To step towards the origin from $x$, the ideal trajectory is $-x$. Because $\vartheta_j<\pi$, there exists a $p\in P_j$ such that $\alpha = \angle(p,-x)$ satisfies
    \begin{equation}
        \alpha\leq\frac{\vartheta_j}{2}.
    \end{equation}
    By the law of cosines
    \begin{equation}
        \|x+p\|_2^2 = \|x\|_2^2+\|p\|_2^2-2\|x\|_2\|p\|_2\cos{\alpha}.
    \end{equation}
    The distance to the origin is strictly reduced if and only if
    \begin{equation}
        \|p\|_2<2\|x\|_2\cos{\alpha}
    \end{equation}
    or equivalently
    \begin{equation}
        \|x\|_2>\frac{r_j}{2}\sec{\frac{\vartheta_j}{2}}.
    \end{equation}
\end{proof}

With \zcref{lem:descend} in hand, we proceed with the proof of \zcref{thm:layerconnect}.
\begin{proof}
    Let $x$ and $y$ be such that \zcref{eq:step} is satisfied. Then, by \zcref{lem:descend}, there exist paths $\gamma_x$ and $\gamma_y$ with steps in $P_j$ such that $\gamma_x$ and $\gamma_y$ take $x$ and $y$ to $0$, respectively. Since both paths meet at $0$, the combined path $\gamma = [\gamma_x, \overline{\gamma}_y]$ takes $x$ to $y$ where $\overline{\gamma}_y$ is the reversal of $\gamma_y$. Since each step in $\gamma$ is the image of a step in $\mathcal{S}_j$, we obtain a path in $\mathbb{Z}^4$ with the property that each step is bounded by $\mu_j\geq\max_{s\in\mathcal{S}_j}\|s\|_2$ with respect to the pseudometric $\mathfrak{d}_{jj}$. Therefore, if $\frac{r_j}{2}\sec\frac{\vartheta_j}{2}<W_j$ and $\|q-\widetilde{K}\|_2<W_j$, then $\mathcal{W}_{j}^*(q,W_j)$ is $(\mu_j,\mathfrak{d}_{jj})$-connected relative to $0$.

   Let $G\in\mathcal{W}_j^*(q,W_j)$ and $C_k = I-B_jB_k^{-1}$. Without loss of generality, suppose $\widetilde{K} = K_j$. Then
    \begin{equation}
        \|q-K_j+C_kG_k+C_lG_l\|_2<W_j
    \end{equation}
    which implies $q+C_kG_k+C_lG_l\in B_{W_j}(K_j)\subset\Gamma_j^*$. Hence
    \begin{align}
        \|[q+G_k+G_l]_j-K_j\| &= \|[q+C_kG_k+C_lG_l]_j-K_j\|\\
        &= \|q-K_j+G_k+G_l\|_2<W_j.
    \end{align}
    Therefore, $\mathcal{W}_j^*(q, W_j)\subset\Lambda_{\mu_j d_j}^*(q,W_j)$.

    It remains to prove the inclusion: $\Lambda_{\mu_j d_j}^*(q,W_j)\subset\mathcal{W}(q, W_j)$. Suppose $x\in\Lambda_{\mu_j d_j}^*(q,W_j)\setminus\mathcal{W}(q, W_j)$. Then, there exists a $(\mu_j,\mathfrak{d}_{jj})$-path $\gamma_x = (x_m)_{m=1}^N$ from $x$ to $0$ where $x_m = (G_k^{(m)}, G_l^{(m)})$. Let
    \begin{equation}
        z_m = q+C_kG_k^{(m)}+C_lG_l^{(m)}+G_j^{(m)}\in\Gamma_j^*
    \end{equation}
    for some $G_j^{(m)}\in\mathcal{R}_j^*$. For $x\not\in \mathcal{W}_j^*(q, W_j)$, either
    \begin{enumerate}
        \item[(i)] $[z_1]_j\neq 0$ or
        \item[(ii)] $z_1\in B_{W_j}(K_j')$.
    \end{enumerate}
    For case (i), there exists $m\in\{1,2,\ldots,N\}$ such that $G_j^{(m)}\neq 0$ and $G_j^{(m+1)}=0$. Observe that $\|z_{m}-z_{m+1}\|_2\leq 2W_j$.
    Taking the norm of the difference, we obtain the lower bound
    \begin{align}
        \|z_m-z_{m+1}\|_2\geq|\|G_j^{(m)}\|_2-\|C_k(G_k^{(m)}-G_k^{(m+1)})+C_l(G_l^{(m)}-G_l^{(m+1)})\|_2|.
    \end{align}
    Since $G_{j}^{(m)}$ is a nonzero lattice vector, there exists $g_{\min}>0$ such that $\|G_j^{(m)}\|_2\geq g_{\min}$. Additionally, $\|C_k\|_2, \|C_l\|_2\ll 1$, so that $\|[C_k,C_l]\|_2\ll1$ and
    \begin{equation}
       \|C_k(G_k^{(m)}-G_k^{(m+1)})+C_l(G_l^{(m)}-G_l^{(m+1)})\|_2 = u(\|G_k^{(m)}-G_k^{(m+1)}\|_2+\|G_l^{(m)}-G_l^{(m+1)}\|_2)
    \end{equation}
    for some $0<\mathfrak{u}\ll1$. Since $g_{\min}\gg \mathfrak{u}$ and $\|G_k^{(m)}\oplus G_l^{(m)}-G_k^{(m+1)}\oplus G_l^{(m+1)}\|_2\leq\mu_j$, it follows that
    \begin{equation}
        2W_j\geq g_{\min}-\mathfrak{u}\mu_j>g_{\min}-2\delta_j.
    \end{equation}
    By hypothesis $W_j<\frac{1}{2}\|K_j-K_j'\|_2-\delta_j$, which implies $2W_j<\|K_j-K_j'\|_2-2\delta_j$, so that $2W_j<g_{min}-2\delta_j$ is a contradiction. Thus, case (i) is excluded.
    For case (ii), there exists $m\in\{1,2,\ldots,N\}$ such that $z_m\in B_{W_j}(K_j')$ and $z_{m+1}\in B_{W_j}(K_j)$. Taking the norm of the difference, we obtain
    \begin{align}
        \|z_{m+1}-z_{m}\|_2&\geq\|K_j-K_j'\|_2-\|z_{m+1}-K_j\|_2-\|z_m-K_j'\|_2\\
        &\geq\|K_j-K_j'\|_2-2W_j>2\delta_j.
    \end{align}
    From case (i), we know that $\mathfrak{u}\mu_j\ll 2\delta_j$, which contradicts the existence of a step from $z_{m}$ to $z_{m+1}$. Thus, case (ii) is excluded. Therefore, $\Lambda_{\mu_j d_j}^*(q,W_j)\subset\mathcal{W}_j^*(q, W_j)$. This completes the proof of \zcref{thm:layerconnect}. 
\end{proof}

\subsection{Proof of \zcref{thm:cumulative}}
To prove \zcref{thm:cumulative}, we prove each of \zcref{lem:errorN,lem:errorW,lem:errorL}. We omit a proof of the $\tau$-truncation error since it is a direct consequence of our assumptions and the standard Estimation Lemma from complex analysis. We begin by proving \zcref{lem:errorN}. This requires two lemmas and two corollaries.
Let $f\in\mathcal{O}\left(\widehat{H}\right)$. Define $f_{j\alpha}:\mathbb{R}^2\to\mathbb{C}$ by
\begin{equation}
    f_{j\alpha}(q) = \left[f\left(\widehat{H}_q\right)\right]_{(0,j\alpha)}^{(0,j\alpha)}
\end{equation}
and $g_{j\alpha}:[0,2\pi)^2\to\mathbb{C}$ by
\begin{equation}
    g_{j\alpha}(x) = f_{j\alpha}(A_j^{-T}x).
\end{equation}
We now prove the first lemma and its corollary required to prove \zcref{lem:errorN}.
\begin{lemma}
    \label{lem:trPeriod}
    Under \zcref{ass:bound}, the function $f_{j\alpha}$ is $\mathcal{R}_j^*$-periodic for all $f\in\mathcal{O}\left(\widehat{H}\right)$.
\end{lemma}

\begin{proof}
    Let $f\in\mathcal{O}\left(\widehat{H}\right)$, $C$ be a contour around the spectrum of $\widehat{H}$ and $q\in\mathbb{R}^2$. For each 
    \begin{equation}
        G_m\in\bigcup_{j=1}^{N_l}\mathcal{R}_j^*,
    \end{equation} 
    define the unitary operator $U_{G_m}\in\mathcal{L}\left(\mathcal{H}\right)$ by 
    \begin{equation}
        \left[U_{G_m}\right]_{(G',k\beta)}^{(G, j\alpha)} = \delta_{k\beta}^{j\alpha}e^{i\delta_{j}^mG_m\cdot \tau_{j\alpha}}\prod_{t\neq j}\delta_{G_t'}^{G_t-\delta_{t}^mG_m}
    \end{equation}
    for all $((G,j\alpha), (\widetilde{G},k\beta))\in \Omega_j^*\times\Omega_k^*$. Let $\widetilde{G}\in\mathcal{R}_j^*$. Then, we have that
    \begin{align}          [U^*_{\widetilde{G}}\widehat{H}_qU_{\widetilde{G}}]_{(G',k\beta)}^{(G,j\alpha)} &= \sum_{G''G'''}e^{-i(\widetilde{G}\cdot\tau_{j\alpha}-\delta_k^j\widetilde{G}\cdot\tau_{k\beta})}\delta_{G_t}^{G_t''}\prod_{t\neq k}\delta_{G_t'}^{G_t'''+\delta_t^j\widetilde{G}}[\widehat{H}_{q}]_{(G''',k\beta)}^{(G'',j\alpha)}\\
    &= \sum_{G'''}e^{-i(\widetilde{G}\cdot\tau_{j\alpha}-\delta_k^j\widetilde{G}\cdot\tau_{k\beta})}\prod_{t\neq k}\delta_{G_t'}^{G_t'''+\delta_t^j\widetilde{G}}[\widehat{H}_{q}]_{(G''',k\beta)}^{(G,j\alpha)}
    \end{align}
    where if $j = k$ we have that
    \begin{align}
        [U^*_{\widetilde{G}}\widehat{H}_qU_{\widetilde{G}}]_{(G',j\beta)}^{(G,j\alpha)} &= \sum_{G'''}e^{-i\widetilde{G}\cdot(\tau_{j\alpha}-\tau_{k\beta})}\delta_{G'}^{G'''}[\widehat{H}_{q}]_{(G''',k\beta)}^{(G,j\alpha)} \\&= \sum_{G'''}e^{-i(\widetilde{G}\cdot(\tau_{j\alpha}-\tau_{k\beta})}\delta_{G'}^{G'''}[\widehat{H}_{q}]_{(G''',k\beta)}^{(G,j\alpha)} \\&= e^{-i(\widetilde{G}\cdot(\tau_{j\alpha}-\tau_{k\beta})}[\widehat{H}_{q}]_{(G',k\beta)}^{(G,j\alpha)}\\
        &=[\widehat{H}_{q+G}]_{(G',k\beta)}^{(G,j\alpha)}
    \end{align}
    since $e^{-i\widetilde{G}\cdot R}=1$ for all $R\in\mathcal{R}_j$. If $j\neq k$, then we obtain
    \begin{align}
        [U^*_{\widetilde{G}}\widehat{H}_qU_{\widetilde{G}}]_{(G',k\beta)}^{(G,j\alpha)} &= \sum_{G'''}e^{-i\widetilde{G}\cdot\tau_{j\alpha}}\delta_{G_j'}^{G_j'''+\widetilde{G}}\delta_{G'\setminus G_j'}^{G'''\setminus G_j'''}[\widehat{H}_{q}]_{(G''',k\beta)}^{(G,j\alpha)}\\
        &=e^{-i\widetilde{G}\cdot\tau_{j\alpha}}[\widehat{H}_{q}]_{(G'_j+\widetilde{G},\cdots,k\beta)}^{(G,j\alpha)}\\
        &=e^{-i\widetilde{G}\cdot\tau_{j\alpha}}\delta_{G'\setminus G_j'}^{G\setminus G_k}e^{i[(G_j'+\widetilde{G})\cdot\tau_{j\alpha}-G_k\cdot\tau_{k\beta}]}\hat{h}_{k\beta}^{j\alpha}(q+\widetilde{G}+G_j'+\sum_{t=1}^NG_t)\\
        &=[\widehat{H}_{q+G}]_{(G',k\beta)}^{(G,j\alpha)}.
    \end{align}
    Hence, we have that
    \begin{align}
        f_{j\alpha}(\widehat{H}_{q+G}) &= \frac{1}{2\pi i}\oint_C f(z)[(z-\widehat{H}_{q+G})^{-1}]_{(0,j\alpha)}^{(0,j\alpha)}dz\\
        &=\frac{1}{2\pi i}\oint_C f(z)[(zU^*_GU_G-U_G^*\widehat{H}_{q}U_G)^{-1}]_{(0,j\alpha)}^{(0,j\alpha)}dz\\
        &=\frac{1}{2\pi i}\oint_Cf(z)[U_G^*(z-\widehat{H}_q)^{-1}U_G]_{(0,j\alpha)}^{(0,j\alpha)}dz\\
        &=\frac{1}{2\pi i}\oint_Cf(z)(U_Ge_{(0,j\alpha)})^*R_z(\widehat{H}_q)U_Ge_{(0,j\alpha)}dz\\
        &=\frac{1}{2\pi i}\oint_Cf(z)e^{-iG\cdot\tau_{j\alpha}}e_{(0,j\alpha)}^*R_z(\widehat{H}_q)e^{iG\cdot\tau_{j\alpha}}e_{(0,j\alpha)}dz\\
        &=\frac{1}{2\pi i}\oint_Cf(z)[R_z(\widehat{H}_q)]_{(0,j\alpha)}^{(0,j\alpha)}dz\\
        &=f_{j\alpha}(\widehat{H}_q).
    \end{align}
    Therefore, $f_{j\alpha}$ is $\mathcal{R}_j^*$-periodic.
\end{proof}

Since $f_{j\alpha}$ is periodic and continuous, it is bounded on $\mathbb{R}^2$. We use the properties of $f_{j\alpha}$ to prove the $2\pi$-periodicity of $g_{j\alpha}$.
\begin{corollary}
    \label{cor:g-periodic}
    Under \zcref{ass:bound}, the function $g_{j\alpha}$ is $2\pi$-periodic in each variable for all $f\in\mathcal{O}\left(\widehat{H}\right)$.
\end{corollary}

\begin{proof}
    By \zcref{lem:trPeriod}, we have that $f_{j\alpha}$ is $\mathcal{R}_j^*$ periodic. Let $e_{t}$ be either of the $2$d standard ordered basis vectors and $x\in[0,2\pi)^2$. Then 
    \begin{equation}
        g_{j\alpha}(x+2\pi e_t) = f_{j\alpha}(A_j^{-T}x+B_je_t) = f_{j\alpha}(A_j^{-T})=g_{j\alpha}(x).
    \end{equation}
    Therefore, $g_{j\alpha}$ is $2\pi$-periodic in each variable.
\end{proof}

Next, we prove the second lemma and its corollary required to prove \zcref{lem:errorN}. 
\begin{lemma}
    \label{lem:f-analytic}
    Under \zcref{ass:bound}, the function $f_{j\alpha}$ extends to a function analytic in each variable within the strip 
    \begin{equation}
        \mathcal{S}_{\rho_j} = \{z\in\mathbb{C}^2: \|\operatorname{Im}z\|_2<\rho_j\}
    \end{equation}
    where
    \begin{equation}
        \rho_j \leq \min\left(\frac{\gamma_1}{2}, \frac{\epsilon}{C_j}\right)
    \end{equation}
    for some $C_j\in\mathbb{R}_+$.
\end{lemma}

\begin{proof}  
    By \zcref{ass:bound}, $\hat{h}_{k\beta}^{j\alpha}$ extends to a function analytic in each variable within the strip $\mathcal{S}_{\gamma_1}$, and $h_{j\beta}^{j\alpha}$ extends to a function analytic in each variable within the strip $\mathcal{S}_{\gamma_1}$. Hence, extension of $\widehat{H}_q$ is well-defined. Let $q,\zeta\in\mathbb{R}^2$ such that $\|\zeta\|_2<\rho_j$. By the second resolvent identity, we have
    \begin{equation}
        R_z(\widehat{H}_{q+i\zeta}) = R_z(\widehat{H}_q)(I-\Delta_{i\zeta}\widehat{H}_{q}R_z(\widehat{H}_q))^{-1}
    \end{equation}
    where $\Delta_{i\zeta}$ is the forward difference operator, that is,
    \begin{equation}
        \Delta_{i\zeta}\widehat{H}_q = \widehat{H}_{q+i\zeta}-\widehat{H}_q.
    \end{equation}
    It follows that the Neumann series for $R_z(\widehat{H}_{q+i\zeta})$ is given by
    \begin{equation}
        R_z(\widehat{H}_{q+i\zeta}) = R_z(\widehat{H}_q)\sum_{n=0}^\infty[(\Delta_{i\zeta}\widehat{H}_q)R_z(\widehat{H}_q)]^n.
    \end{equation}
    This series converges uniformly for all $z\in\mathbb{C}$ provided
    \begin{equation}
        \|(\Delta_{i\zeta}\widehat{H}_q)R_z(\widehat{H}_q)\|_2<1
    \end{equation}
    in which case $f(\widehat{H}_q)$ is analytic via the Cauchy integral formula. Since $\|R_z(\widehat{H}_q)\|_2\leq\epsilon^{-1}$ on the contour $C$, it suffices to show that
    \begin{equation}
        \|\Delta_{i\zeta}\widehat{H}_q\|<\epsilon.
    \end{equation}
    Let $J_j:\ell^2(\Omega_j^*)\to\ell^2(\Omega^*)$ be the operator extending to $0$ in $\Omega^*\setminus\Omega_j^*$. Define the orthogonal projection $P_j:\ell^2(\Omega^*)\to\ell^2(\Omega^*)$ by $P_j = J_jJ_j^*$. Then $\widehat{H}_q$ has the representation
    \begin{equation}
        \widehat{H}_q = \sum_{t\in\mathbb{Z}}\sum_{j-k = t}\left[P_j\widehat{H}_{q}P_k\right]_{jk}.
    \end{equation}
    For a block diagonal operator matrix $X$, we have
    \begin{equation}
        \|X\|_{2} \leq \max_{j=k}\|X_{jk}\|_2
    \end{equation}
    where $\|X\|_2$ is the $\ell^2\to\ell^2$ operator norm. We denote by $\|X\|_1$ and $\|X\|_\infty$ the $\ell^1\to\ell^1$ and $\ell^{\infty}\to\ell^{\infty}$ operator norms (that is, the supremum absolute column sum and supremum absolute row sum). Observe that a generic operator $X\in\mathcal{L}(\ell^2(\Omega^*))$ acts on
    \begin{equation}
        \bigoplus_{j=1}^{N_l}\ell^2(\Omega_j^*,\mu_j)
    \end{equation}
    where $\mu_j$ is the counting measure on $\Omega_j^*$. Since $\mathcal{A}_j$ is finite for each $j$, $\Omega_j^*$ is countable, which implies that $\mu_j$ is $\sigma$-finite. By \zcref{ass:bound}, the operator $\Delta_{i\zeta}\widehat{H}_q$ is bounded on both $\ell^1(\Omega^*)$ and $\ell^\infty(\Omega^*)$, which share the common dense subspace of finite support sequences. Consequently, the Riesz-Thorin interpolation theorem (Theorem IX.17 in \cite{simon_fourier_2007}) is satisfied. By the triangle inequality, taking the maximum over each diagonal, and invoking Riesz-Thorin interpolation, we have
    \begin{align}
        \|\Delta_{i\zeta}\widehat{H}_q\|_2\\
        &\leq \sum_{t\in\mathbb{Z}}\left\|\sum_{j-k=t}\Delta_{i\zeta}P_j\widehat{H}_qP_k\right\|_2\\
        &\leq \sum_{t\in\mathbb{Z}}\max_{j-k=t}\|\Delta_{i\zeta}P_j\widehat{H}_qP_k\|_2\\
        &\leq \sum_{t\in\mathbb{Z}}\max_{j-k=t}\sqrt{\|\Delta_{i\zeta}P_j\widehat{H}_qP_k\|_1\|[\Delta_{i\zeta}P_j\widehat{H}_qP_k\|_\infty}.
    \end{align}
    Hence, it suffices to bound $\|\Delta_{i\zeta}P_j\widehat{H}_qP_k\|_1$ and $\|\Delta_{i\zeta}P_j\widehat{H}_qP_k\|_\infty$ for each $j-k\in\mathbb{Z}$. Define the total shifted momentum function $Q:\mathbb{R}^2\times\kappa_j^*\times\kappa_k^*\to\mathbb{R}^2$ by
    \begin{equation}
        Q(q,G,G') = q+G_j'+\sum_{t=1}^{N_l}G_t.
    \end{equation} 
    Evaluating the difference $\Delta_{i\zeta}P_j\widehat{H}_qP_k$, we have 
    \begin{gather}
        [\Delta_{i\zeta}P_j\widehat{H}_qP_k]_{(G',k\beta)}^{(G,j\alpha)} = \delta_{k}^j\sum_{R\in\mathcal{R}_j}(h_{j\beta}^{j\alpha})_Re^{-i\left(q+\sum_{t=1}^{N_l}G_t\right)\cdot(R+\tau_{j\alpha}-\tau_{j\beta})}(e^{\zeta\cdot(R+\tau_{j\alpha}-\tau_{j\beta})}-1)\\
            +(1-\delta_k^j)\delta_{(G',k\beta)}^{(G,j\alpha)}T_{(G',k\beta)}^{(G,j\alpha)}\left[\hat{h}_{k\beta}^{j\alpha}\left(Q(q, G, G')+i\zeta\right)-\hat{h}_{k\beta}^{j\alpha}\left(Q(q,G,G')\right)\right]
    \end{gather}
    Examining an intralayer block, we have
    \begin{align}
        \|\Delta_{i\zeta}P_j\widehat{H}_qP_j\|_{1}
        &\lesssim\|\zeta\|_2\max_{j\beta\in\mathcal{A}_j}\sum_{(R,j\alpha)\in\mathcal{R}_j}\|R+\tau_{j\alpha}-\tau_{j\beta}\|_2e^{\|\zeta\|_2\|R+\tau_{j\alpha}-\tau_{j\beta}\|_2-\gamma_1\|R\|_2}\\
        &\lesssim\|\zeta\|_2\max_{j\beta\in\mathcal{A}_j}\sum_{j\alpha\in\mathcal{A}_j}e^{\|\zeta\|_2\|\tau_{j\alpha}-\tau_{j\beta}\|_2}\int_{0}^\infty (r^2+r\|\tau_{j\alpha}-\tau_{j\beta}\|_2)e^{(\|\zeta\|_2-\gamma_1)r}dr\\
        &=\|\zeta\|_2\max_{j\beta\in\mathcal{A}_j}\sum_{j\alpha\in\mathcal{A}_j}\left[\frac{4\pi}{(\gamma_1-\|\zeta\|_2)^3}+\frac{2\pi\|\tau_{j\alpha}-\tau_{j\beta}\|_2}{(\gamma_1-\|\zeta\|_2)^2}\right]e^{\|\zeta\|_2\|\tau_{j\alpha}-\tau_{j\beta}\|_2}\\
        &\lesssim\|\zeta\|_2.
     \end{align}
    Swapping the arguments of the max and the sum, we obtain $\|\Delta_{i\zeta}P_j\widehat{H}_qP_j\|_{\infty}\lesssim\|\zeta\|_2$. Examining an interlayer block, we have
    \begin{align}
        \|\Delta_{i\zeta}P_j\widehat{H}_qP_{j+k}\|_{1} &= \leq\|\zeta\|_2\sup_{(G',(j+k)\beta)\in\Omega_{j+k}^*}\sum_{(G,j\alpha)\in\Omega_j^*}\sup_{t\in(0,1)}\|\nabla \widehat{h}_{(j+k)\beta}^{j\alpha}(Q(q, G, G')+it\zeta)\|_2\\
        &\lesssim \|\zeta\|_2e^{\gamma_2\|\zeta\|_2}\sup_{(G',(j+k)\beta)\in\Omega_{j+k}^*}\sum_{(G,j\alpha)\in\Omega_j^*}e^{-\gamma_2\|Q(q,G,G')\|_2}\\
        &\lesssim \|\zeta\|_2\int_{0}^\infty re^{-\gamma_2r}dr\\
        &\lesssim\|\zeta\|_2.
    \end{align}
    Swapping the arguments of the max and the sum, we obtain $\|\Delta_{i\zeta}P_j\widehat{H}_qP_{j+k}\|_{\infty}\lesssim\|\zeta\|_2$. So altogether, we obtain the bound
    \begin{equation}
        \|\Delta_{i\zeta}\widehat{H}_q\|_2\leq C_j\|\zeta\|_2
    \end{equation}
    for some $C_j\in\mathbb{R}_+$ (which may be computed).
    Since $\|\zeta\|_2<\frac{\epsilon}{C_j}$, it follows that
    \begin{equation}
        \|\Delta_{i\zeta}\widehat{H}_q\|_2<\epsilon.
    \end{equation}
    This completes the proof.
\end{proof}

By the properties of linear transformations, we obtain the following corollary immediately.
\begin{corollary}
    \label{cor:g-analytic}
    Under \zcref{ass:bound}, the function $g_{j\alpha}$ extends to a function analytic in each variable within the strip 
    \begin{equation}
        \mathcal{S}_{\widetilde{\rho}_j} = \{z\in\mathbb{C}^2:\|A_j^{-T}\operatorname{Im}(z)\|_2<\widetilde{\rho}_j\}
    \end{equation}
    where
    \begin{equation}
        \widetilde{\rho}_j\leq\|A_j^{-T}\|_2^{-1}\min\left(\frac{\gamma_1}{2}, \frac{\epsilon}{C_j}\right)
    \end{equation}
    for some $C_j\in\mathbb{R}_+$.
\end{corollary}

For $\mathcal{T}f(\widehat{H})$ to converge uniformly under \zcref{ass:bound}, we take the minimum width strip of analyticity $\mathcal{S}_{\widetilde{\rho}}$ where
\begin{equation}
    \tilde{\rho} = \min\left(\frac{\gamma_1}{2},\frac{\epsilon}{C_j}\right)\min_{j\in\{1,2,3\}}\|A_j^{-T}\|_2^{-1}.
\end{equation}

We are now ready to prove \zcref{lem:errorN}.
\begin{proof}
    Let $M_f = \sup_{z\in C}|f(z)|$. Applying the estimation lemma and the triangle inequality, we have
    \begin{gather}
       |\mathcal{T} f(\widehat{H})-\mathcal{T}_{N}f(\widehat{H})| \lesssim M_f\int_0^{2\pi}\left|\int_0^{2\pi}g_{j\alpha}(x_1, x_2)dx_1-\frac{2\pi}{N}\sum_{k_1=0}^{N-1}g_{j\alpha}\left(\frac{2\pi k_1}{N}, x_2\right)\right|dx_2\\
        \left.+\frac{2\pi M_f}{N}\sum_{k_1=0}^{N-1}\left|\int_0^{2\pi}g_{j\alpha}\left(\frac{2\pi k_1}{N}, x2\right)dx_2-\frac{2\pi}{N}\sum_{k_2=0}^{N-1}g_{j\alpha}\left(\frac{2\pi k_1}{N},\frac{2\pi k_2}{N}\right)\right|\right.
    \end{gather}
    By \zcref{cor:g-periodic,cor:g-analytic}, $g_{j\alpha}$ is $2\pi$-periodic in each variable and analytic in the strip $\mathcal{S}_{\widetilde{\rho}_j}$. Hence, by Theorem 9.28 in \cite{kress_numerical_1998}, we obtain the bound
    \begin{equation}
        \int_0^{2\pi}g_{j\alpha}(x_1, x_2)dx_1-\frac{2\pi}{N}\sum_{k_1=0}^{N-1}g_{j\alpha}\left(\frac{2\pi k_1}{N}, x_2\right)\lesssim\frac{1}{e^{\widetilde{\rho}_j N}-1}.
    \end{equation}
    The same argument applies to the second term. Substituting both bounds and taking the minimum over $j$, we obtain
    \begin{equation}
        |\mathcal{T} f(\widehat{H})-\mathcal{T}_{N}f(\widehat{H})|\lesssim \frac{M_f}{e^{\widetilde{\rho} N}-1}\lesssim M_fe^{-\widetilde{\rho} N}.
    \end{equation}
\end{proof}

Next, we prove \zcref{lem:errorW}.
\begin{proof}
    Let $C$ be a contour about $\sigma(\widehat{H})$ such that
    \begin{equation}
        d(z,\sigma(\widehat{H}))\in(\epsilon,2\epsilon)
    \end{equation}
    for all $z\in C$. Decompose $C$ as follows
    \begin{align}
        C_+ &= \{z\in C:\operatorname{Re}(z)\in\Sigma+B_\eta\},\\
        C_- &= C\setminus C_+,
    \end{align}
    where $\Sigma$ is the energy window of interest and 
    \begin{equation}
        \eta = (2+\alpha)\|\widehat{H}^{\tau}_{\mathrm{inter}}\|_2.
    \end{equation}
    Let
    \begin{align}
        \eta_W^+:&=\oint_{C_-}|\delta_\epsilon(E-z)|\|R_z(\widehat{H}^{\tau})-R_z(\widehat{H}^{(\tau,W)})\|_2|dz|\;\;\text{and}\\
        \eta_W^-:&=\oint_{C_+}|\delta_\epsilon(E-z)|\|R_z(\widehat{H}^{\tau})-R_z(\widehat{H}^{(\tau,W)})\|_2|dz|.
    \end{align}
    Then we have that
    \begin{equation}
        \eta_W\lesssim\eta_W^++\eta_W^-.
    \end{equation}
    Examining $\eta_W^-$, we obtain
    \begin{align}
        \eta_W^-&\lesssim\oint_{C_-}|\delta_\epsilon(E-z)|\|R_z(\widehat{H}^{\tau})-R_z(\widehat{H}^{(\tau, W)})\|_2|dz|\\
        &=\oint_{C_-}|\delta_\epsilon(E-z)|\|R_z(\widehat{H}_q^{\tau})(\widehat{H}_q^{(\tau, W)}-\widehat{H}_q^{\tau})R_z(\widehat{H}_q^{(\tau,W)})\|_2|dz|\\
        &\lesssim\epsilon^{-2}\oint_{C_-}|\delta_\epsilon(E-z)||dz|\\
        &\lesssim\epsilon^{-3}e^{-\eta^2\epsilon^{-2}}.
    \end{align}

    Let $0<W_0<W_1<\cdots<W_{N_W-1}<W_{N_W} = W$ be a strictly positive increasing sequence of radii. Define the following sets
    \begin{align}
        U_1 &= \Omega_{W_0}^*(q),\\
        U_k &= \Omega_{W_k}^*(q)\setminus\Omega_{W_{k-1}}^*(q) \;\;\text{for each}\;\;k\in\{1,\ldots,N_W\}\\
        U_\infty & = \Omega^*\setminus\Omega_W^*(q).
    \end{align}
    Additionally, for each $k\in\{0,1,\ldots,N_W,\infty\}$, let $J_k:\ell^2(U_k)\to\ell^2(\Omega^*)$ be the operator that extends a state $\psi\in\ell^2(U_k)$ by $0$ in $\Omega^*\setminus U_k$, that is,
    \begin{equation}
        (J_k\psi)_x = 
        \begin{cases}
            \psi_x, & x\in U_k,\\
            0, & x\in\Omega^*\setminus U_k.
        \end{cases}
    \end{equation}
    Then the adjoint $J_k^*:\ell^2(\Omega^*)\to\ell^2(U_k)$ is the operator that restricts the domain of a state $\psi\in\ell^2(\Omega^*)$ to $U_k$, that is,
    \begin{equation}
        J_k^*\psi = \psi\big\vert_{U_k}.
    \end{equation}
    Additionally, we define the composition of these injections
    $J_{k:k'}:\bigoplus_{j=k}^{k'}\ell^2(U_j)\to\ell^2(\Omega^*)$ by
    \begin{equation}
        J_{k:k'}\psi = \sum_{j=k}^{k'}J_j\psi_j.
    \end{equation}
    Then, the adjoint $J_{k:k'}^*:\ell^2(\Omega^*)\to\bigoplus_{j=k}^{k'}\ell^2(U_j)$ is defined by
    \begin{equation}
        J_{k:k'}^*\psi = \bigoplus_{j=k}^{k'}J_j^*\psi.
    \end{equation}
    We define the entries of the shell decomposition $\mathfrak{h}$ of $\widehat{H}_q^{\tau}$ by
    \begin{equation}
        \mathfrak{h}_{jk} = J_j^*\widehat{H}_q^{\tau}J_k.
    \end{equation}
    Choose $N_W>0$ such that
    \begin{equation}
        W_j = W_0+j\frac{W-W_0}{N_W}
    \end{equation}
    and
    \begin{equation}
        \mathfrak{h}_{jk} = 0
    \end{equation}
    for all $|j-k|>1$. For the sake of brevity, we write $\mathfrak{h}_j = \mathfrak{h}_{jj}$ and $\mathfrak{h}_{j:k} = \mathfrak{h}_{j:k,j:k} = J_{j:k}^*\widehat{H}_q^{\tau}J_{j:k}$. Then $\mathfrak{h}$ is a block tri-diagonal operator matrix with a nearest neighbor structure. Let $m = \left\lfloor\frac{N_W}{2}\right\rfloor$. Now, for the $C_+$ part of the contour, there are two cases:
     \begin{enumerate}
         \item[(i)] $k\leq m$ and
         \item[(ii)] $k>m$.
     \end{enumerate}
     Let $\eta_1$ and $\eta_2$ denote the error for case (i) and case (ii), respectively.
     Then we have that
     \begin{equation}
         \eta_W^+\leq\max(\eta_1,\eta_2).
     \end{equation}
     For case (i), we have
     \begin{equation}
         \eta_1\lesssim\oint_{C+}|\delta_\epsilon(E-z)|\|J_k^*RJ_k-J_k^*R_{0:N_W}J_k\|_2|dz|.
     \end{equation}
     Let $M$ be the block operator matrix defined by
     \begin{equation}
         M = \begin{bmatrix}
             z-\mathfrak{h}_{0:N_W} & -\mathfrak{h}_{0:N_W,\infty}\\
             -\mathfrak{h}_{\infty,0:N_W} & z-\mathfrak{h}_\infty
         \end{bmatrix}.
     \end{equation}
    Then, the Banachiewicz inversion formula yields
     \begin{equation}
         [M^{-1}]_{11} = R_{0:N_W}-R_{0:N_W}\mathfrak{h}_{0:N_W,\infty}J_\infty^*RJ_\infty\mathfrak{h}_{\infty,0:N_W}R_{0:N_W}.
     \end{equation}
     It follows that
     \begin{equation}
         \|J_k^*RJ_k-J_k^*R_{0:N_W}J_k\|_2 = \|J_k^*[M^{-1}]_{11}J_k-J_k^*R_{0:N_W}J_k\|_2\lesssim\epsilon^{-2}\|J_k^*R_{0:N_W}J_{N_W}\|_2.
     \end{equation}
     Define the following sequence of operators
     \begin{align}
         \mathfrak{W}_j &= \mathfrak{h}_{j-1,j}J_j^*R_{j:N_W}^*J_J\;\;\text{for all}\;\;j\in\{k+1,\ldots, N_W\},\\
         \mathfrak{W}_k&=J_k^*R_{0:N_W}J_k,
     \end{align}
     that is, the sequence of $[M^{-1}]_{12}$ as $N_W$ decreases to $k$. Then, we have
     \begin{equation}
         J_k^*R_{0:N_W}J_{N_W} = (-1)^{N_W-k}\mathfrak{W}_k\prod_{j=k+1}^{N_W}\mathfrak{W}_j
     \end{equation}
     which implies
     \begin{equation}
         \|J_k^*R_{0:N_W}J_{N_W}\|_2\lesssim\epsilon^{-1}\prod_{j=k+1}^{N_W}\|\mathfrak{W}_j\|_2.
     \end{equation}
     Observe that
     \begin{equation}
         \|\mathfrak{h}_{j-1, j}\|\leq\left\|\widehat{H}^{\tau}_{\mathrm{inter}}\right\|_2
     \end{equation}
     and
     \begin{align}
         \|J_j^*R_{j:N_W}J_j\|_2&\leq\|R_{j:N_W}\|_2\\
         &\leq\left(d\left(z,\sigma\left(\mathfrak{h}_{j:N_W}^{\mathrm{intra}}\right)\right)-\left\|\widehat{H}_{\mathrm{inter}}^{\tau}\right\|_2\right)^{-1}\\
         &\leq\left(\frac{2+\alpha}{2}\left\|\widehat{H}_{\mathrm{inter}}^{\tau}\right\|_2\right)^{-1}
     \end{align}
     which holds for
     \begin{equation}
         d\left(z,\sigma\left(\mathfrak{h}_{j:N_W}^{\mathrm{intra}}\right)\right)\geq \frac{4+\alpha}{2}\left\|\widehat{H}_{\mathrm{inter}}^{\tau}\right\|_2.
     \end{equation}
     Hence,
     \begin{equation}
         \|\mathfrak{W}_j\|_2\leq\left(\frac{2+\alpha}{2}\right)^{-1}
     \end{equation}
     This gives a bound of
     \begin{equation}
         \eta_1\lesssim\epsilon^{-4}\prod_{j=k+1}^{N_W}\left(\frac{2+\alpha}{2}\right)^{-1}\leq\epsilon^{-4}\left(\frac{2+\alpha}{2}\right)^{-\left\lceil\frac{N_W}{2}\right\rceil}=\epsilon^{-4}e^{-\left\lceil\frac{N_W}{2}\right\rceil\ln{\frac{2+\alpha}{2}}}.
     \end{equation}
    Now, solving for the tightest $N_W$ yields
     \begin{equation}
         N_W = \left\lfloor\frac{W-W_0}{\mathfrak{u}}\right\rfloor
     \end{equation}
    where $\mathfrak{u}$ is defined by
     \begin{equation}
         \mathfrak{u} = \max_{j\in\{1,2,3\}}\mu_j\|[I-B_jB_k^{-1}, I-B_jB_l^{-1}]\|_2.
     \end{equation}
     For $z\in C_+$, we have
     \begin{equation}
         |z|\leq\sup_{E\in\Sigma}|E|+\eta+2\epsilon.
     \end{equation}
     Since $\mathfrak{h}^{\mathrm{intra}}$ outside of $\Omega_{W_0}^*$ has eigenvalues bounded below by $vW_0$,
     the distance is
     \begin{equation}
          d\left(z,\sigma\left(\mathfrak{h}_{j:N_W}^{\mathrm{intra}}\right)\right)\geq vW_0-\left(\sup_{E\in\Sigma}|E|+\eta+2\epsilon\right).
     \end{equation}
     It follows that
     \begin{equation}
         vW_0-\left(\sup_{E\in\Sigma}|E|+(2+\alpha)\left\|\widehat{H}_{\mathrm{inter}}^{\tau}\right\|_2+2\epsilon\right)\geq \frac{4+\alpha}{2}\left\|\widehat{H}_{\mathrm{inter}}^{\tau}\right\|_2
     \end{equation}
     which implies
     \begin{equation}
          W_0\geq \frac{1}{v}\left(\sup_{E\in\Sigma}|E|+\frac{8+3\alpha}{2}\left\|\widehat{H}^{\tau}_{\mathrm{inter}}\right\|_2+2\epsilon\right).
     \end{equation}
     For case (ii), we have
     \begin{align}
         \eta_2&\lesssim\oint_{C_+}|\delta_\epsilon(E-z)|\left(\|J_k^*RJ_k-J_k^*R_{m_k:\infty}J_k\|_2+\|J_k^*R_{0:N_W}J_k-J_k^*R_{m_k:N_W}J_k\|_2\right)|dz|\\\
         &+\|J_k^*\delta_\epsilon(E-\mathfrak{h}_{m_k:\infty})J_K\|_2+\|J_k^*\delta_\epsilon(E-\mathfrak{h}_{m_k:N_W})J_k\|_2
     \end{align}
     where $m_k = \left\lfloor\frac{k}{2}\right\rfloor$.
     Let $\widetilde{M}$ be the block operator matrix defined by
     \begin{equation}
         \widetilde{M} = 
         \begin{bmatrix}
             z-\mathfrak{h}_{0:m_k-1} & -\mathfrak{h}_{0:m_k-1,m_k:\infty}\\
             -\mathfrak{h}_{m_k:\infty,0:m_k-1} & z-\mathfrak{h}_{m_k:\infty}
         \end{bmatrix}.
     \end{equation}
     Then, the Banachiewicz inversion formula yields
     \begin{equation}
         [\widetilde{M}^{-1}]_{22} = R_{m_k:\infty}-R_{m_k:\infty}\mathfrak{h}_{m_k:\infty,0:m_k-1}J_{0:m_k-1}^*RJ_{0:m_k-1}\mathfrak{h}_{0:m_k-1,m_k:\infty}R_{m_k:\infty}
     \end{equation}
     where $R_{m_k:\infty} = R_z(\mathfrak{h}_{m_k:\infty})$. It follows that
     \begin{equation}
         \|J_k^*RJ_k-J_k^*R_{m_k:\infty}J_k\|_2 = \|J_k^*[\widetilde{M}^{-1}]_{22}J_k-J_k^*R_{m_k:\infty}J_k\|_2\lesssim\epsilon^{-2}\|J_k^*R_{m_k:\infty}J_m\|_2.
     \end{equation}
     Define the following sequence of operators
     \begin{align}
         \mathfrak{V}_j &= \mathfrak{h}_{j+1,j}J_j^*R_{j:\infty}J_j\;\;\text{for all}\;\; j\in\{m_k,\ldots,k-1\},\\
         \mathfrak{V}_k &= J_k^*R_{k:\infty}J_k.
     \end{align}
     Then, we have
     \begin{equation}
         J_k^*R_{m_k:\infty}J_{m_k} = (-1)^{k-m_k}\mathfrak{V}_k\prod_{j=m_k}^{k-1}\mathfrak{V}_j
     \end{equation}
     which implies
     \begin{equation}
         \|J_k^*R_{m_k:\infty}J_{m_k}\|_2\lesssim\epsilon^{-1}\prod_{j=m_k}^{k-1}\|\mathfrak{B}_j\|_2.
     \end{equation}
    Similarly to the first case, we obtain
     \begin{equation}
         \|\mathfrak{V}_j\|\leq\left(\frac{2+\alpha}{2}\right)^{-1}.
     \end{equation}
     This gives a bound of
     \begin{equation}
         \|J_k^*RJ_k-J_k^*R_{m_k:\infty}J_k\|_2\lesssim\epsilon^{-4}\left(\frac{2+\alpha}{2}\right)^{-\left\lceil\frac{k}{2}\right\rceil}.
     \end{equation}
     By a similar argument, we obtain a bound of
     \begin{equation}
         \|J_k^*RJ_k-J_k^*R_{m_k:\infty}J_k\|_2\lesssim\epsilon^{-4}\left(\frac{2+\alpha}{2}\right)^{-\left\lceil\frac{k}{2}\right\rceil}.
     \end{equation}
     By Weyl's inequality, it follows that
     \begin{equation}
         \min|\sigma(\mathfrak{h}_{m_k:\infty})|\geq\min|\sigma(\mathfrak{h}_{m_k:\infty}^{\mathrm{intra}})|-\left\|\widehat{H}^{\tau}_{\mathrm{inter}}\right\|_2.
     \end{equation}
     The minimum unperturbed energy scales linearly with the radius due to the Dirac cone, that is,
     \begin{equation}
         \min|\sigma(\mathfrak{h}_{m_k:\infty}^{\mathrm{intra}})|\approx vW_{m_k}= k
     \end{equation}
     where $v$ is the Fermi velocity. Thus, we obtain
     \begin{equation}
         k+1\lesssim\min|\sigma(\mathfrak{h}_{m_k:\infty})|
     \end{equation}
     which gives the bound
     \begin{equation}
         \|J_k^*\delta_\epsilon(E-\mathfrak{h}_{m_k:\infty})J_k\|_2\lesssim\epsilon^{-1}e^{-\frac{v^2\mathfrak{u}^2}{8}(k+1)^2\epsilon^{-2}}.
     \end{equation}
     By a similar argument, we obtain a bound of
     \begin{equation}
         \|J_k^*\delta_\epsilon(E-\mathfrak{h}_{m_k:N_W})J_k\|_2\lesssim\epsilon^{-1}e^{-\frac{v^2\mathfrak{u}^2}{8}(k+1)^2\epsilon^{-2}}.
     \end{equation}
    For case (ii), we obtain the following bound of $\eta_2$:
     \begin{equation}
         \eta_2\lesssim\epsilon^{-4}e^{-\ln{\frac{2+\alpha}{2}}\left\lceil\frac{1}{2}\left\lfloor\frac{W-W_0}{2\mathfrak{u}}+1\right\rfloor\right\rceil}+\epsilon^{-1}e^{-\frac{v^2\mathfrak{u}^2}{8}\left\lfloor\frac{W-W_0}{2\mathfrak{u}}+2\right\rfloor^2\epsilon^{-2}}.
     \end{equation}
    Together, this gives a final error bound
     \begin{align}
         \eta_W&\lesssim\eta_{W}^-+\max(\eta_1,\eta_2)\\
         &\lesssim\epsilon^{-3}e^{-\eta^2\epsilon^{-2}}+\epsilon^{-4}e^{-\ln{\frac{2+\alpha}{2}}\left\lceil\frac{1}{2}\left\lfloor\frac{W-W_0}{2\mathfrak{u}}+1\right\rfloor\right\rceil}+\epsilon^{-1}e^{-\frac{v^2\mathfrak{u}^2}{8}\left\lfloor\frac{W-W_0}{2\mathfrak{u}}+2\right\rfloor^2\epsilon^{-2}}.
     \end{align}
\end{proof}

Before we can prove \zcref{lem:errorL}, we first require another lemma. We extend our earlier pseudometric $\mathfrak{d}_{jk}$ to a full metric by uniquely extending the direct sums. Define the extension of $G\in\Omega_j^*$, denoted by $G^e$, by
\begin{equation}
     G^e = G_j\oplus G_k\oplus G_l
\end{equation}
where $G_j$ is such that $m_j+m_k+m_l=0$ for the associated integer vectors. Although this appears to be a slight abuse of notation, the layer is clear in context. This leads to the following lemma, which is a modified version of Lemma 2.2 from \cite{chen_qmmm_2016}.
\begin{lemma}
    \label{lem:resolventDecay}
    For double-incommensurate twisted trilayer graphene, the following bound holds
    \begin{equation}
    \left|R_z(\widehat{H}_q^{(\tau,W)})\right|_{(\widehat{G},k\beta)}^{(G,j\alpha)}|\lesssim \epsilon^{-1}e^{-\gamma_\epsilon\|G^e-\widehat{G}^e\|_2}
    \end{equation}
    for all $q\in\mathbb{R}^2$ and $(G,j\alpha),(G,k\beta)\in\Omega^*$ where
    \begin{equation}
        \gamma_\epsilon = \frac{C'\epsilon}{\ln(\epsilon^{-1})}
    \end{equation}
    for some $C'>0$
    and $G^e = G_j\oplus G_k\oplus G_l$ for all $G\in\Omega_j^*$ where $G_j$ is such that $m_j+m_k+m_l=0$ for the associated integer vectors.
\end{lemma}

\begin{proof}
     Define
    \begin{equation}
        [B_q]_{(\widehat{G},k\beta)}^{(G,j\alpha)} =\delta_{(\widehat{G},k\beta)}^{(G,j\alpha)}e^{-\widetilde{\gamma}\|G^e-\widetilde{G}^e\|} 
    \end{equation}
    for some fixed $\widetilde{G}\in\Omega^*$ and all $q\in\mathbb{R}^2$. Conjugating $\hat{H}_q$ by $B$ and taking the difference gives:
    \begin{align}
        \left[B_q\widehat{H}_q^{(\tau,W)}B_q^{-1}-\widehat{H}_q^{(\tau,W)}\right]_{(\widehat{G},k\beta)}^{(G,j\alpha)} &= \left[\widehat{H}_q^{(\tau,W)}\right]_{(\widehat{G},k\beta)}^{(G,j\alpha)}(e^{\widetilde{\gamma}(\|G^e-\widetilde{G}^e\|_2-\|\widehat{G}^e-\widetilde{G}^e\|_2)}-1)\\
        &\leq \left[\widehat{H}_q^{(\tau,W)}\right]_{(\widehat{G},k\beta)}^{(G,j\alpha)}\left(e^{\widetilde{\gamma}\left\|G^e-\widehat{G}^e\right\|_2}-1\right)
    \end{align}
    
    Recall that
    \begin{equation}
        \left|[\widehat{H}_q]_{(\widehat{G},k\beta)}^{(G,j\alpha)}\right|\lesssim \delta_{k}^j\delta_{\widehat{G}}^G\frac{1-\gamma_1}{\gamma_1^2}+(1-\delta_k^j)\delta_{\widehat{G}\setminus \widehat{G}_j}^{G\setminus G_k}e^{-\gamma_2\left\|q+\widehat{G}_j+G_k+G_l\right\|_2}
    \end{equation}
    for $(G,j\alpha)\in\Omega_j^*$ and $(\widehat{G},k\beta)\in\Omega_k^*$. Observe that 
    \begin{align}
        \|q+\widehat{G}_j+G_j-G_j+G_k+G_l\|_2=\|q+\widehat{G}_j+G_k+G_l\|_2<\tau
    \end{align}
    where $m_j = -m_k-m_l$ such that
    \begin{equation}
        \|q+(B_k-B_j)m_k+(B_l-B_j)m_l\|_2<W.
    \end{equation}
    So, by the lower bound of the reverse triangle inequality, we obtain
    \begin{equation}
        \|\widehat{G}_j-G_j\|_2<W+\tau.
    \end{equation}
    By a similar argument, one obtains the bounds necessary for $\|G^e-\widehat{G}^e\|_2<\sqrt{2}(W+\tau)$. It immediately follows that,
    \begin{equation}
        \left|[\widehat{H}_q^{(\tau,W)}]_{(\widehat{G},k\beta)}^{(G,j\alpha)}\right|\lesssim e^{-\frac{\gamma_2}{2}\left\|G^e-\widehat{G}^e\right\|_2}.
    \end{equation}
    Then, taking the infinity norm, we have
    \begin{gather}
        \|B\widehat{H}_q^{(\tau,W)}B^{-1}-\widehat{H}_q^{(\tau,W)}\|_{\infty} \leq \sup_{(G,j\alpha)\in\Omega^*}\sum_{(\widehat{G},k\beta)\in\Omega^*}\left|\left[\widehat{H}_q^{(\tau,W)}\right]_{(\widehat{G},k\beta)}^{(G,j\alpha)}\right|\left|e^{\widetilde{\gamma}\left\|G^e-\widehat{G}^e\right\|_2}-1\right|\\
        \lesssim \max_{j\alpha\in\mathcal{A}}\sup_{\substack{|j-k|=1\\G_k\in\mathcal{R}_k^*}}\sum_{k\beta\in\mathcal{A}_k}\sum_{\substack{\widehat{G}_j\in\mathcal{R}_j^*\\ G_l\in\mathcal{R}_l^*}}e^{-\frac{\gamma_2}{2}\|G^e-\widehat{G}^e\|_2}\left|e^{\widetilde{\gamma}\left\|G^e-\widehat{G}^e\right\|_2}-1\right|\\
        \lesssim \max_{j\alpha\in\mathcal{A}}\sup_{|j-k|=1}\sum_{\substack{G_k\in\mathcal{R}_k^*\\ k\beta\in\mathcal{A}_k}}\sum_{\substack{\\G_l\in\mathcal{R}_l^*\\ \widehat{G}_j\in\mathcal{R}_j^*}}\begin{cases}
            (e^{-\widetilde{\gamma}R}-1)e^{-\frac{\gamma_2}{2}\|G^e-\widehat{G}^e\|_2}, & \|G^e-\widehat{G}^e\|_2\leq R,\\
            e^{-\left(\frac{\gamma_2}{2}-\widetilde{\gamma}\right)\|G^e-\widehat{G}^e\|_2}, & \|G^e-\widehat{G}^e\|_2>R
        \end{cases}
        \\
        \lesssim \left(e^{\widetilde{\gamma}R}-1\right)+e^{-\frac{1}{2}(\frac{\gamma_2}{2}-\widetilde{\gamma})R}
    \end{gather}
    where the sum over $G_l\in\mathcal{R}_l^*$ is finite by the $(\tau,W)$-truncation for each fixed $G_k\in\mathcal{R}_k^*$, and the maximum over $\mathcal{A}$ (sum over $\mathcal{A}_k$) is finite since each $\mathcal{A}_k$ has finitely many elements. Additionally, we have implicitly chosen $\tau>R$. We force the norm to be less than $\epsilon$ by splitting $\epsilon$ across the summands. Hence,
    \begin{equation}
        e^{-\frac{1}{2}(\frac{\gamma_2}{2}-\widetilde{\gamma})R}=\frac{\epsilon}{2}
    \end{equation}
    for which
    \begin{equation}
        R=\frac{2}{\frac{\gamma_2}{2}-\widetilde{\gamma}}\ln\left(\frac{2}{\epsilon}\right).
    \end{equation}
    From the remaining summand, we have that
    \begin{equation}
        \widetilde{\gamma}=\frac{1}{R}\ln\left(1+\frac{\epsilon}{2}\right)
    \end{equation}
    where by substituting $R$, we obtain
    \begin{equation}
        \widetilde{\gamma} = \frac{\gamma_2}{2}\left[\frac{\ln\left(1+\epsilon 2^{-1}\right)}{2\ln\left(2\epsilon^{-1}\right)+\ln(1+\epsilon 2^{-1})}\right]
    \end{equation}
    which, for $\epsilon$ small, yields
    \begin{equation}
        \widetilde{\gamma}\approx \frac{C'\epsilon}{\ln(\epsilon^{-1})}.
    \end{equation}
    Hence, for any $\epsilon>0$, we may choose $R<\tau$ sufficiently large and $\widetilde{\gamma} = \gamma_\epsilon := C'\frac{\epsilon}{\ln(\epsilon^{-1})}$ for $C'>0$ chosen as above to obtain $\|B_q\widehat{H}_q^{(\tau,W)}B_q^{-1}\|_\infty<\epsilon$. Exchanging the roles of $(G,j\alpha)$ and $(\widehat{G},k\beta)$, we obtain $\|B_q\widehat{H}_q^{(\tau,W)}B_q^{-1}\|_1<\epsilon$. Thus, by the Riesz-Thorin interpolation theorem, we have $\|B_q\widehat{H}_q^{(\tau,W)}B_q^{-1}\|_2<\epsilon$. By the properties of $\widehat{H}_q^{(\tau,W)}$, there exists a contour around $\sigma(H_q)$ such that $d(z,\sigma(\widehat{H}_q^{(\tau,W)}))\in(\epsilon, 2\epsilon)$ and $\|R_z(\widehat{H}_q^{(\tau,W)})\|_2\leq\epsilon^{-1}$. Hence, we may choose $R$ and $\widetilde{\gamma}$ such that $R_z(B_q\widehat{H}_q^{(\tau,W)}B_q^{-1})$ exists and
    \begin{equation}
        \|B_qR_z(\widehat{H}_q^{(\tau,W)})B_q^{-1}\|_2\leq 2\epsilon^{-1}.
    \end{equation}
    It follows that,
    \begin{gather}
        \left|\left[R_z(\widehat{H}_q^{(\tau,W)})\right]_{(\widehat{G},k\beta)}^{(G,j\alpha)}e^{\widetilde{\gamma}(\|G^e-\widetilde{G}^e\|_2-\|\widehat{G}^e-\widetilde{G}^e\|_2)}\right| \\
        =|[B_qR_z(\widehat{H}_q^{(\tau,W)})B_q^{-1}]_{(\widehat{G},k\beta)}^{(G,j\alpha)}|\leq\|B_qR_z(\widehat{H}_q^{(\tau,W)})B_q^{-1}\|_2\lesssim\epsilon^{-1}.
    \end{gather}
    Therefore, taking $\widetilde{G} = \widehat{G}$, we obtain
    \begin{equation}
        \left|R_z(\widehat{H}_q^{(\tau,W)})\right|_{(\widehat{G},k\beta)}^{(G,j\alpha)}|\lesssim \epsilon^{-1}e^{-\gamma_\epsilon\|G^e-\widehat{G}^e\|_2}.
    \end{equation}
\end{proof}

At long last, we prove \zcref{lem:errorL}, which concludes the proof of \zcref{thm:cumulative}.
\begin{proof}
    Let $C$ be a contour around $\sigma(\widehat{H})$ such that 
    \begin{equation}
        d(z,\sigma(\widehat{H}))\in(\epsilon,2\epsilon)
    \end{equation}
    for all $z\in C$. Then
    \begin{equation}
        \eta_L\lesssim\frac{1}{N_W^2}\sum_{\substack{j\alpha\in\mathcal{A}\\q\in\mathfrak{S}_j(W, N)}}\oint_C|\delta_\epsilon(E-z)|\left|\left[R_z(\widehat{H}_q^{(\tau,W)})-R_z(\widehat{H}_q^{(\tau,W, L)})\right]_{(0,j\alpha)}^{(0,j\alpha)}\right||dz|
    \end{equation}
Hence, it suffices to bound
\begin{equation}
    \zeta_L = \left|\left[R_z(\widehat{H}_q^{(\tau,W)})-R_z(\widehat{H}_q^{(\tau,W, L)})\right]_{(0,j\alpha)}^{(0,j\alpha)}\right|.
\end{equation}
Define
\begin{equation}
   \mathcal{C}(q,W,L) = \mathcal{W}^*(q, W)\setminus\mathcal{L}(L).
\end{equation}
By the second resolvent identity, it follows that
{\small
\begin{equation}
    \zeta_L\leq\sum_{\substack{(G,k\beta)\in\mathcal{C}(q,W,L)\\ (\widehat{G},l\gamma)\in\mathcal{W}^*(q,W)\cap\mathcal{L}(L)}}\left|\left[R_z(\widehat{H}_q^{(\tau,W)})\right]_{(G,k\beta)}^{(0,j\alpha)}\right|\left|\left[\widehat{H}_q^{(\tau,W)}\right]_{(\widehat{G},l\gamma)}^{(G,k\beta)}\right|\left|\left[R_z(\widehat{H}_q^{(\tau,W,L)})\right]_{(0,j\alpha)}^{(\widehat{G},l\gamma)}\right|
\end{equation}
}
Thus, by \zcref{lem:resolventDecay}, we obtain the resolvent bounds
\begin{align}
    \left|\left[R_z(\widehat{H}_q^{(\tau,W)})\right]_{(G,k\beta)}^{(0,j\alpha)}\right|&\lesssim \epsilon^{-1}e^{-\gamma_\epsilon\|G^e\|_2}\\
    \left|\left[R_z(\widehat{H}_q^{(\tau,W)})\right]_{(0,j\alpha)}^{(\widehat{G},l\gamma)}\right|&\lesssim \epsilon^{-1}e^{-\gamma_\epsilon\|\widehat{G}^e\|_2}.
\end{align}
 which gives
\begin{equation}
    \zeta_L\lesssim\epsilon^{-2}\sum_{\substack{(G,k\beta)\in\mathcal{C}(q,W,L)\\ (\widehat{G},l\gamma)\in\mathcal{W}^*(q,W)\cap\mathcal{L}(L)}}e^{-\gamma_\epsilon(\|G^e\|_2+\|\widehat{G}^e\|_2)}\left|\left[\widehat{H}_q^{(\tau,W)}\right]_{(\widehat{G},l\gamma)}^{(G,k\beta)}\right|.
\end{equation}
By the definition of $\mathcal{C}(q,W,L)$ and $\mathcal{W}^*(q,W)\cap\mathcal{L}(L)$, it follows that
\begin{equation}
    \left[\widehat{H}_q^{(\tau,W)}\right]_{(\widehat{G},l\gamma)}^{(G,k\beta)}\neq 0
\end{equation}
if and only if both $k\neq l$ and $G_m = \widehat{G}_m$. Thus, the bound reduces further to
\begin{equation}
    \zeta_L\lesssim\epsilon^{-2}\sum_{\substack{(\widehat{G}_l,G_m,k\beta)\in\mathcal{C}(q,W,L)\\ (G_k,G_m,l\gamma)\in\mathcal{W}(q,W)\cap\mathcal{L}(L)\\k\neq l}}e^{-\gamma_\epsilon(\|G^e\|_2+\|\widehat{G}^e\|_2)}\left|\left[\widehat{H}_q^{(\tau,W)}\right]_{(\widehat{G}_l,G_m,l\gamma)}^{(G_k,G_m,k\beta)}\right|.
\end{equation}
Since $k\neq l$, entries from $\widehat{H}_q^{(\tau,W)}$ come from interlayer terms that decay as
\begin{equation}
    \left|\left[\widehat{H}_q^{(\tau,W)}\right]_{(\widehat{G}, l\gamma)}^{(G,k\beta)}\right|\lesssim e^{-\gamma_2\|G^e-\widehat{G}^e\|_2}.
\end{equation}
This yields the bound
\begin{align}
    \zeta_L&\lesssim\epsilon^{-2}\sum_{\substack{(G_l,G_m,k\beta)\in\mathcal{C}(q,W,L)\\ (G_k,G_m,l\gamma)\in\mathcal{W}(q,W)\cap\mathcal{L}(L)\\k\neq l}}e^{-\gamma_\epsilon(\|G^e\|_2+\|\widehat{G}^e\|_2)-\gamma_2\|G^e-\widehat{G}^e\|_2}\\
    &\lesssim\epsilon^{-2}\sum_{(G_l,G_m,k\beta)\in\mathcal{C}(q,W,L)}e^{-\gamma_\epsilon\|\widehat{G}^e\|_2}\sum_{(G_k,G_m,l\gamma)\in\mathcal{W}(q,W)\cap\mathcal{L}(L)}e^{-\gamma_2\|G^e-\widehat{G}^e\|_2}\\
    &\lesssim\epsilon^{-2}\sum_{(G_l,G_m,k\beta)\in\mathcal{C}(q,W,L)}e^{-\gamma_\epsilon\|\widehat{G}^e\|_2}\\
    &\lesssim\epsilon^{-2}\int_{L}^\infty re^{-\gamma_\epsilon r}dr\\
    &\lesssim\epsilon^{-2}e^{-\gamma_\epsilon L}\frac{1+\gamma_\epsilon L}{\gamma_\epsilon^2}.
\end{align}
Observe that $G_m$ comes from a finite set, so that the reduction to the $2$D integral produces a constant dependent on $W$, $L$, and $\max_{j=1}^3|\theta_j|$ for all nonzero $\theta_j$ in radians. This constant is essentially $O(W^2L^2)$.
We pick up an $\epsilon^{-1}$ from the Gaussian, bringing the final bound to
\begin{equation}
    \eta_L\lesssim\epsilon^{-3}\gamma_\epsilon^{-1}Le^{-\gamma_\epsilon L}.
\end{equation}
\end{proof}
This completes the proof of \zcref{thm:cumulative}.

\appendix
\section{Hopping function details}
\label{app:hop}
 Let $\mathcal{V}_\beta:\mathbb{N}_0^2\to\mathbb{R}$ be defined by
\begin{gather}
    \mathcal{V}_{\beta}m = \delta_\beta^A(m_1^2+m_1m_2+m_2^2)+\delta_\beta^B(3(m_1^2+m_1m_2+m_2^2)+3(m_1+m_2)+1)
\end{gather}
for all $m\in\mathbb{N}_0^2$ and $v_n^\beta$, let $n$ be the $n$th unique value in $\mathcal{V}_{\beta}\mathbb{N}_0^2$ such that $v_n^\beta < v_{n+1}^\beta$. Define the distance
\begin{gather}
    d(n,\beta) = \left(a\delta_\beta^A+\frac{a}{\sqrt{3}}\delta_\beta^B\right)\sqrt{v_n^\beta}.
\end{gather}
We define the $n$th shell of the $\beta$ sublattice of monolayer graphene $S_{n}^{\beta}$ by
\begin{gather}
    S_n^\beta = \{R\in\mathcal{R}+\tau_\beta: \|R\|_2=d(n,\beta)\}
\end{gather}
where $S_0^A$ is considered the center site. Then the intralayer hopping entries are given by
\begin{gather}
    [\widehat{H}_q]_{(G_k,G_l,j\beta)}^{(G_k,G_l,j\alpha)} = \sum_{\{n:d(n,\beta)<\tau\}}t_n^\beta\sum_{R\in S_n^\beta}e^{i(q+G_k+G_l)\cdot \mathfrak{R}_{\theta_j}(R)}
\end{gather}
where the hopping strengths $t_{n}^\beta$ are given in Table \zcref{tab:intraParam}.

\begin{table}[H]
    \centering
    \begin{tabular}{|c|c|c|}
        \hline
        $\boldsymbol{n}$ &  $\boldsymbol{\beta=A}$ & $\boldsymbol{\beta=B}$\\ \hline
        0 & $0.3208$ & ---\\ \hline
        1 & $0.22378$ & $-2.92181$\\ \hline
        2 & $0.04813$ & $-0.27897$\\ \hline
        3 & $-0.02402$ & $0.02669$\\ \hline
        4 & $0.00263$ & $-0.00885$\\ \hline
        5 & $0.00111$ & $-0.01772$\\ \hline
        6 & $0.00018$ & $0.00675$\\ \hline
        7 & $-0.00008$ & $-0.00262$\\ \hline
        8 & --- & $0.00019$\\ \hline
        9 & --- & $-0.00068$ \\ \hline
        10 & --- & $-0.00237$\\ \hline
    \end{tabular}
    \caption{Intralayer hopping strengths $t_n^\beta$ across neighbor shells $n = 0$ to $n = 10$ for Graphene.}
    \label{tab:intraParam}
    \vspace{6pt}
    \raggedright
    \small\textit{Source:} Data adapted from \cite{jung_tight-binding_2013}.
\end{table}

For interlayer hopping functions, we adapt the ab initio interlayer hopping functions from \cite{fang_electronic_2016}. Define
\begin{equation}
    \vec{r} = (R_k+\tau_{k\beta})-(R_j+\tau_{j\alpha})
\end{equation}
for $(R_j,j\alpha)\in\Omega_j$ and $(R_k,k\beta)\in\Omega_k$ with $j\neq k$. In addition, let
\begin{equation}
    r = \frac{\|\vec{r}\|_2}{a},\;\;\hat{r} = \frac{\vec{r}}{\|\vec{r}\|_2}\;\;\text{and}\;\;\hat{\tau}_{j\alpha} = \frac{\tau_{j\alpha}}{\|\tau_{j\alpha}\|_2}.
\end{equation}
Then
\begin{align}
    h_{k\beta}^{j\alpha}(\vec{r}) &= \lambda_0e^{-\xi_0r^2}\cos(\kappa_0r)+\lambda_3r^2e^{-\xi_3\left(r-x_3\right)^2}[\theta_3]_{j\alpha}^{k\beta}(\hat{r})\\
    &+\lambda_6e^{-\xi_6\left(r-x_6\right)^2}\sin(\kappa_6r)[\theta_6]_{j\alpha}^{k\beta}(\hat{r})
\end{align}
where
\begin{equation}
    [\theta_n]_{j\alpha}^{k\beta}(\hat{r}) = \cos(n\cos^{-1}(\hat{r}\cdot(-1)^{\alpha+2}\hat{\tau}_{j\alpha}))+\cos(n\cos^{-1}(\hat{r}\cdot(-1)^{\beta+1}\hat{\tau}_{k\beta}))
\end{equation}
and the parameters $\lambda_j$, $\xi_j$, $\kappa_j$, and $x_j$ for $j\in\{0,3,6\}$ are given in Table \zcref{tab:interParam}.

\begin{table}[H]
    \centering
    \begin{tabular}{|c|c|c|c|c|}
        \hline
       $\boldsymbol{j}$ & $\boldsymbol{\lambda_j}$ & $\boldsymbol{\xi_j}$ & $x_j$ & $\boldsymbol{\kappa_j}$ \\ \hline
        $0$ & $0.3155$ & $1.7543$ & --- & $2.0010$\\ \hline
        $3$ & $-0.0688$ & $3.4692$ & $0.5212$ & ---\\ \hline
        $6$ & $-0.0083$ & $2.8764$ & $1.5206$ & $1.5731$\\ \hline
    \end{tabular}
    \caption{Interlayer tight-binding parameters for graphene.}
    \label{tab:interParam}
    \vspace{6pt}
    \raggedright
    \small\textit{Source:} Data adapted from \cite{fang_electronic_2016}.
\end{table}

We require $\hat{h}_{k\beta}^{j\alpha}$ in order to construct the truncated reciprocal Hamiltonian $\widehat{H}_q^{(\tau,W,L)}$. Ideally, one would obtain the analytic Fourier transform. However, there does not seem to be a closed form expression for $\hat{h}_{k\beta}^{j\alpha}$. Consequently, we must resort to numerical integration and interpolation. One notices that $h_{k\beta}^{j\alpha}$ consists of the products of radial functions and trigonometric functions. Using the Jacobi-Anger expansion
\begin{equation}
    e^{-i2\pi\rho r\cos(\theta-\phi)} = \sum_{n = -\infty}^{\infty}(-i)^nJ_n(2\pi\rho r)e^{in(\theta-\phi)}
\end{equation}
where $J_n$ is the $n$th Bessel function of the first kind, one obtains 
\begin{equation}
    F(\rho,\phi) = \sum_{m=-\infty}^{\infty}(-i)^me^{im\phi}\mathcal{H}_{m}\{V_m(r)\}(\rho)
\end{equation}
for the $2$D Fourier transform
\begin{equation}
    F(\rho,\phi) = \int_{0}^\infty\int_{0}^{2\pi}f(r,\theta)e^{-i2\pi\rho r\cos(\theta-\phi)}rd\theta dr
\end{equation}
where
\begin{equation}
    \mathcal{H}_{m}\{V_m(r)\}(\rho) = 2\pi\int_{0}^\infty V_m(r)J_m(2\pi\rho r)rdr
\end{equation}
is the $m$th order Hankel transform. Applying this to $h_{k\beta}^{j\alpha}$, one obtains
\begin{equation}
    \hat{h}_{k\beta}^{j\alpha}(\rho) =  \mathcal{H}_{0}\{V_0(r)\}(\rho)+i[\theta_3]_{k\beta}^{j\alpha}(\phi)\mathcal{H}_3\{V_3(r)\}(\rho)-[\theta_6]_{k\beta}^{j\alpha}(\phi)\mathcal{H}_6\{V_6(r)\}(\rho)
\end{equation}
where $V_j(r)$ refers to the $j = \{0, 3, 6\}$ radial parts of $h_{k\beta}^{j\alpha}$. This leads to the following algorithm for computing $\hat{h}_{k\beta}^{j\alpha}$.

\noindent\makebox[\textwidth][c]{
\begin{minipage}{0.75\textwidth}
    \begin{algorithm}[H]
    \caption{Computing $\hat{h}_{k\beta}^{j\alpha}$}
    \label{alg:matlab_hopping}
        \begin{algorithmic}[1]
            \State Define a high resolution $1$D momentum grid $\rho$
            \State For each $n$ compute $\mathcal{H}_{n}\{V_n(r)\}(\rho)$
            \State Construct $1$D interpolants
            \State Define a high resolution $2$D momentum grid $q$
            \State Evaluate the $1$D interpolants across $\|q\|_2$
            \State Ensure $C_3$ symmetry by extending $q$ across compatible rotations
            \State Compute the angular phase factors $[\theta_n]_{k\beta}^{j\alpha}$ based on $q$
            \State Accumulate corresponding phase shifted radial components
            \State Convert the accumulation into a $2$D interpolant
        \end{algorithmic}
    \end{algorithm}
\end{minipage}
}
\vspace{1em}

In computing the angular phase factors, we take advantage of the relation
\begin{equation}
    T_n(x) = \cos(n\cos^{-1}(x))
\end{equation}
for $\|x\|\leq 1$.
To highlight changes in the accuracy of the local density of states, we compare against the established continuum model \cite{Zoe2020}. 

\section{Continuum model description}
\label{app:bm}

The Taylor expansion of $[\widehat{H}_q]_{jj}$ around $K_j$ centered at $K_2$ yields the rotated Dirac equation
\begin{equation}
    [(\widehat{H}_q^{\mathrm{BM}})_{jj}]_{G'}^{G} = \delta_{G'}^Gv_F\widetilde{q}\cdot\left[\sigma_x^{\theta_j},-\sigma_y^{\theta_j}\right]
\end{equation}
where 
\begin{align}
    \sigma_x^{\theta_j} &= \sigma_x\cos\theta_j-\sigma_y\sin\theta_j,\;\;\sigma_y^{\theta_j} = \sigma_x\sin\theta_j+\sigma_y\cos\theta_j,\\
    \widetilde{q} &= q+(I-B_jB_k^{-1})G_k+(I-B_jB_l^{-1})G_l+K_j-K_2
\end{align} 
For $[\widehat{H}_q]_{jk}$, they make the approximation
\begin{equation}
    [(\widehat{H}_q^{\mathrm{BM}})_{jk}]_{G'}^{G} = \sum_{n=1}^3T_{jk}^{(n)}\delta_{q'-q'',q_{jk}^{(n)}}
\end{equation}
where
\begin{align}
    T_{jk}^{(n)} &= \begin{bmatrix}
        \omega_0 & \omega_1e^{-\frac{2\pi i(n-1)}{3}}\\
        \omega_1e^{-\frac{2\pi i(1-n)}{3}} & \omega_0
    \end{bmatrix},\\
    q' &= (I-B_jB_k^{-1})G_k+(I-B_jB_l^{-1})G_l,\\
    q''&= (I-B_kB_j^{-1})G_j'+(I-B_kB_l^{-1})G_l',\\
    q_{jk}^{(n)} &= \mathfrak{R}_{(1-n)\frac{2\pi}{3}}(K_j-K_k)
\end{align}
for $G\in\Omega_j^*$ and $G'\in\Omega_k^*$, and suitable interlayer hopping strengths $\omega_0\approx 0.07\unit{eV}$ and $\omega_1\approx 0.11\unit{eV}$. The difference in $\omega_0$ and $\omega_1$ results in a small degree of out of plane relaxation. This construction is a simplification of the model proposed by \cite{amorim_electronic_2018} and is similar to the Bistritzer-MacDonald model \cite{Bistritzer2011} for twisted bilayer graphene. However, it suppresses the particle-hole asymmetry that is present in the tight-binding model. 

\bibliographystyle{elsarticle-harv}
\bibliography{references}

\end{document}